\documentclass[11pt]{article}

\usepackage{amsmath,amssymb,amsthm,amscd,amsfonts}
\usepackage{epsfig,pinlabel,paralist}
\usepackage[vmargin=1in, hmargin=1.25in]{geometry}
\usepackage[font=small,format=plain,labelfont=bf,up,textfont=it,up]{caption}
\usepackage{calc}
\usepackage{etoolbox}
\usepackage{microtype}
\usepackage[all,cmtip]{xy}

\usepackage[bookmarks, bookmarksdepth=2, colorlinks=true, linkcolor=blue, citecolor=blue, urlcolor=blue]{hyperref}

\apptocmd{\thebibliography}{\raggedright}{}{}

\numberwithin{equation}{section}

\theoremstyle{plain}
\newtheorem{theorem}{Theorem}[section]
\newtheorem{maintheorem}{Theorem}

\newtheorem{proposition}[theorem]{Proposition}
\newtheorem{lemma}[theorem]{Lemma}

\newtheorem{corollary}[theorem]{Corollary}

\newtheorem*{claim}{Claim}
\newtheorem{claims}{Claim}
\newtheorem{casea}{Case}

\theoremstyle{definition}

\newtheorem{defn}[theorem]{Definition}

\theoremstyle{remark}
\newtheorem{rmk}[theorem]{Remark}
\newenvironment{remark}[1][]{\begin{rmk}[#1] \pushQED{}}{\popQED \end{rmk}}

\newtheorem{eg}[theorem]{Example}
\newenvironment{example}[1][]{\begin{eg}[#1] \pushQED{\qed}}{\popQED \end{eg}}

\DeclareMathOperator{\Hom}{Hom}

\DeclareMathOperator{\Mod}{Mod}

\newcommand\Torelli{\ensuremath{{\mathcal I}}}

\DeclareMathOperator{\Sp}{Sp}

\DeclareMathOperator{\GL}{GL}

\newcommand\Z{\ensuremath{\mathbb{Z}}}

\newcommand\Field{\ensuremath{\mathbb{F}}}

\DeclareMathOperator{\HH}{H}

\newcommand\RH{\ensuremath{\widetilde{\HH}}}
\newcommand\RZ{\ensuremath{\widetilde{\Z}}}

\DeclareMathOperator{\Interior}{Int}

\newcommand\Set[2]{\ensuremath{\left\{\text{#1 $|$ #2}\right\}}}

\newcommand\Surf{\ensuremath{\tt Surf}}
\newcommand\PSurf{\ensuremath{\tt PSurf}}
\newcommand\cL{\ensuremath{\mathcal{L}}}
\newcommand\cB{\ensuremath{\mathcal{B}}}
\newcommand\cP{\ensuremath{\mathcal{P}}}

\newcommand\TS{\ensuremath{\mathcal{TS}}}
\newcommand\Sur{\ensuremath{\mathcal{S}}}
\newcommand\TL{\ensuremath{\mathcal{TL}}}
\newcommand\DTL{\ensuremath{\mathcal{DTL}}}
\newcommand\MTL{\ensuremath{\mathcal{MTL}}}
\newcommand\ODTL{\ensuremath{\mathcal{ODTL}}}

\DeclareMathOperator{\rank}{rk}
\DeclareMathOperator{\Fib}{Fib}
\DeclareMathOperator{\DP}{DP}

\newcommand\tmu{\ensuremath{\widetilde{\mu}}}
\newcommand\tsigma{\ensuremath{\widetilde{\sigma}}}

\newcommand\hiota{\ensuremath{\widehat{\iota}}}
\newcommand\hcL{\ensuremath{\widehat{\cL}}}
\newcommand\hisect{\ensuremath{\widehat{\isect}}}
\newcommand\hSigma{\ensuremath{\widehat{\Sigma}}}
\newcommand\hx{\ensuremath{\widehat{x}}}
\newcommand\hcP{\ensuremath{\widehat{\cP}}}

\newcommand\hmu{\ensuremath{\widehat{\mu}}}

\newcommand\isect{\ensuremath{\mathfrak{i}}}

\newcommand{\p}[1]{{\bf #1.}}

\title{\vspace{-40pt}Partial Torelli groups and homological stability\vspace{-15pt}}
\author{Andrew Putman\thanks{Supported in part by NSF grant DMS-1811322}}
\date{}

\begin{document}

\vspace{-10pt}
\maketitle

\vspace{-18pt}
\begin{abstract}
\noindent
We prove a homological stability theorem for the subgroup of the mapping class group
acting as the identity on some fixed portion of the first homology group of the surface.
We also prove a similar theorem for the subgroup of the mapping class group preserving a
fixed map from the fundamental group to a finite group, which can be viewed as a mapping
class group version of a theorem of Ellenberg--Venkatesh--Westerland about braid groups.
These results require studying various simplicial complexes formed by subsurfaces of the surface,
generalizing work of Hatcher--Vogtmann.
\end{abstract}

\section{Introduction}
\label{section:introduction}

Let $\Sigma_g^b$ be an oriented genus $g$ surface with $b$ boundary components.
The {\em mapping class group} $\Mod(\Sigma_g^b)$ is the group
of isotopy classes of orientation-preserving homeomorphisms of $\Sigma_g^b$ that
fix $\partial \Sigma_g^b$ pointwise.  Harer \cite{HarerStable} proved 
that $\Mod(\Sigma_g^b)$ satisfies homological stability.  More precisely, an 
orientation-preserving embedding 
$\Sigma_g^b \hookrightarrow \Sigma_{g'}^{b'}$ induces a map
$\Mod(\Sigma_g^b) \rightarrow \Mod(\Sigma_{g'}^{b'})$ that
extends mapping classes by the identity, and Harer's theorem says that the induced
map
$\HH_k(\Mod(\Sigma_g^b)) \rightarrow \HH_k(\Mod(\Sigma_{g'}^{b'}))$
is an isomorphism if $g \gg k$.

\p{Torelli}
The group $\Mod(\Sigma_g^b)$ acts on $\HH_1(\Sigma_g^b)$.  
For $b \leq 1$, the algebraic intersection
pairing on $\HH_1(\Sigma_g^b)$ is a $\Mod(\Sigma_g^b)$-invariant symplectic form.  
We thus get a map
$\Mod(\Sigma_g^b) \rightarrow \Sp_{2g}(\Z)$ whose kernel $\Torelli(\Sigma_g^b)$ is
the {\em Torelli group}.  The group $\Torelli(\Sigma_g^b)$ is not homologically stable;
indeed, Johnson \cite{JohnsonAbel} showed that $\HH_1(\Torelli(\Sigma_g^b))$ does
not stabilize.  Church--Farb's work on representation stability \cite{ChurchFarbRepStability}
connects this to the $\Sp_{2g}(\Z)$-action on $\HH_k(\Torelli(\Sigma_g^b))$
induced by the conjugation action of $\Mod(\Sigma_g^b)$.
Much recent work on $\HH_k(\Torelli(\Sigma_g^b))$ focuses on this action; see \cite{BoldsenDollerup, KassabovPutman, MillerPatztWilson}.

\p{Partial Torelli}
We show that homological stability can be restored by enlarging the Torelli group
to the group acting trivially on some fixed portion of homology.  As an illustration
of our results, we begin by describing a very special case of them.  Fix a symplectic
basis $\{a_1,b_1,\ldots,a_g,b_g\}$ for $\HH_1(\Sigma_g^1)$ in the usual way:\\
\centerline{\psfig{file=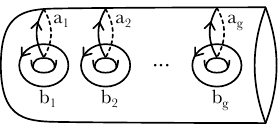,scale=100}}
For $0 \leq h \leq g$, define $\Torelli(\Sigma_g^1,h)$ to be the subgroup of $\Mod(\Sigma_g^1)$ fixing
all elements of $\{a_1,b_1,\ldots,a_h,b_h\}$.  These groups interpolate between
$\Mod(\Sigma_g^1)$ and $\Torelli(\Sigma_g^1)$ in the sense that
\[\Torelli(\Sigma_g^1) = \Torelli(\Sigma_g^1,g) \subset \Torelli(\Sigma_g^1,g-1) \subset \Torelli(\Sigma_g^1,g-2) \subset \cdots \subset \Torelli(\Sigma_g^1,0) = \Mod(\Sigma_g^1).\]
They were introduced by Bestvina--Bux--Margalit \cite{BestvinaBuxMargalitTorelli}; see especially
\cite[Conjecture 1.2]{BestvinaBuxMargalitTorelli}.
For a fixed $h \geq 1$, we have an increasing chain of groups
\begin{equation}
\label{eqn:stablechain}
\Torelli(\Sigma_h^1,h) \subset \Torelli(\Sigma_{h+1}^1,h) \subset \Torelli(\Sigma_{h+2}^1,h) \subset \cdots,
\end{equation}
where $\Torelli(\Sigma_g^1,h)$ is embedded in $\Torelli(\Sigma_{g+1}^1,h)$ as follows:\\
\centerline{\psfig{file=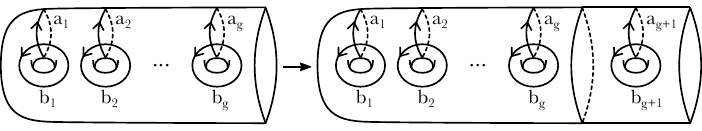,scale=100}}
Our main theorem shows that \eqref{eqn:stablechain} satisfies homological stability: for $h,k \geq 1$, we
have
\[\HH_k(\Torelli(\Sigma_g^1,h)) \cong \HH_k(\Torelli(\Sigma_{g+1}^1,h))\]
for $g \geq (2h+2)k + (4h+2)$.

\p{Homology markings}
To state our more general result, we need the notion of a homology marking.  Let $A$ be a finitely generated
abelian group.  An {\em $A$-homology marking} on $\Sigma_g^1$ is
a homomorphism $\mu\colon \HH_1(\Sigma_g^1) \rightarrow A$.  Associated to this is a {\em partial
Torelli group}
\[\Torelli(\Sigma_g^1,\mu) = \Set{$f \in \Mod(\Sigma_g^1)$}{$\mu(f(x))=\mu(x)$ for all $x \in \HH_1(\Sigma_g^1)$}.\]

\begin{example}
If $A = \HH_1(\Sigma_g^1)$ and $\mu = \text{id}$, then 
$\Torelli(\Sigma_g^1,\mu) = \Torelli(\Sigma_g^1)$.
\end{example}

\begin{example}
If $A = \HH_1(\Sigma_g^1;\Z/\ell)$ and $\mu\colon \HH_1(\Sigma_g^1) \rightarrow A$
is the projection, then $\Torelli(\Sigma_g^1,\mu)$ is the {\em level-$\ell$
subgroup} of $\Mod(\Sigma_g^1)$, i.e., the kernel of the action of $\Mod(\Sigma_g^1)$ on
$\HH_1(\Sigma_g^1;\Z/\ell)$.
\end{example}

\begin{example}
Let $A$ be a {\em symplectic subspace} of $\HH_1(\Sigma_g^1)$, i.e., a subspace with
$\HH_1(\Sigma_g^1) = A \oplus A^{\perp}$, where $\perp$ is defined via the intersection
form.  Such an $A$ is of the form $A \cong \Z^{2h}$ for some $h \geq 0$ called the {\em genus}
of $A$.  If $\mu\colon \HH_1(\Sigma_g^1) \rightarrow A$ is the projection, 
then 
\[\Torelli(\Sigma_g^1,\mu) = \Set{$f \in \Mod(\Sigma_g^1)$}{$f(x)=x$ for all $x \in A$}.\]
If $A$ has genus $h$, then $\Torelli(\Sigma_g^1,\mu) \cong \Torelli(\Sigma_g^1,h)$.
\end{example}

\p{Stability}
Our first main theorem is a homological stability theorem for the groups $\Torelli(\Sigma_g^1,\mu)$.
Define the {\em stabilization to $\Sigma_{g+1}^1$} 
of an $A$-homology marking $\mu$ on $\Sigma_g^1$
to be the following $A$-homology marking $\mu'$ on $\Sigma_{g+1}^1$.
Embed $\Sigma_g^1$ in $\Sigma_{g+1}^1$ just like we did above:\\
\centerline{\psfig{file=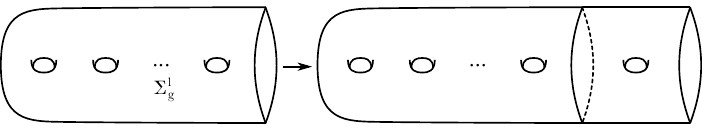,scale=100}}
This identifies $\HH_1(\Sigma_g^1)$ with a symplectic subspace of $\HH_1(\Sigma_{g+1}^1)$, so
$\HH_1(\Sigma_{g+1}^1) = \HH_1(\Sigma_g^1) \oplus \HH_1(\Sigma_g^1)^{\perp}$.  Let
$\mu'\colon \HH_1(\Sigma_{g+1}^1) \rightarrow A$ be the composition
\[\HH_1(\Sigma_{g+1}^1) = \HH_1(\Sigma_g^1) \oplus \HH_1(\Sigma_g^1)^{\perp} \longrightarrow \HH_1(\Sigma_g^1) \stackrel{\mu}{\longrightarrow} A,\]
where the first arrow is the orthogonal projection.  The map
$\Mod(\Sigma_g^1) \rightarrow \Mod(\Sigma_{g+1}^1)$ induced by the above embedding
restricts to a map $\Torelli(\Sigma_g^1,\mu) \rightarrow \Torelli(\Sigma_{g+1}^1,\mu')$
called the {\em stabilization map}.  Our main theorem is as follows.  For a
finitely generated
abelian group $A$, let $\rank(A)$ denote the minimal size of a generating set\footnote{Equivalently,
$\rank(A)$ is the maximal $n \geq 0$ such that $A$ is the direct sum of $n$ cyclic subgroups.  There
are several different commonly used definitions of the rank of an abelian group, and we emphasize
that our $\rank(A)$ is {\em not} the maximal $n$ such that $A$ contains a subgroup isomorphic to $\Z^n$.
In particular, $\rank(\Z/\ell) = 1$ for $\ell \geq 2$, and $\rank(A) = 0$ if and only if $A = 0$.}
for $A$.

\begin{maintheorem}
\label{maintheorem:stable}
Let $A$ be a finitely generated abelian group, let $\mu$ be an $A$-homology marking on $\Sigma_g^1$, and let
$\mu'$ be its stabilization to $\Sigma_{g+1}^1$.  The map
$\HH_k(\Torelli(\Sigma_g^1,\mu)) \rightarrow \HH_k(\Torelli(\Sigma_{g+1}^1,\mu'))$
induced by the stabilization map 
$\Torelli(\Sigma_g^1,\mu) \rightarrow \Torelli(\Sigma_{g+1}^1,\mu')$
is an isomorphism if $g \geq (\rank(A)+2)k + (2\rank(A)+2)$ and
a surjection if $g = (\rank(A)+2)k+(2\rank(A)+1)$.
\end{maintheorem}

\p{Closed surface trouble}
Harer's stability theorem implies that the map $\Mod(\Sigma_g^1) \rightarrow \Mod(\Sigma_g)$ arising
from gluing a disc to $\partial \Sigma_g^1$ induces an isomorphism on $\HH_k$ for $g \gg k$.  One might
expect a similar result to hold for the partial Torelli groups.  Unfortunately, this is 
completely false.  In \S \ref{section:closed}, we will prove that it fails
even for $\HH_1$ for $A$-homology markings satisfying a mild nondegeneracy condition
called {\em symplectic nondegeneracy}.  One special case of this is the following.
For $1 \leq h \leq g$, define $\Torelli(\Sigma_g,h)$ just like $\Torelli(\Sigma_g^1,h)$,
so we have a surjection $\Torelli(\Sigma_g^1,h) \rightarrow \Torelli(\Sigma_g,h)$.

\begin{maintheorem}
\label{maintheorem:closed}
For $h \leq g$ with $g \geq 3$ and $h \geq 2$, the map
$\HH_1(\Torelli(\Sigma_g^1,h)) \rightarrow \HH_1(\Torelli(\Sigma_g,h))$
is not an isomorphism.
\end{maintheorem}

The proof uses an extension of the Johnson homomorphism to the partial Torelli groups
that was constructed by Broaddus--Farb--Putman \cite{BroaddusFarbPutman}.

\p{Multiple boundary components}
In addition to Theorem \ref{maintheorem:stable} concerning surfaces with one boundary component, we also
have a theorem for surfaces with multiple boundary components.  The correct statement here is a bit
subtle since the phenomenon underlying Theorem \ref{maintheorem:closed} also obstructs many obvious
kinds of generalizations.  The purpose of having a generalization like this is to understand
how the partial Torelli groups restrict to subsurfaces, which turns out to be fundamental in
the author's forthcoming work on the cohomology of the moduli space of curves with level
structures \cite{PutmanStableLevel}.  Here is an example of the kind of result we prove; in fact,
this is precisely the special case needed in \cite{PutmanStableLevel}.  

\begin{example}
Consider an $A$-homology marking $\mu$ on $\Sigma_g^1$.  For some $h \geq 1$, let
$\mu'$ be its stabilization to $\Sigma_{g+h}^1$.  Consider the following
subsurfaces $S \cong \Sigma_g^{1+h}$ and $S' \cong \Sigma_g^{1+2h}$ of $\Sigma_{g+h}^1$:\\
\centerline{\psfig{file=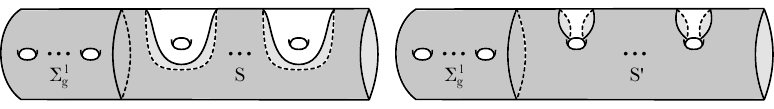,scale=100}}
Both $S$ and $S'$ include the entire shaded subsurface (including $\Sigma_g^1$).
The inclusions $S \hookrightarrow \Sigma_{g+h}^1$ and $S' \hookrightarrow \Sigma_{g+h}^1$ induce homomorphisms
$\phi\colon \Mod(S) \rightarrow \Mod(\Sigma_{g+h}^1)$ and $\psi\colon \Mod(S') \rightarrow \Mod(\Sigma_{g+h}^1)$;
define $\Torelli(S,\mu') = \phi^{-1}(\Torelli(\Sigma_{g+h}^1,\mu'))$ and 
$\Torelli(S',\mu') = \psi^{-1}(\Torelli(\Sigma_{g+h}^1,\mu'))$.  Be warned: while it turns out that
$\Torelli(S,\mu')$ can be defined using the action of $\Mod(S)$ on $\HH_1(S)$, the group $\Torelli(S',\mu')$ cannot be
defined using only $\HH_1(S')$.
Then our theorem will show that the
map
\[\HH_k(\Torelli(S,\mu')) \longrightarrow \HH_k(\Torelli(S',\mu'))\]
is an isomorphism if the genus of $S$ (namely $g$) is at least $(\rank(A)+2)k + (2\rank(A)+2)$.  However,
except in degenerate cases the maps
\[\HH_1(\Torelli(\Sigma_g^1,\mu)) \longrightarrow \HH_1(\Torelli(S,\mu')) \quad \text{and} \quad
\HH_1(\Torelli(S,\mu')) \rightarrow \HH_1(\Torelli(\Sigma_{g+h}^1,\mu'))\]
are never isomorphisms no matter how large $g$ is. 
\end{example}

In the above example, we defined the partial Torelli groups on surfaces with multiple boundary components in 
an ad-hoc way.  Correctly formulating our theorem requires a more intrinsic definition, and we define
a category of ``homology-marked surfaces'' with multiple boundary components that
is inspired by the author's work on the Torelli group on surfaces with
multiple boundary components in \cite{PutmanCutPaste}.

\p{Nonabelian markings}
We also have a theorem for nonabelian markings, whose definition is as follows.\footnote{I am not sure
who first defined this concept.  Related things appear, e.g., in \cite{DunfieldThurston, Ivanov}.}
Fix a basepoint $\ast \in \partial \Sigma_g^1$.  For a group $\Lambda$, a {\em $\Lambda$-marking} on
$\Sigma_g^1$ is a group homomorphism $\mu \colon \pi_1(\Sigma_g^1,\ast) \rightarrow \Lambda$.
If $\Lambda$ is abelian, then this is equivalent to a $\Lambda$-homology marking on $\Sigma_g^1$.  Given a $\Lambda$-marking
$\mu \colon \pi_1(\Sigma_g^1,\ast) \rightarrow \Lambda$, define the associated partial
Torelli group via the formula
\[\Torelli(\Sigma_g^1,\mu) = \Set{$f \in \Mod(\Sigma_g^1)$}{$\mu(f(x))=\mu(x)$ for all $x \in \pi_1(\Sigma_g^1,\ast)$}.\]
Again, this reduces to our previous definition if $\Lambda$ is abelian.

\p{Nonabelian stabilization}
Let $\mu$ be a $\Lambda$-marking on $\Sigma_g^1$.  Due to basepoint issues, stabilizing 
$\mu$ to $\Sigma_{g+1}^1$ is a little more complicated than the case of homology markings.
Let $\ast \in \partial \Sigma_g^1$ and $\ast' \in \partial \Sigma_{g+1}^1$ be the basepoints.
Embed $\Sigma_g^1$ into $\Sigma_{g+1}^1$ as in the following figure:\\
\centerline{\psfig{file=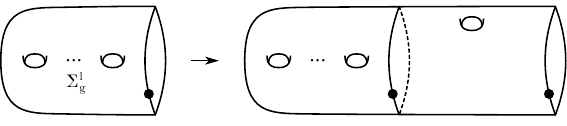,scale=100}}
Let $\lambda$ and $\eta$ and $S \cong \Sigma_1^1$ be as in the following figure:\\
\centerline{\psfig{file=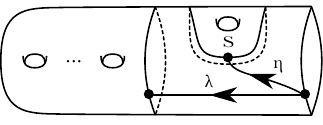,scale=100}}
Letting $\ast'' \in \partial S$ be the basepoint of $S$ as in that figure, the paths
$\lambda$ and $\eta$ induce injective homomorphisms
\[\pi_1(\Sigma_g^1,\ast) \hookrightarrow \pi_1(\Sigma_{g+1}^1,\ast') \quad \text{and} \quad \pi_1(S,\ast'') \hookrightarrow \pi_1(\Sigma_{g+1}^1,\ast')\]
taking $x \in \pi_1(\Sigma_g^1,\ast)$ to $\lambda \cdot x \cdot \lambda^{-1} \in \pi_1(\Sigma_{g+1}^1,\ast')$ and
$y \in \pi_1(S,\ast'')$ to $\eta \cdot y \cdot \eta^{-1} \in \pi_1(\Sigma_{g+1}^1,\ast')$.  Identifying
$\pi_1(\Sigma_g^1,\ast)$ and $\pi_1(S,\ast'')$ with the corresponding subgroups of $\pi_1(\Sigma_{g+1}^1,\ast')$, we
have a free product decomposition
\[\pi_1(\Sigma_{g+1}^1,\ast') = \pi_1(\Sigma_g^1,\ast) \star \pi_1(S,\ast'').\]
Define the stabilization $\mu'\colon \pi_1(\Sigma_{g+1}^1,\ast') \rightarrow \Lambda$ of
$\mu\colon \pi_1(\Sigma_g^1,\ast) \rightarrow \Lambda$ to be the composition
\[\pi_1(\Sigma_{g+1}^1,\ast') = \pi_1(\Sigma_g^1,\ast) \star \pi_1(S,\ast'') \longrightarrow \pi_1(\Sigma_g^1,\ast) \stackrel{\mu}{\longrightarrow} \Lambda,\]
where the first arrow quotients out by the normal closure of $\pi_1(S,\ast'')$.  
Just like in the abelian setting, the map $\Mod(\Sigma_g^1) \rightarrow \Mod(\Sigma_{g+1}^1)$ 
induced by our embedding 
$\Sigma_g^1 \hookrightarrow \Sigma_{g+1}^1$ restricts to a map 
$\Torelli(\Sigma_g^1,\mu) \rightarrow \Torelli(\Sigma_{g+1}^1,\mu')$ that we will call the 
{\em stabilization map}.

\p{Nonabelian stability}
Our main theorem about this is as follows.  It can be viewed as an analogue for the mapping
class group of a theorem of Ellenberg--Venkatesh--Westerland \cite[Theorem 6.1]{EllenbergVenkateshWesterland} 
concerning braid groups and Hurwitz spaces.

\begin{maintheorem}
\label{maintheorem:nonabelianstable}
Let $\Lambda$ be a finite group, let $\mu$ be a $\Lambda$-marking on $\Sigma_g^1$, and let $\mu'$ be its
stabilization to $\Sigma_{g+1}^1$.  The map
$\HH_k(\Torelli(\Sigma_g^1,\mu)) \rightarrow \HH_k(\Torelli(\Sigma_{g+1}^1,\mu'))$
induced by the stabilization map
$\Torelli(\Sigma_g^1,\mu) \rightarrow \Torelli(\Sigma_{g+1}^1,\mu')$
is an isomorphism if $g \geq (|\Lambda|+2)k + (2|\Lambda|+2)$ and a surjection
if $g = (|\Lambda|+2)k+(2|\Lambda|+1)$.
\end{maintheorem}

\begin{remark}
Ellenberg--Venkatesh--Westerland's main application in \cite{EllenbergVenkateshWesterland}
of their stability result concerns point-counting in Hurwitz spaces via the Weil conjectures.
Unfortunately, the vast amount of unknown unstable cohomology precludes such applications
here.
\end{remark}

\begin{remark}
If $\Lambda$ is a finite abelian group, then Theorems \ref{maintheorem:stable} and
\ref{maintheorem:nonabelianstable} give a similar kind of stability, but the bounds
in Theorem \ref{maintheorem:stable} are much stronger.
\end{remark}

\begin{remark}
Because of basepoint issues, stating a version of Theorem \ref{maintheorem:nonabelianstable} on
surfaces with multiple boundary components would be rather technical, and unlike for
Theorem \ref{maintheorem:stable} we do not know any potential applications of such a result.  We
thus do not pursue this kind of generalization of Theorem \ref{maintheorem:nonabelianstable}.
\end{remark}

\p{Proof techniques}
There is an enormous literature on homological stability theorems, starting
with unpublished work of Quillen on $\GL_n(\Field_p)$.  A standard
proof technique has emerged that first appeared in its modern formulation
in \cite{VanDerKallen}.  Consider a sequence of groups
\begin{equation}
\label{eqn:homstab}
G_0 \subset G_1 \subset G_2 \subset \cdots
\end{equation}
that we want to prove enjoys homological stability, i.e., $\HH_k(G_{n-1}) \cong \HH_k(G_{n})$
for $n \gg k$.  To compute $\HH_k(G_n)$, we would need a
contractible simplicial complex on which $G_n$ acts freely.  Since we are only
interested in the low-degree homology groups, we can weaken contractibility
to high connectivity.  The key insight for homological stability
is that since we only want to compare $\HH_k(G_n)$ with the homology of previous
groups in \eqref{eqn:homstab}, what we want is not a free
action but one whose stabilizer subgroups are related to the previous groups.

\p{Machine}
There are many variants on the above machine.  For proving homological stability for
the groups $G_n$ in \eqref{eqn:homstab}, the easiest version
requires simplicial complexes $X_n$ upon which $G_n$ acts with the following three properties:
\setlength{\parskip}{0pt}
\begin{compactitem}
\item The connectivity of $X_n$ goes to $\infty$ as $n \mapsto \infty$.
\item For $0 \leq k \leq n-1$, the $G_n$-stabilizer of a $k$-simplex of $X_n$ is conjugate
to $G_{n-k-1}$.
\item The group $G_n$ acts transitively on the $k$-simplices of $X_n$ for all $k \geq 0$.
\end{compactitem}
Some additional technical hypotheses are needed as well; we will review these
in \S \ref{section:stabilitymachine}.
Hatcher--Vogtmann \cite{HatcherVogtmannTethers} constructed such $X_n$ for the mapping 
class group.  Our proof of Theorem \ref{maintheorem:stable} is inspired by their work,
so we start by describing a variant of it.\setlength{\parskip}{\baselineskip}

\p{Subsurface complex}
For $h \geq 1$, the {\em complex of genus $h$ subsurfaces} of $\Sigma_g^b$, denoted
$\Sur_{h}(\Sigma_g^b)$, is the simplicial complex whose $k$-simplices are sets
$\{\iota_0,\ldots,\iota_k\}$ of isotopy classes of orientation-preserving embeddings 
$\iota_i\colon \Sigma_{h}^1 \rightarrow \Sigma_g^b$ that can be isotoped such that
for $0 \leq i < j \leq k$, the subsurfaces $\iota_i(\Sigma_{h}^1)$ and
$\iota_j(\Sigma_{h}^1)$ are disjoint.  The group $\Mod(\Sigma_g^b)$ acts
on $\Sur_{h}(\Sigma_g^b)$.  However, it turns out that this is not quite
the right complex for homological stability.

\p{Tethered subsurfaces}
Let $\tau(\Sigma_{h}^1)$ be the result of gluing the interval $[0,1]$ to $\Sigma_h^1$ by identifying
$1 \in [0,1]$ with a point of $\partial \Sigma_h^1$.
The subset $[0,1] \subset \tau(\Sigma_{h}^1)$
is the {\em tether} and $0 \in [0,1] \subset \tau(\Sigma_{h}^1)$ the {\em initial point}
of the tether.  Let $I \subset \partial \Sigma_g^b$ be a finite disjoint union of open intervals.
An {\em $I$-tethered genus $h$ subsurface}
of $\Sigma_g^b$ is an embedding $\iota\colon \tau(\Sigma_{h}^1) \rightarrow \Sigma_g^b$
taking the initial point of the tether to a point of $I$
whose restriction to $\Sigma_{h}^1$ preserves the orientation.  For instance, here
is an $I$-tethered genus $2$ subsurface:\\
\centerline{\psfig{file=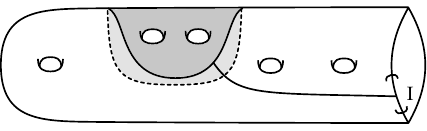,scale=100}}

\p{Tethered subsurface complex}
The {\em complex of $I$-tethered genus $h$ subsurfaces} of $\Sigma_g^b$, 
denoted $\TS_{h}(\Sigma_g^b,I)$,
is the simplicial complex whose $k$-simplices are collections
$\{\iota_0,\ldots,\iota_k\}$ of isotopy classes of $I$-tethered genus
$h$ subsurfaces of $\Sigma_g^b$ that can be realized disjointly.  These
isotopies are allowed to move the images of the initial points of the tethers within $I$,
so the tethers can be slid past each other and made disjoint.
For instance, here is a $2$-simplex in 
$\TS_{1}(\Sigma_5^1,I)$:\\
\centerline{\psfig{file=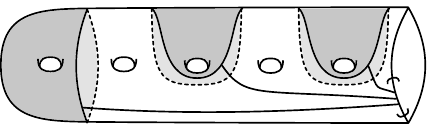,scale=100}}

\p{High connectivity}
The complexes $\Sur_1(\Sigma_g^b)$ and $\TS_1(\Sigma_g^b,I)$ 
are closely related to
complexes that were introduced by Hatcher--Vogtmann 
\cite{HatcherVogtmannTethers}, and it follows easily from their work that they
are $\frac{g-3}{2}$-connected (see \cite[proof of Theorem 6.25]{PutmanSamLinear} for details). 
We generalize this as follows:

\begin{maintheorem}
\label{maintheorem:subsurfacescon}
Consider $g \geq h \geq 1$ and $b \geq 0$.
\setlength{\parskip}{0pt}
\begin{compactitem}
\item The complex $\Sur_{h}(\Sigma_g^b)$ is $\frac{g-(2h+1)}{h+1}$-connected.
\item Assume that $b \geq 1$, and let $I \subset \partial \Sigma_g^b$ be a finite disjoint
union of open intervals.
The complex $\TS_{h}(\Sigma_g^b,I)$ is $\frac{g-(2h+1)}{h+1}$-connected.
\setlength{\parskip}{\baselineskip}
\end{compactitem}
\end{maintheorem}

\begin{remark}
Our convention is that a space is $(-1)$-connected if it is nonempty.  Using
this convention, the genus bounds for $(-1)$-connectivity and $0$-connectivity
in Theorem \ref{maintheorem:subsurfacescon} are sharp.  We do not know whether they
are sharp for higher connectivity.
\end{remark}

\begin{remark}
Hatcher--Vogtmann's proof in \cite{HatcherVogtmannTethers} that
$\Sur_1(\Sigma_g^b)$ and $\TS_1(\Sigma_g^b,I)$ are $\frac{g-3}{2}$-connected
is closely connected to their proof that the separating curve complex
is $\frac{g-3}{2}$-connected.  Looijenga \cite{LooijengaSep} later
showed that the separating curve complex is $(g-3)$-connected.
Unfortunately, his techniques do not appear to give an improvement to
Theorem \ref{maintheorem:subsurfacescon}.
\end{remark}

\begin{remark}
In applications to homological stability, we will only use complexes made out
of genus $1$ subsurfaces.  However, the more general result of 
Theorem \ref{maintheorem:subsurfacescon} will be needed for the proof even of the $h=1$
case of Theorem \ref{maintheorem:vanishtetheredconone} below.
\end{remark}

\p{Mod stability}
Consider the groups
\begin{equation}
\label{eqn:modstab}
\Mod(\Sigma_1^1) \subset \Mod(\Sigma_2^1) \subset \Mod(\Sigma_3^1) \subset \cdots
\end{equation}
Let $I \subset \partial \Sigma_g^1$ be an open interval.
The group $\Mod(\Sigma_g^1)$ acts on $\TS_1(\Sigma_g^1,I)$, and this complex has all three properties
needed by the machine to prove homological stability for \eqref{eqn:modstab}:
\setlength{\parskip}{0pt}
\begin{compactitem}
\item As we said above, $\TS_1(\Sigma_g^1,I)$ is $\frac{g-3}{2}$-connected.
\item The $\Mod(\Sigma_g^1)$-stabilizer of a $k$-simplex $\{\iota_0,\ldots,\iota_k\}$ of
$\TS_1(\Sigma_g^1,I)$ is the mapping class group of the complement of a regular neighborhood of
\[\partial \Sigma_g^1 \cup \iota_0\left(\tau\left(\Sigma_{1}^1\right)\right) \cup \cdots \cup \iota_k\left(\tau\left(\Sigma_1^1\right)\right).\]
See here:
\end{compactitem}
\centerline{\psfig{file=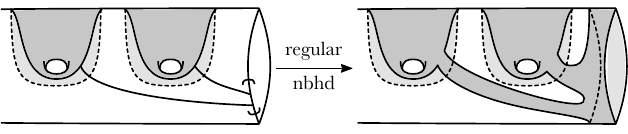,scale=100}}
\begin{compactitem}
\item[] This complement is homeomorphic to $\Sigma_{g-k-1}^1$, so this stabilizer is isomorphic to 
$\Mod(\Sigma_{g-k-1}^1)$.  All such subsurface mapping class groups are conjugate; this follows
from the ``change of coordinates principle'' from \cite[\S 1.3.2]{FarbMargalitPrimer}.
\item Another application of the ``change of coordinates principle'' shows that
$\Mod(\Sigma_g^1)$ acts transitively on the $k$-simplices of $\TS_1(\Sigma_g^1,I)$.
\end{compactitem}
\setlength{\parskip}{\baselineskip}

\p{Partial Torelli problem}
A first idea for proving homological stability for the partial Torelli 
groups $\Torelli(\Sigma_g^1,\mu)$ is to consider their actions on $\TS_1(\Sigma_g^1,I)$.
Unfortunately, this does not work.  The fundamental problem is that
$\Torelli(\Sigma_g^1,\mu)$ does not act transitively on the $k$-simplices
of $\TS_1(\Sigma_g^1,I)$; indeed, it does not even act transitively on the vertices.
For $A$-homology markings $\mu$, the issue is that for an $I$-tethered torus 
$\iota\colon \tau(\Sigma_1^1) \rightarrow \Sigma_g^1$ and $f \in \Torelli(\Sigma_g^1,\mu)$,
the compositions
\[\HH_1(\Sigma_1^1) \cong \HH_1(\tau(\Sigma_1^1)) \stackrel{\iota_{\ast}}{\longrightarrow} \HH_1(\Sigma_g^1) \stackrel{\mu}{\longrightarrow} A \quad \text{and} \quad
\HH_1(\Sigma_1^1) \cong \HH_1(\tau(\Sigma_1^1)) \stackrel{(f \circ \iota)_{\ast}}{\longrightarrow} \HH_1(\Sigma_g^1) \stackrel{\mu}{\longrightarrow} A\]
will be the same, but the functions 
$\mu \circ \iota_{\ast} \colon \HH_1(\tau(\Sigma_1^1)) \rightarrow A$ need not be the same
for different tethered tori.  A similar issue arises in the nonabelian setting.
To fix this, we use a subcomplex of $\TS_1(\Sigma_g^1,I)$
that is adapted to $\mu$.

\begin{remark}
The stabilizers are also wrong, but fixing the transitivity will
also fix this.  
\end{remark}

\p{Vanishing surfaces}
For an $A$-homology marking $\mu$ on $\Sigma_g^1$, define
$\TS_{h}(\Sigma_g^1,I,\mu)$ to be the full 
subcomplex of $\TS_{h}(\Sigma_g^1,I)$ spanned by
vertices $\iota$ such that the composition
\[\HH_1(\tau(\Sigma_{h}^1)) \stackrel{\iota_{\ast}}{\longrightarrow} \HH_1(\Sigma_g^1) \stackrel{\mu}{\longrightarrow} A\]
is the zero map.  We will show that $\Torelli(\Sigma_g^1,\mu)$ acts transitively
on the $k$-simplices of $\TS_{1}(\Sigma_g^1,I,\mu)$ (at least for $k$ not too large).
However, there is a problem: a priori the subcomplex
$\TS_{1}(\Sigma_g^1,I,\mu)$ of $\TS_{1}(\Sigma_g^1,I)$ might not be highly connected.  Our third 
main theorem says that in fact it is 
$\frac{g-(2\rank(A)+3)}{\rank(A)+2}$-connected.
More generally, we prove the following:

\begin{maintheorem}
\label{maintheorem:vanishtetheredconone}
Let $A$ be a finitely generated abelian group, let $\mu$ be an $A$-homology marking on $\Sigma_g^1$, and let 
$I \subset \partial \Sigma_g^1$ be a finite disjoint union of open intervals.  Then
the complex $\TS_{h}(\Sigma_g^1,I,\mu)$ is $\frac{g-(2\rank(A)+2h+1)}{\rank(A)+h+1}$-connected.
\end{maintheorem}

We also prove a similar theorem in the nonabelian setting.

\p{Outline}
We start in \S \ref{section:subsurfacescon} 
by proving Theorem \ref{maintheorem:subsurfacescon}.
We then prove Theorems \ref{maintheorem:stable}, \ref{maintheorem:nonabelianstable}, and \ref{maintheorem:vanishtetheredconone}
in \S \ref{section:stableone}.
Next, in \S \ref{section:marked} we define a category of homology-marked surfaces
with multiple boundary components.  
In \S \ref{section:maintheorem}, we use our
category to state and prove Theorem \ref{theorem:stableboundary}, which
generalizes Theorem \ref{maintheorem:stable} to surfaces with multiple boundary
components.  This proof depends on a stabilization result which is proved in
\S \ref{section:doubleboundarystabilization}.
We close with \S \ref{section:closed}, which proves Theorem \ref{maintheorem:closed}.

\p{Conventions}
Throughout this paper, $A$ denotes a fixed finitely generated abelian group and
$\Lambda$ is a fixed finite group.  We also fix a basepoint $\ast \in \partial \Sigma_g^1$.

\p{Acknowledgments}
I want to thank Jordan Ellenberg for a useful correspondence and Allen Hatcher for pointing out a confusing typo.  I also want to thank the referee for their very careful reading of the paper and many helpful suggestions.

\section{The complex of subsurfaces}
\label{section:subsurfacescon}

This section is devoted to the proof of Theorem \ref{maintheorem:subsurfacescon},
which asserts that $\Sur_{h}(\Sigma_g^b)$ and $\TS_{h}(\Sigma_g^b,I)$
are highly connected.  There are three parts: \S \ref{section:conprop} contains
a technical result about fibers of maps, \S \ref{section:linkarguments} discusses ``link arguments'', and
\S \ref{section:surfacescon} proves Theorem \ref{maintheorem:subsurfacescon}.

\subsection{Fibers of maps}
\label{section:conprop}

Our proofs will require a technical tool.  

\p{Homotopy theory conventions}
A space $X$ is said to be $n$-connected if for $k \leq n$, all maps
$S^k \rightarrow X$ extend to maps $D^{k+1} \rightarrow X$.  Since
$S^{-1} = \emptyset$ and $D^0$ is a single point, a space is $(-1)$-connected
precisely when it is nonempty.  A map $\psi\colon X \rightarrow Y$ of spaces
is an $n$-homotopy equivalence if for all $0 \leq k \leq n$, the induced map
$[S^k,X] \rightarrow [S^k,Y]$ on unbased homotopy classes of maps out of $S^k$
is a bijection.  This is equivalent to saying that the induced map
on $\pi_k$ is a bijection for each choice of basepoint.

\p{Relative fibers}
If $\psi\colon X \rightarrow Y$
is a map of simplicial complexes, $\sigma$ is a simplex of $Y$, and $\sigma'$ is
a face of $\sigma$, then denote by $\Fib_{\psi}(\sigma',\sigma)$ the subcomplex
of $X$ consisting of all simplices $\eta'$ of $X$ with the following properties:
\setlength{\parskip}{0pt}
\begin{compactitem}
\item $\psi(\eta')$ is a face of $\sigma'$, and
\item there exists a simplex $\eta$ of $X$ such that $\eta'$ is a face of
$\eta$ and $\psi(\eta) = \sigma$.
\end{compactitem}
For instance, consider the following map, where $\psi$ takes each $1$-simplex $\sigma'_i$
to $\sigma'$ (with the specified orientation) and each $2$-simplex $\sigma_i$ to $\sigma$:\\
\centerline{\psfig{file=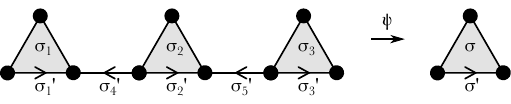,scale=100}}
The relative fiber $\Fib_{\psi}(\sigma',\sigma)$ then consists of $\sigma_1'$ and
$\sigma_2'$ and $\sigma_3'$ (but not $\sigma_4'$ or $\sigma_5'$). \setlength{\parskip}{\baselineskip}

\p{Fiber lemma}
With these definitions, we have the following lemma.

\begin{lemma}
\label{lemma:fiberlemma}
Let $\psi\colon X \rightarrow Y$ be a map of simplicial complexes.  For some
$n \geq 0$, assume the space $\Fib_{\psi}(\sigma',\sigma)$ is $n$-connected
for all simplices $\sigma$ of $Y$ and all faces $\sigma'$ of $\sigma$.
Then $\psi$ is an $n$-homotopy equivalence.
\end{lemma}
\begin{proof}
Replacing $Y$ by its $(n+1)$-skeleton $Y_{n+1}$ and $X$ by $\psi^{-1}(Y_{n+1})$, we can
assume that $Y$ is finite-dimensional.  The proof will be by induction on 
$m=\dim(Y)$.  The base case $m=0$ is trivial since in that case $Y$
is a discrete set of points and our assumptions imply that the fiber over
each of these points is $n$-connected.  Assume now that $m \geq 1$.  The
key step in the proof is the following claim.

\begin{claim}
Assume that $Y$ is the union of a subcomplex $Y'$ and an $m$-simplex $\sigma$ with
$\sigma \cap Y' = \partial \sigma$.
Define $X' = \psi^{-1}(Y')$, and assume that
$\psi\colon X \rightarrow Y$ restricts 
to an $n$-homotopy equivalence $\psi'\colon X' \rightarrow Y'$.  Then
$\psi$ is an $n$-homotopy equivalence.
\end{claim}
\begin{proof}[Proof of claim]
Let $X'' = \Fib_{\psi}(\sigma,\sigma)$.  In other words, $X''$ consists of all simplices of $X$ mapping surjectively onto $\sigma$, along with
their faces.  We thus have $X = X' \cup X''$.  By assumption, $X''$ is $n$-connected,
which implies that $\psi$ restricts to an $n$-homotopy equivalence 
$\psi''\colon X'' \rightarrow \sigma$.  Define $Z = X' \cap X''$.  The map
$\psi$ restricts to a map $\psi_Z\colon Z \rightarrow \partial \sigma$.

We now come to the key observation: the space $Z$ is precisely
the subcomplex of $X$ consisting of the union of the subcomplexes
$\Fib_{\psi}(\sigma',\sigma)$ as $\sigma'$ ranges over all simplices of $\partial \sigma$.  
Moreover, for all simplices $\sigma'$ of $\partial \sigma$ and all faces $\sigma''$
of $\sigma'$, we have $\Fib_{\psi_Z}(\sigma'',\sigma') = \Fib_{\psi}(\sigma'',\sigma)$,
and thus by assumption $\Fib_{\psi_Z}(\sigma'',\sigma')$ is $n$-connected.  We can
therefore apply our inductive hypothesis to see that 
$\psi_Z\colon Z \rightarrow \partial \sigma \cong S^{m-1}$ is an $n$-homotopy equivalence.

Summing up, we have $X = X' \cup X''$ and $Y = Y' \cup \sigma$.  The map $\psi$
restricts to $n$-homotopy equivalences
\[\psi'\colon X' \rightarrow Y' \quad \text{and} \quad \psi''\colon X'' \rightarrow \sigma \quad \text{and} \quad \psi_Z\colon X' \cap X'' = Z \rightarrow \partial \sigma = Y' \cap \sigma.\]
The map $\psi$ induces a map between the Mayer--Vietoris exact sequences
associated to the decompositions $X = X' \cup X''$ and $Y = Y' \cup \sigma$:
\[\begin{CD}
\cdots @>>> \HH_k(X' \cap X'') @>>> \HH_k(X') \oplus \HH_k(X'') @>>> \HH_k(X) @>>> \cdots\\
@.          @VVV                    @VVV                             @VVV          @.\\
\cdots @>>> \HH_k(Y' \cap \sigma) @>>> \HH_k(Y') \oplus \HH_k(\sigma) @>>> \HH_k(Y) @>>> \cdots\\
\end{CD}\]
Other than the maps $\HH_k(X) \rightarrow \HH_k(Y)$, the vertical maps in this commutative diagram are isomorphisms for $k \leq n$, so by
the five-lemma the maps $\HH_k(X) \rightarrow \HH_k(Y)$ are also isomorphisms for $k \leq n$.
This implies in particular that the map $X \rightarrow Y$ is $0$-connected, and thus
induces a bijection between path-components.  If $n \geq 1$, then a similar argument on each path component
using the Seifert--van Kampen theorem shows that the map $\psi\colon X \rightarrow Y$ induces
an isomorphism on $\pi_1$ for each choice of basepoint.  This allows us to identify
local coefficient systems on $Y$ with local coefficient systems on $X$, and for each local 
coefficient system $A$ on $Y$
we can run the above Mayer--Vietoris argument on homology with coefficients in $A$
to prove that the map $\psi\colon \HH_k(X;A) \rightarrow \HH_k(Y;A)$ is an isomorphism
for $k \leq n$.  Applying the non-simply-connected version of
Whitehead's theorem \cite[Theorem 6.71]{DavisKirk}, 
we deduce that the map $X \rightarrow Y$ is an $n$-homotopy equivalence, as desired.
\end{proof}

Repeatedly applying this claim, we see that the lemma holds for $m$-dimensional
$Y$ with finitely many $m$-simplices.  

The general case reduces to the case
where $Y$ has finitely many $m$-simplices as follows.  Consider some $0 \leq k \leq n$.
Our goal is to prove that the map $[S^k,X] \rightarrow [S^k,Y]$ induced by
$\psi$ is a bijection.  The proofs that it is injective and surjective are
similar compactness arguments, so we give the details for surjectivity and
leave injectivity to the reader.

Consider a map $f\colon S^k \rightarrow Y$.  By compactness, $f(S^k)$ lies in a subcomplex
of $Y'$ of $Y$ containing the $(m-1)$-skeleton and finitely many $m$-simplices.  Letting
$X' = \psi^{-1}(Y')$, we know that the map $[S^k,X'] \rightarrow [S^k,Y']$ is a bijection,
so there exists some $\widetilde{f}\colon S^k \rightarrow X'$ such that 
$\psi \circ \widetilde{f} \colon S^k \rightarrow Y'$ is homotopic to $f$.  It follows
that the map $[S^k,X] \rightarrow [S^k,Y]$ induced by $\psi$ is surjective, as desired.
\end{proof}

\begin{corollary}
\label{corollary:fibers}
Let $\psi\colon X \rightarrow Y$ be a map of simplicial complexes.  For some
$n \geq 0$, assume that the following hold.
\setlength{\parskip}{0pt}
\begin{compactitem}
\item $Y$ is $n$-connected.
\item All $(n+1)$-simplices of $Y$ are in the image of $\psi$.
\item For all simplices $\sigma$ of $Y$ whose dimension is at most $n$ and
all faces $\sigma'$ of $\sigma$, the space $\Fib_{\psi}(\sigma',\sigma)$ is $n$-connected.
\end{compactitem}
Then $X$ is $n$-connected.\setlength{\parskip}{\baselineskip}
\end{corollary}
\begin{proof}
Let $Y'$ be the $n$-skeleton of $Y$ and $X' = \psi^{-1}(Y')$, so $X'$ contains the
$n$-skeleton of $X$.  Let $\psi'\colon X' \rightarrow Y'$ be the restriction of
$\psi$ to $X'$.  Our assumptions allow us to apply Lemma \ref{lemma:fiberlemma}
to $\psi'$, so $\psi'$ is an $n$-homotopy equivalence.  Since $Y$ is $n$-connected,
the space $Y'$ is $(n-1)$-connected, so this implies that $X'$ and thus $X$ are
$(n-1)$-connected.  We also know that the induced map 
$\psi'\colon \pi_n(X') \rightarrow \pi_n(Y')$ is an isomorphism.  Since $Y$
is $n$-connected, attaching the $(n+1)$-simplices of $Y$ to $Y'$ kills
$\pi_n(Y')$.  By assumption, for each of these $(n+1)$-simplices $\sigma$ of $Y$
there is an $(n+1)$-simplex $\tsigma$ of $X$ such that $\psi(\tsigma)=\sigma$.
It follows that the element of $\pi_n(Y')$ represented by 
$\partial \sigma \rightarrow Y'$ lifts to the element of $\pi_n(X')$ represented
by $\partial \tsigma \rightarrow X'$.
We conclude that attaching to $X'$ the $(n+1)$-simplices of $X$ that do not already lie
in $X'$ kills $\pi_n(X')$, which implies that $\pi_n(X)=0$, as desired.
\end{proof}

\subsection{Link arguments}
\label{section:linkarguments}

Let $X$ be a simplicial complex and let $Y \subset X$ be a subcomplex.  This section is
devoted to a result of Hatcher--Vogtmann \cite{HatcherVogtmannTethers} that gives
conditions under which the pair $(X,Y)$ is $n$-connected, i.e., $\pi_k(X,Y)=0$ for $0 \leq k \leq n$.
The idea is to identify a collection $\cB$ of ``bad simplices'' of $X$ that characterize
$Y$ in the sense that a simplex lies in $Y$ precisely when none of its faces lie in $\cB$.
We then have to understand the local topology of $Y$ around a simplex of $\cB$.  To that
end, if $\cB$ is a collection of simplices of $X$ and $\sigma \in \cB$, then define $G(X,\sigma,\cB)$ 
to be the subcomplex of $X$ consisting of simplices $\sigma'$ satisfying the following two
conditions:
\setlength{\parskip}{0pt}
\begin{compactitem}
\item The join $\sigma \ast \sigma'$ is a simplex of $X$, i.e., $\sigma'$ is a simplex in the link of $\sigma$.
\item If $\sigma''$ is a face of $\sigma \ast \sigma'$ such that $\sigma'' \in \cB$, then $\sigma'' \subset \sigma$.
\end{compactitem}
Hatcher--Vogtmann's result is then as follows.

\begin{proposition}[{\cite[Proposition 2.1]{HatcherVogtmannTethers}}]
\label{proposition:avoidbad}
Let $Y$ be a subcomplex of a simplicial complex $X$.  Assume that there exists a collection
$\cB$ of simplices of $X$ satisfying the following conditions for some $n \geq 0$:
\setlength{\parskip}{0pt}
\begin{compactitem}
\item[(i)] A simplex of $X$ lies in $Y$ if and only if none of its faces lie in $\cB$.
\item[(ii)] If $\sigma_1, \sigma_2 \in \cB$ are such that $\sigma_1 \cup \sigma_2$ is a simplex of $X$, then
$\sigma_1 \cup \sigma_2 \in \cB$.  Here $\sigma_1$ and $\sigma_2$ might share vertices, so
$\sigma_1 \cup \sigma_2$ might not be the join $\sigma_1 \ast \sigma_2$.
\item[(iii)] For all $k$-dimensional $\sigma \in \cB$, the complex $G(X,\sigma,\cB)$ is $(n-k-1)$-connected.
\end{compactitem}
Then the pair $(X,Y)$ is $n$-connected. \setlength{\parskip}{\baselineskip}
\end{proposition}

\setlength{\parskip}{\baselineskip}
As an illustration of how Proposition \ref{proposition:avoidbad} might be used, we use it to prove
the following result (which will in fact be how we use that proposition in all but two cases).

\begin{corollary}
\label{corollary:avoidsubcomplex}
Let $X$ be a simplicial complex and let $Y,Y' \subset X$ be disjoint full subcomplexes such that
every vertex of $X$ lies in either $Y$ or $Y'$.  For some $n \geq 0$, assume that for
all $k$-dimensional simplices $\sigma$ of $Y'$ the intersection of $Y$ with the link of $\sigma$ is $(n-k-1)$-connected.
Then the pair $(X,Y)$ is $n$-connected.
\end{corollary}
\begin{proof}
We will verify the hypotheses of Proposition \ref{proposition:avoidbad} for the set $\cB$ of
all simplices of $Y'$.  Since $Y$ is a full subcomplex of $X$ and all vertices of $X$ lie in
either $Y$ or $Y'$, a simplex of $X$ lies in $Y$ if and only if none of its vertices lie in $Y'$.
Hypothesis (i) follows.  Hypothesis (ii) is immediate from the fact that $Y'$ is a full subcomplex of $X$.
As for hypothesis (iii), it is immediate from the definitions that for a simplex $\sigma \in \cB$ the
complex $G(X,\sigma,\cB)$ is precisely the intersection of the link of $\sigma$ with $Y$.
\end{proof}

\subsection{Subsurface complexes}
\label{section:surfacescon}

\setlength{\parskip}{\baselineskip}
We now prove Theorem \ref{maintheorem:subsurfacescon}, which says
that $\Sur_{h}(\Sigma_g^b)$ and $\TS_{h}(\Sigma_g^b,I)$ are
$\frac{g-(2h+1)}{h+1}$-connected.

\begin{proof}[Proof of Theorem \ref{maintheorem:subsurfacescon}]
The proofs that $\Sur_{h}(\Sigma_g^b)$ and $\TS_{h}(\Sigma_g^b,I)$ are
$\frac{g-(2h+1)}{h+1}$-connected are similar.  Keeping track of the tethers
introduces a few complications, so we will give the details for
$\TS_{h}(\Sigma_g^b,I)$ and leave $\Sur_h(\Sigma_g^b)$ to the reader.

The proof that $\TS_{h}(\Sigma_g^b,I)$ is $\frac{g-(2h+1)}{h+1}$-connected
will be by induction on $h$.  The base case $h=1$ is
\cite[Theorem 6.25]{PutmanSamLinear}, which we remark shows how to derive it
from a closely related result of Hatcher--Vogtmann \cite{HatcherVogtmannTethers}.
For the inductive step, assume that $\TS_{h}(\Sigma_g^b,I)$ is
$\frac{g-(2h+1)}{h+1}$-connected.  
We will prove that $\TS_{h+1}(\Sigma_g^b,I)$ is $\frac{g-(2h+3)}{h+2}$-connected.

Let $\tau(\Sigma_{h}^1,\Sigma_1^1)$ be the space obtained from
$\tau(\Sigma_{h}^1) \sqcup \Sigma_1^1$ by gluing in an interval $[0,1]$
with $0$ being attached to a point of $\partial \Sigma_{h}^1$ different
from the attaching point of the tether in $\tau(\Sigma_h^1)$
and $1$ being attached to a point of $\partial \Sigma_1^1$:\\
\centerline{\psfig{file=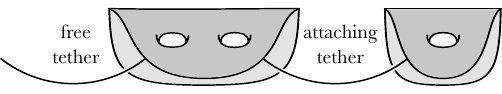,scale=100}}
The tether in $\tau(\Sigma_h^1)$ will be called the {\em free tether} and the
interval connecting $\tau(\Sigma_h^1)$ to $\Sigma_1^1$ will be called
the {\em attaching tether}.  The points $0$ of the two tethers will
be called their {\em initial points} and the points $1$ will be called
their {\em endpoints}.

Given an embedding $\tau(\Sigma_{h}^1,\Sigma_1^1) \rightarrow \Sigma_g^b$ taking
the initial point of the free tether to a point of $I$, thickening
up the attaching tether gives an $I$-tethered $\Sigma_{h+1}^1$:\\
\centerline{\psfig{file=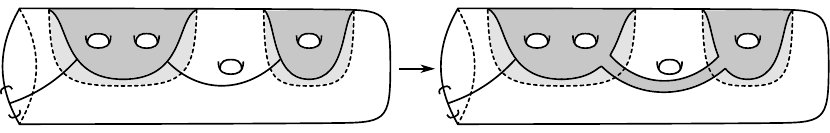,scale=100}}
In fact, there is a bijection between isotopy classes of orientation-preserving $I$-tethered
$\Sigma_{h+1}^1$ in $\Sigma_g^b$ and isotopy classes of embeddings
$\tau(\Sigma_{h}^1,\Sigma_1^1) \rightarrow \Sigma_g^b$ whose restrictions to 
$\Sigma_{h}^1$ and $\Sigma_1^1$ preserve the orientation and which take the initial point of the free
tether to a point of $I$.  For short, we will call these
{\em orientation-preserving $I$-tethered} $\tau(\Sigma_{h}^1,\Sigma_1^1)$
in $\Sigma_g^b$.  We remark that this is slightly awkward terminology since the free
tether is part of $\tau(\Sigma_{h}^1,\Sigma_1^1)$, while on the other
hand we previously talked about $I$-tethered $\Sigma_{h+1}^1$ with the
tether implicit.  By the above, we can regard
$\TS_{h+1}(\Sigma_g^b,I)$ as being the simplicial complex whose $k$-simplices are
collections $\{\iota_0,\ldots,\iota_k\}$ of isotopy classes of orientation-preserving 
$I$-tethered $\tau(\Sigma_h^1,\Sigma_1^1)$ in $\Sigma_g^b$ that can be realized such
that their images are disjoint.

We now define an auxiliary space.  Let $X$ be the simplicial complex
whose $k$-simplices are collections $\{\iota_0,\ldots,\iota_k\}$ of 
isotopy classes of orientation-preserving $I$-tethered $\tau(\Sigma_{h}^1,\Sigma_1^1)$
in $\Sigma_g^b$ that can be realized such that the following hold for all distinct
$0 \leq i, j \leq k$:
\setlength{\parskip}{0pt}
\begin{compactitem}
\item Either $\iota_i|_{\tau(\Sigma_{h}^1)} = \iota_j|_{\tau(\Sigma_{h}^1)}$, or the images
under $\iota_i$ and $\iota_j$ of $\tau(\Sigma_{h}^1)$ are disjoint.
\item If $\iota_i|_{\tau(\Sigma_{h}^1)} = \iota_j|_{\tau(\Sigma_{h}^1)}$, then the 
images under $\iota_i$ and $\iota_j$ of $\Sigma_1^1$ together with the attaching tether
are disjoint except for the initial point of the attaching tether.
\item If the images under $\iota_i$ and $\iota_j$ of $\tau(\Sigma_{h}^1)$ are disjoint, then
the images under $\iota_i$ and $\iota_j$ of $\tau(\Sigma_{h}^1,\Sigma_1^1)$ are disjoint.
\end{compactitem}
For instance, here is a $3$-simplex of $X$ for $g=9$ and $b=1$ and $h=2$:\setlength{\parskip}{\baselineskip}\\
\centerline{\psfig{file=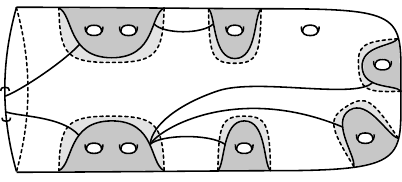,scale=100}}
We have $\TS_{h+1}(\Sigma_g^b,I) \subset X$.  The next claim says that $X$ enjoys
the connectivity property we are trying to prove for $\TS_{h+1}(\Sigma_g^b,I)$.

\begin{claim}
$X$ is $\frac{g-(2h+3)}{h+2}$-connected.
\end{claim}
\begin{proof}[Proof of claim]
Let $\psi\colon X \rightarrow \TS_{h}(\Sigma_g^b,I)$ be the map that takes
a vertex $\iota\colon \tau(\Sigma_{h}^1,\Sigma_1^1) \rightarrow \Sigma_g^b$ of
$X$ to the vertex $\iota|_{\tau(\Sigma_{h}^1)}\colon \tau(\Sigma_{h}^1) \rightarrow \Sigma_g^b$ 
of $\TS_{h}(\Sigma_g^b,I)$.
We will prove that the map $\psi\colon X \rightarrow \TS_{h}(\Sigma_g^b,I)$ satisfies
the conditions of Corollary \ref{corollary:fibers} for
$n = \frac{g-(2h+3)}{h+2}$.  Once we have done this, Corollary \ref{corollary:fibers}
will show that $X$ is $n$-connected, as desired.

The first condition is that $\TS_{h}(\Sigma_g^b,I)$ is $n$-connected.  In
fact, our inductive hypothesis says that it is $\frac{g-(2h+1)}{h+1}$-connected,
which is even stronger.

The second condition says that all $(n+1)$-simplices of $\TS_{h}(\Sigma_g^b,I)$ are
in the image of $\psi$.  The map $\psi$ is $\Mod(\Sigma_g^b)$-equivariant, and
by the change of coordinates principle from \cite[\S 1.3.2]{FarbMargalitPrimer} the
actions of $\Mod(\Sigma_g^b)$ on $\TS_{h+1}(\Sigma_g^b,I)$ and $\TS_{h}(\Sigma_g^b,I)$
are transitive on $k$-simplices for all $k$.  To prove the second condition, therefore,
it is enough to show that $\TS_{h+1}(\Sigma_g^b,I) \subset X$ contains an $(n+1)$-simplex.
Such a simplex contains $(n+2)$ disjoint copies of $\tau(\Sigma_{h}^1,\Sigma_1^1)$.  Since
\begin{align*}
\left(n+2\right)\left(h+1\right) &= \left(\frac{g-\left(2h+3\right)}{h+2}+2\right)\left(h+1\right) 
= \left(\frac{g-\left(2h+3\right)}{h+2}\right)\left(h+1\right) + 2\left(h+1\right) \\
&< \left(g-\left(2h+3\right)\right) + 2\left(h+1\right) 
= g-1 < g,
\end{align*}
there is enough room on $\Sigma_g^b$ to find these $(n+2)$ disjoint copies of 
$\tau(\Sigma_{h}^1,\Sigma_1^1)$.

The final condition says that for all simplices $\sigma$ of $\TS_{h}(\Sigma_g^b,I)$
whose dimension is at most $n$ and all faces $\sigma'$ of $\sigma$, the space
$\Fib_{\psi}(\sigma',\sigma)$ defined right before Lemma \ref{lemma:fiberlemma}
is $n$-connected.  Recall that $\Fib_{\psi}(\sigma',\sigma)$ is the subcomplex
of $X$ consisting of all simplices $\eta'$ of $X$ with the following properties:
\setlength{\parskip}{0pt}
\begin{compactitem}
\item[(i)] $\psi(\eta')$ is a face of $\sigma'$, and
\item[(ii)] there exists a simplex $\eta$ of $X$ such that $\eta'$ is a face of
$\eta$ and $\psi(\eta) = \sigma$.
\end{compactitem}
Write
\[\sigma' = \{\iota_0,\ldots,\iota_{m'}\} \quad \text{and} \quad \sigma = \{\iota_0,\ldots,\iota_{m'},\ldots,\iota_m\},\]
so $0 \leq m' \leq m \leq n$.  We will illustrate all our constructions here with the following running example,
where $\sigma' = \{\iota_0,\iota_1\}$ and $\sigma = \{\iota_0,\iota_1,\iota_2\}$:\\
\centerline{\psfig{file=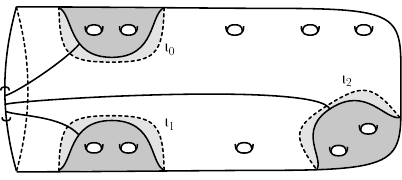,scale=100}}
The left side of the following depicts a $2$-simplex of $X$ lying in $\Fib_{\psi}(\sigma',\sigma)$ and
the right side depicts a $2$-simplex of $X$ that does not lie in $\Fib_{\psi}(\sigma',\sigma)$:\\
\centerline{\psfig{file=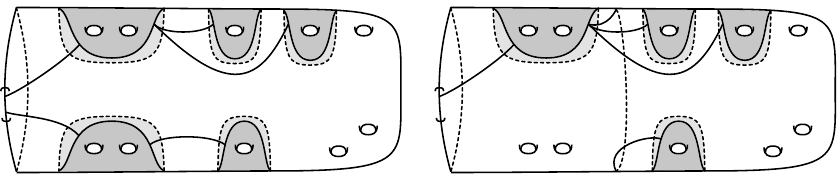,scale=100}}
The issue with the simplex $\eta'$ on the right is that there is not enough genus remaining on the surface\setlength{\parskip}{\baselineskip}
to find a simplex $\eta$ of $X$ satisfying Condition (ii) above.  For the simplex $\eta'$ on the left, the
desired simplex $\eta$ is as follows:\\
\centerline{\psfig{file=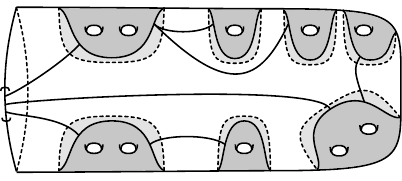,scale=100}}

To understand the connectivity of $\Fib_{\psi}(\sigma',\sigma)$,
we must relate it to a complex we already understand.
Let $\Sigma$ be the surface obtained by first removing the interior of
\[\iota_0\left(\Sigma_{h}^1\right) \cup \cdots \cup \iota_m\left(\Sigma_{h}^1\right)\]
from $\Sigma_g^b$ and then cutting open the result along the images of the tethers.  In our running example,
$\Sigma$ is obtained as follows:\\
\centerline{\psfig{file=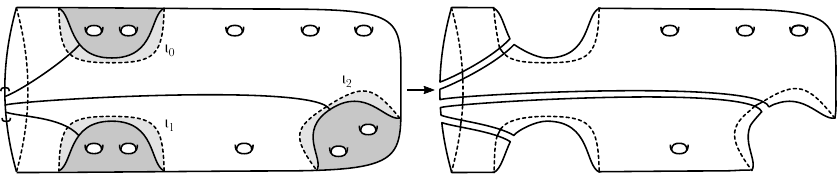,scale=100}}
We thus have $\Sigma \cong \Sigma_{g-(m+1)h}^{b}$.  For $0 \leq i \leq m'$, let
$J_i \subset \partial \Sigma$ be an open interval in 
$\iota_i(\partial \Sigma_{h}^1)$ containing the image of the point
on $\partial \Sigma_h^1$ to which the attaching tether is attached when forming 
$\tau(\Sigma_h^1,\Sigma_1^1)$.  Set $J = J_1 \cup \cdots \cup J_{m'}$.

The complex $\Fib_{\psi}(\sigma',\sigma)$ is isomorphic to a subcomplex $\TS_{1}'(\Sigma,J)$ of $\TS_1(\Sigma,J)$.
In our running example, the simplex of $\Fib_{\psi}(\sigma',\sigma)$ on the left hand side
corresponds to the simplex of $\TS_1(\Sigma,J)$ on the right hand side:\\
\centerline{\psfig{file=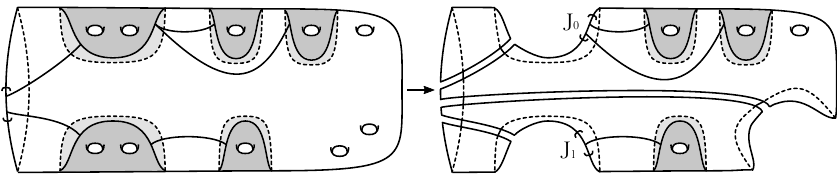,scale=100}}
The different tethers in a simplex of 
$\Fib_{\psi}(\sigma',\sigma) \subset X$ that meet
at a point of $\iota_i(\partial \Sigma_{h}^1)$ are 
``spread out'' in $J_i$ so as to be disjoint.  The reason that $\Fib_{\psi}(\sigma',\sigma)$ is only
isomorphic to a subcomplex $\TS_{1}'(\Sigma,J)$ of $\TS_1(\Sigma,J)$ and not the whole thing
is that it only corresponds to simplices where there is enough genus remaining to ensure Condition (ii)
above holds.

Recall that $m'$ and $m$ satisfy $0 \leq m' \leq m \leq n$.  As we noted in the first paragraph, the connectivity of $\TS_1(\Sigma,J)$
is at least 
\[\frac{g-(m+1)h-3}{2} \geq \frac{g-(n+1)h-3}{2}.\]
To prove that connectivity of the subcomplex $\TS_{1}'(\Sigma,J)$ is at least
$n = \frac{g-(2h+3)}{h+2}$, it is enough to prove the following two
facts:
\setlength{\parskip}{0pt}
\begin{compactitem}
\item $\frac{g-(n+1)h-3}{2} \geq n$, and
\item $\TS_{1}'(\Sigma,J)$ contains the $(n+1)$-skeleton of $\TS_1(\Sigma,J)$.
\end{compactitem}
For the first fact, we calculate as follows: \setlength{\parskip}{\baselineskip}
\begin{align*}
\frac{g-(n+1)h-3}{2} &=
\frac{1}{2}\left(g-\left(\frac{g-(2h+3)}{h+2}+1\right)h-3\right) 
= \frac{g+h^2/2-h-3}{h+2} 
\geq \frac{g-2h-3}{h+2}.
\end{align*}
Here the final inequality follows from the inequality 
$h^2/2-h \geq -2h$, which holds for $h \geq 0$.  

For the second fact, consider
a simplex $\{j_0,\ldots,j_{\ell}\}$ of the $(n+1)$-skeleton of $\TS_1(\Sigma,J)$.
Let $\Sigma'$ be the surface obtained by first removing the interior of
\[j_0\left(\Sigma_{1}^1\right) \cup \cdots \cup j_{\ell}\left(\Sigma_{1}^1\right)\]
from $\Sigma \cong \Sigma_{g-(m+1)h}^{b}$ and then cutting open the result along the images of the tethers.
It follows that
\[\Sigma' \cong \Sigma_{g-(m+1)h-(\ell+1)}^b.\]
In the worst case where the corresponding simplex $\eta'$ of $X$ maps to a vertex of $\sigma'$, 
we need at least $m$ genus remaining to complete $\eta'$ to a simplex mapping to $\sigma$.  In other
words, what we must prove is that
\[g-(m+1)h-(\ell+1) \geq m.\]
Since our simplex lies in the $(n+1)$-skeleton, we have $\ell \leq n+1$.  Also, $m \leq n$.  It follows
that it is enough to prove that
\[g-(n+1)h-(n+2) \geq n.\]
Rearranging, this, we get
\[\frac{g-h-2}{h+2} \geq n.\]
This follows from the fact that $n = \frac{g-(2h+3)}{h+2}$.
\end{proof}

We now use this to prove the desired connectivity property for
$\TS_{h+1}(\Sigma_g^b,I)$.

\begin{claim}
$\TS_{h+1}(\Sigma_g^b,I)$ is $\frac{g-(2h+3)}{h+2}$-connected.
\end{claim}
\begin{proof}[Proof of claim]
We will prove that $\TS_{h+1}(\Sigma_g^b,I)$ is $n$-connected for
$-1 \leq n \leq \frac{g-(2h+3)}{h+2}$ by induction on $n$.  The base
case $n=-1$ simply asserts that $\TS_{h+1}(\Sigma_g^b,I)$ is nonempty when 
$\frac{g-(2h+3)}{h+2} \geq -1$.  This condition is equivalent to
$g \geq h+1$, in which case $\TS_{h+1}(\Sigma_g^b,I) \neq \emptyset$ is obvious.

Assume now that $0 \leq n \leq \frac{g-(2h+3)}{h+2}$ and that for all surfaces
$\Sigma_{g'}^{b'}$ and all finite disjoint unions of open intervals 
$I' \subset \partial \Sigma_{g'}^{b'}$, the space
$\TS_{h+1}(\Sigma_{g'}^{b'},I')$ is $n'$-connected for 
$n' = \min\{n-1, \frac{g'-(2h+3)}{h+2}\}$.
We must prove that $Y:=\TS_{h+1}(\Sigma_g^b,I)$ is $n$-connected.

We know that $X$ is $n$-connected, so to prove that its subcomplex $Y$ is $n$-connected
it is enough to prove that the pair $(X,Y)$ is $(n+1)$-connected.
We will do this using Proposition \ref{proposition:avoidbad}.  For this, we 
must identify a set $\cB$ of ``bad simplices'' of $X$ and verify the three 
hypotheses of the proposition.  Define $\cB$
to be the set of all simplices $\sigma$ of $X$ such that for all vertices $v$ of $\sigma$,
there exists another vertex $v'$ of $\sigma$ such that 
the edge $\{v,v'\}$ of $\sigma$ does not lie in $Y=\TS_{h+1}(\Sigma_g^b,I)$.

We now verify the hypotheses of Proposition \ref{proposition:avoidbad}.  The first two are
easy:
\setlength{\parskip}{0pt}
\begin{compactitem}
\item (i) says that a simplex of $X$ lies in $Y = \TS_{h+1}(\Sigma_g^b,I)$ if and only if none of its faces
lie in $\cB$, which is obvious.
\item (ii) says that if $\sigma_1,\sigma_2 \in \cB$ are such that $\sigma_1 \cup \sigma_2$ is a simplex
of $X$, then $\sigma_1 \cup \sigma_2 \in \cB$, which again is obvious.
\end{compactitem}
The only thing left to check is (iii), which says that for all $k$-dimensional $\sigma \in \cB$, the
complex $G(X,\sigma,\cB)$ has connectivity at least \setlength{\parskip}{\baselineskip}
$(n+1) - k - 1 = n - k$.

Write $\sigma = \{\iota_0,\ldots,\iota_{k}\}$.
Let $\Sigma'$ be the surface obtained by first removing the interiors of
\[\iota_0\left(\Sigma_{h}^1 \sqcup \Sigma_1^1\right) \cup \cdots \cup \iota_{k}\left(\Sigma_{h}^1 \sqcup \Sigma_1^1\right)\]
from $\Sigma_g^b$ and then cutting open the result along the images of the free and
attaching tethers:\\
\centerline{\psfig{file=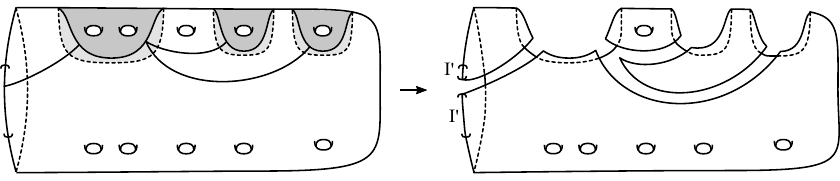,scale=100}}
The surface $\Sigma'$ is connected, and when the surface is cut open along the free and
attaching tethers the open set 
$I \subset \partial \Sigma_g^b$ is divided into a finer collection $I'$ of open 
segments (as in the above example).  Examining its definition in \S \ref{section:linkarguments}, we see that
\[G(X,\sigma,\cB) \cong \TS_{h+1}(\Sigma',I').\]
We must prove that $\TS_{h+1}(\Sigma',I')$ is $(n-k)$-connected.  Let 
$g'$ be the genus of $\Sigma'$.  Since $k \geq 1$, we have $n-k < n$, so
our inductive hypothesis will say that $\TS_{h+1}(\Sigma',I')$ is $(n-k)$-connected
if we can prove that $n-k \leq \frac{g'-(2h+3)}{h+2}$.

This requires estimating $g'$.
The most naive such estimate of $g'$ is
\[g' \geq g-(k+1)(h+1).\]
This is a poor estimate since it does not use the fact that $\sigma \in \cB$,
which implies that every genus $h$ surface contributing to this
estimate is at least double-counted.  Taking this into account, we see that in fact
\[g' \geq g-\left(\frac{k+1}{2}\right)h - \left(k+1\right)
= g-\left(\frac{k+1}{2}\right)\left(h+2\right)\]
This implies that
\begin{align*}
\frac{g'-(2h+3)}{h+2} &\geq
\frac{g-(2h+3)}{h+2} - \frac{\left(\frac{k+1}{2}\right)\left(h+2\right)}{h+2} 
\geq n - \frac{k+1}{2} \geq n-k,
\end{align*}
where the final inequality uses the fact that $k \geq 1$.
\end{proof}

This completes the proof of Theorem \ref{maintheorem:subsurfacescon}.
\end{proof}

\section{Stability for surfaces with one boundary component}
\label{section:stableone}

In this section, we prove Theorems \ref{maintheorem:stable} and \ref{maintheorem:nonabelianstable}.  
The outline is as follows.  In \S \ref{section:stabilitymachine},
we discuss the homological stability machine.  In
\S \ref{section:destabilizemarking} -- \S \ref{section:vanishsurfaces}
we prove a number of preliminary results needed to apply this machine.  Our proof
of Theorem \ref{maintheorem:vanishtetheredconone} (and its nonabelian analogue) is
in \S \ref{section:vanishsubsurfacescon}.  Finally, in \S \ref{section:proofone} we
prove Theorems \ref{maintheorem:stable} and \ref{maintheorem:nonabelianstable}.

\subsection{The stability machine}
\label{section:stabilitymachine}

\setlength{\parskip}{\baselineskip}
We now introduce the standard homological stability machine.  This is discussed in many
places, but the account in \cite[\S 1]{HatcherVogtmannTethers} is particularly convenient
for our purposes.  We remark that the results in this paper could also be proved using
the framework of \cite{Krannich} (which generalizes \cite{RandalWilliamsWahl}), but since
it would not simplify our proofs we decided not to use that framework.

\p{Semisimplicial sets}
The natural setting for the machine is that of semisimplicial sets, whose definition
we now briefly recall.  For more details, see \cite{FriedmanSimplicial}, which
calls them $\Delta$-sets.
Let $\Delta$ be the category with
objects the sets $[k] = \{0,\ldots,k\}$ for $k \geq 0$ and whose morphisms
$[k] \rightarrow [\ell]$ are the strictly increasing functions.  A {\em semisimplicial
set} is a contravariant functor $X$ from $\Delta$ to the category of sets.
The {\em $k$-simplices} of $X$ are the image $X_k$ of $[k] \in \Delta$.
The maps $X_{\ell} \rightarrow X_k$ corresponding to the $\Delta$-morphisms
$[k] \rightarrow [\ell]$ are called the {\em face maps}.

\p{Geometric properties}
A semisimplicial set $X$ has a geometric realization $|X|$
obtained by taking standard $k$-simplices
for each element of $X_k$ and then gluing these simplices together using the
face maps.  Whenever
we talk about topological properties of a semisimplicial set, we are referring
to its geometric realization.  An action of a group $G$ on a semisimplicial set $X$
consists of actions of $G$ on each $X_n$ that commute with the face maps.  This
induces an action of $G$ on $|X|$.

\p{The machine}
The version of the homological stability machine we need is as follows.  In it,
the indexing is chosen such that the complex $X_1$ upon which $G_1$ acts is
connected.

\begin{theorem}
\label{theorem:stabilitymachine}
Let
\[G_0 \subset G_1 \subset G_2 \subset \cdots\]
be an increasing sequence of groups.  For each $n \geq 1$, let $X_n$ be
a semisimplicial set upon which $G_n$ acts.  Assume for some
$c \geq 2$ that the following hold:
\setlength{\parskip}{0pt}
\begin{compactenum}
\item The space $X_n$ is $(n-1)/c$-connected.
\item For all $0 \leq i < n$, the group $G_{n-i-1}$ is the
$G_n$-stabilizer of some $i$-simplex of $X_n$.
\item For all $0 \leq i < n$, the group $G_n$ acts transitively
on the $i$-simplices of $X_n$.
\item For all $n \geq c+1$ and all $1$-simplices $e$ of $X_n$ 
whose boundary consists of vertices
$v$ and $v'$, there exists some $\lambda \in G_n$ such that $\lambda(v) = v'$ and such that
$\lambda$ commutes with all elements of $(G_n)_e$.
\end{compactenum}
Then for $k \geq 1$ the map $\HH_k(G_{n-1}) \rightarrow \HH_k(G_n)$ is an isomorphism for
$n \geq ck+1$ and a surjection for $n = ck$.
\end{theorem}
\begin{proof}
This is proved exactly like \cite[Theorem 1.1]{HatcherVogtmannTethers}.  
\end{proof}

\subsection{Stabilizing and destabilizing markings}
\label{section:destabilizemarking}

We next discuss the process of stabilizing and destabilizing markings.  Recall that $A$ is a fixed finitely generated abelian group and $\Lambda$ is a fixed finite group.

\p{Stabilizing and destabilizing, abelian}
If $\mu$ is an $A$-homology marking on $\Sigma_{g}^1$ and
$\Sigma_{g}^1 \hookrightarrow \Sigma_{g'}^1$ is an embedding,
then we can define the stabilization to $\Sigma_{g'}^1$ of $\mu$ just like we did
in the introduction.  Namely, $\HH_1(\Sigma_{g})$ can be identified with
a symplectic subspace of $\HH_1(\Sigma_{g'})$, so
\[\HH_1(\Sigma_{g'}) = \HH_1(\Sigma_{g}^1) \oplus \HH_1(\Sigma_g^1)^{\perp},\]
where the $\perp$ is with respect to the algebraic intersection pairing.  Define 
the {\em stabilization} $\mu'\colon \HH_1(\Sigma_{g'}^1) \rightarrow A$
of $\mu$ to be the composition
\[\HH_1(\Sigma_{g'}^1) = \HH_1(\Sigma_{g}^1) \oplus \HH_1(\Sigma_g^1)^{\perp} \longrightarrow \HH_1(\Sigma_{g}^1) \stackrel{\mu}{\longrightarrow} A,\]
where the first arrow is the orthogonal projection.  We will also say that $\mu$ is a {\em destabilization} of $\mu'$.

\p{Stabilizing and destabilizing, non-abelian}
Now let $\mu$ be a $\Lambda$-marking on $\Sigma_g^1$ and
$\Sigma_{g}^1 \hookrightarrow \Sigma_{g'}^1$ be an embedding.
Defining the stabilization of $\mu$ to $\Sigma_{g'}^1$ is subtle
since there is not a canonical\footnote{In the introduction, we made a very specific
choice when we stabilized a $\Lambda$-marking on $\Sigma_g^1$ to $\Sigma_{g+1}^1$.}
way to stabilize.  We thus need to make some auxiliary choices.

Let $\ast \in \partial \Sigma_g^1$ and $\ast' \in \partial \Sigma_{g'}^1$ be the basepoints.
Let $S$ be a subsurface of $\Sigma_{g'}^1 \setminus \Interior(\Sigma_g^1)$ with
$S \cong \Sigma_{g'-g}^1$.  Choose a basepoint $\ast'' \in \partial S$, and let
$\lambda$ and $\eta$ be embedded paths in $\Sigma_{g'}^1 \setminus \Interior(\Sigma_g^1 \cup S)$ connecting
$\ast$ to $\ast'$ and $\ast''$, respectively.  Assume that $\lambda$ and $\eta$ are disjoint aside from their
initial points:\\
\centerline{\psfig{file=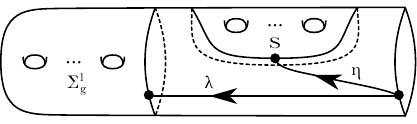,scale=100}}
The paths $\lambda$ and $\eta$ induce injective homomorphisms
\[\pi_1(\Sigma_g^1,\ast) \hookrightarrow \pi_1(\Sigma_{g'}^1,\ast') \quad \text{and} \quad \pi_1(S,\ast'') \hookrightarrow \pi_1(\Sigma_{g'}^1,\ast')\]
taking $x \in \pi_1(\Sigma_g^1,\ast)$ to $\lambda \cdot x \cdot \lambda^{-1} \in \pi_1(\Sigma_{g'}^1,\ast')$ and
$y \in \pi_1(S,\ast'')$ to $\eta \cdot y \cdot \eta^{-1} \in \pi_1(\Sigma_{g'}^1,\ast')$.  Identifying
$\pi_1(\Sigma_g^1,\ast)$ and $\pi_1(S,\ast'')$ with the corresponding subgroups of $\pi_1(\Sigma_{g'}^1,\ast')$, we
have a free product decomposition
\[\pi_1(\Sigma_{g'}^1,\ast') = \pi_1(\Sigma_g^1,\ast) \star \pi_1(S,\ast'').\]
Define $\mu'\colon \pi_1(\Sigma_{g'}^1,\ast') \rightarrow \Lambda$
to be the composition
\[\pi_1(\Sigma_{g'}^1,\ast') = \pi_1(\Sigma_g^1,\ast) \star \pi_1(S,\ast'') \longrightarrow \pi_1(\Sigma_g^1,\ast) \stackrel{\mu}{\longrightarrow} \Lambda,\]
where the first arrow quotients out by the normal closure of $\pi_1(S,\ast'')$.  

A different choice of $\eta$ would change the subgroup $\pi_1(S,\ast'')$ of $\pi_1(\Sigma_{g'}^1,\ast')$ to a conjugate
subgroup, so would not change $\mu'$.  It follows that $\mu'$ only depends on the pair $(S,\lambda)$, and
we will call $\mu'$ the {\em $(S,\lambda)$-stabilization} of $\mu$ to $\Sigma_{g'}^1$.
If we do not
want to specify $(S,\lambda)$, we will just say that $\mu'$ is a stabilization of $\mu$, but
be warned that different choices of $(S,\lambda)$ will lead to different stabilizations. 
We will also say that $\mu$ is a {\em destabilization} of $\mu'$ with {\em destabilization data $(S,\lambda)$}.

\begin{remark}
The choice of $S$ is more important than the choice of $\lambda$.  Indeed, changing $\lambda$ would have
the effect of conjugating $\mu'$ by an element of $\Lambda$.  This would not affect the associated
partial Torelli group $\Torelli(\Sigma_{g'}^1,\mu')$.
\end{remark}

\p{Maps between partial Torelli groups}
Let $\mu$ be either an $A$-homology marking or a $\Lambda$-marking on $\Sigma_g^1$, let $\Sigma_{g}^1 \hookrightarrow \Sigma_{g'}^1$
be an embedding, and let $\mu'$ be a stabilization of $\mu$ to $\Sigma_{g'}^1$.  The embedding
$\Sigma_{g}^1 \hookrightarrow \Sigma_{g'}^1$ induces an injective map $\Mod(\Sigma_g^1) \rightarrow \Mod(\Sigma_{g'}^1)$
on mapping class groups, and from our definitions it is clear that this restricts to a map
$\Torelli(\Sigma_g^1,\mu) \rightarrow \Torelli(\Sigma_{g'}^1,\mu')$ between the associated
partial Torelli groups.  In fact, we have the following:

\begin{lemma}
\label{lemma:asbigaspossible}
Let $\mu$ be either an $A$-homology marking or a $\Lambda$-marking on $\Sigma_g^1$, let $\Sigma_{g}^1 \hookrightarrow \Sigma_{g'}^1$
be an embedding, and let $\mu'$ be a stabilization of $\mu$ to $\Sigma_{g'}^1$.  
Let $\iota\colon \Mod(\Sigma_g^1) \rightarrow \Mod(\Sigma_{g'}^1)$ be the map induced by $\Sigma_{g}^1 \hookrightarrow \Sigma_{g'}^1$.
We then have that
\[\Torelli(\Sigma_g^1,\mu) = \Set{$\phi \in \Mod(\Sigma_g^1)$}{$\iota(\phi) \in \Torelli(\Sigma_{g'}^1,\mu')$}.\]
\end{lemma}
\begin{proof}
Immediate.
\end{proof}

\p{Vanishing surfaces}
Recall that the rank $\rank(A)$ of the finitely generated abelian group
$A$ is the minimal size of a generating set for $A$.
Consider a subsurface $S$ of $\Sigma_g^1$.  For an $A$-homology marking $\mu$ on $\Sigma_g^1$,
we say that {\em $\mu$ vanishes on $S$} if $\mu$ vanishes on the image of $\HH_1(S)$ in
$\HH_1(\Sigma_g^1)$.  Similarly, for a $\Lambda$-marking $\mu$ we say that {\em $\mu$ vanishes
on $S$} if $\mu(x) = 1$ for all $x \in \pi_1(\Sigma_g^1,\ast)$ that are freely homotopic to
a loop in $S$.  Here $\ast \in \partial \Sigma_g^1$ is our fixed basepoint.

\begin{proposition}
\label{proposition:findvanish}
Consider some $g,h \geq 1$.  The following then hold.
\setlength{\parskip}{0pt}
\begin{compactitem}
\item Let $\mu$ be an $A$-homology marking on $\Sigma_g^1$, and assume that
$g \geq \rank(A) + h$.  Then there exists an embedding
$S \hookrightarrow \Sigma_{g}^1$ with $S \cong \Sigma_h^1$ such that $\mu$ vanishes on $S$.
\item Let $\mu$ be a $\Lambda$-marking on $\Sigma_g^1$, and assume that
$g \geq |\Lambda|+h$.  Then there exists an embedding
$S \hookrightarrow \Sigma_{g}^1$ with $S \cong \Sigma_h^1$ such that $\mu$ vanishes on $S$.
\end{compactitem}
\end{proposition}
\setlength{\parskip}{\baselineskip}
\begin{proof}[Proof of Proposition \ref{proposition:findvanish} for $A$-homology markings]
Consider a symplectic subspace $U$ of $\HH_1(\Sigma_g^1)$, i.e., a subgroup such that
$\HH_1(\Sigma_g^1) = U \oplus U^{\perp}$, where the $\perp$ is with respect to the
algebraic intersection pairing.  Such a $U$ is of the
form $U \cong \Z^{2k}$ for an integer $k \geq 0$ called the {\em genus} of $U$.
Every genus $h$ symplectic subspace $U$ of $\HH_1(\Sigma_g^1)$ can be written
as $U = \HH_1(S)$ for some subsurface $S$ of $\Sigma_g^1$
satisfying $S \cong \Sigma_{h}^1$ (see, e.g., \cite[Lemma 9]{JohnsonConjugacy}).
The proposition is thus equivalent to the purely algebraic Lemma
\ref{lemma:destabilizeabelianalgebra} below.
\end{proof}

\begin{lemma}
\label{lemma:destabilizeabelianalgebra}
Let $V \cong \Z^{2g}$ be a free abelian group equipped with a symplectic form
$\omega(-,-)$ and let $\mu\colon V \rightarrow A$ be a group homomorphism.
Assume that $g \geq \rank(A)+h$ for some $h \geq 1$.
There then exists a genus $h$ symplectic subspace $U$ of $V$ such that $\mu|_{U} = 0$.
\end{lemma}
\begin{proof}
Without loss of generality, $\mu$ is surjective and $A \neq 0$.  Also, increasing
$h$ if necessary we can assume that $g = \rank(A)+h$.  We will prove
the ``dual'' statement that there exists a genus $\rank(A)$ symplectic subspace
$W$ of $V$ such that $\mu|_{W^{\perp}} = 0$.  The desired $U$ is then $U = W^{\perp}$.
The proof will be by induction on $\rank(A)$.  The base case is $\rank(A)=1$, so $A$
is cyclic.  We can factor $\mu$ as
\[V \stackrel{\tmu}{\twoheadrightarrow} \Z \rightarrow A.\]
By definition, $\omega(-,-)$ identifies $V$ with its dual $\Hom(V,\Z)$.
There thus exists some $a \in V$ such that $\tmu(x)=\omega(a,x)$ for all
$x \in V$.  Pick $b \in V$ with $\omega(a,b)=1$
and let $W = \langle a,b \rangle$.  Then $W$ is a genus $1$ symplectic subspace and
\[W^{\perp} \subset \ker(\omega(a,-)) = \ker(\tmu) \subset \ker(\mu),\]
as desired.

Now assume that $\rank(A)>1$ and that the lemma is true for all smaller ranks.  We
can then find a short exact sequence
\[0 \longrightarrow A' \longrightarrow A \stackrel{\phi}{\longrightarrow} A'' \longrightarrow 0\]
such that $0<\rank(A') < \rank(A)$ and $\rank(A'') + \rank(A') = \rank(A)$.
By our inductive hypothesis, there
exists a genus $\rank(A'')$ symplectic subspace $W''$ of $V$ such that
$(\phi \circ \mu)|_{(W'')^{\perp}} = 0$.  Set $V' = (W'')^{\perp}$, so $V'$ is a symplectic
subspace of $V$ and the image of $\mu' := \mu|_{V'}$ lies in $A'$.
Our inductive hypothesis implies that there exists a genus $\rank(A')$
symplectic subspace $W'$ of $V'$ such that $\mu'|_{(W')^{\perp}} = 0$.  Setting
$W = W' \oplus W''$, we have that $W$ is a genus $\rank(A') + \rank(A'') = \rank(A)$ symplectic
subspace of $V$ such that $\mu|_{W^{\perp}} = 0$, as desired.
\end{proof}

\begin{proof}[Proof of Proposition \ref{proposition:findvanish} for $\Lambda$-markings]
The proposition is a small variant
of a result of Dunfield--Thurston \cite[Proposition 6.16]{DunfieldThurston} -- the
only difference is that \cite[Proposition 6.16]{DunfieldThurston} is for closed
surfaces, while we need to deal with $\Sigma_{g}^1$.  However, the exact same
proof works, so we omit the details.
\end{proof}

\p{Deeply destabilizing}
Proposition \ref{proposition:findvanish} has the following corollary.

\begin{corollary}
\label{corollary:destabilizeone}
Consider some $g' \geq 1$.  The following hold.
\setlength{\parskip}{0pt}
\begin{compactitem}
\item Let $\mu'$ be an $A$-homology marking on $\Sigma_{g'}^1$.  Assume that $g' > \rank(A)$, and let $g = \rank(A)$.  Then there exists
an embedding $\Sigma_{g}^1 \hookrightarrow \Sigma_{g'}^1$ and
an $A$-homology marking $\mu$ on $\Sigma_{g}^1$ such that $\mu$ is a destabilization of $\mu'$.
\item Let $\mu'$ be a $\Lambda$-marking on $\Sigma_{g'}^1$.  Assume that $g' > |\Lambda|$, and let $g = |\Lambda|$.  Then there exists
an embedding $\Sigma_{g}^1 \hookrightarrow \Sigma_{g'}^1$ and a $\Lambda$-marking $\mu$ on $\Sigma_g^1$
such that $\mu$ is a destabilization of $\mu'$.
\end{compactitem}
\end{corollary}
\setlength{\parskip}{\baselineskip}
\begin{proof}
The proofs for $A$-homology markings and $\Lambda$-markings are similar, so we
give the details for $\Lambda$-markings (which are slightly more complicated).
Let $\ast' \in \partial \Sigma_{g'}^1$ be the basepoint.  By Proposition \ref{proposition:findvanish}, we can find a subsurface
$S \hookrightarrow \Sigma_{g'}^1$ with $S \cong \Sigma_{g'-g}^1$ such that $\mu'$
vanishes on $S$.  Pick the following:
\setlength{\parskip}{0pt}
\begin{compactitem}
\item An embedding $\Sigma_g^1 \hookrightarrow \Sigma_{g'}^1$ that is disjoint
from $S$, as well as a basepoint $\ast \in \partial \Sigma_g^1$.
\item An embedded path $\lambda$ in $\Sigma_{g'}^1 \setminus \Interior(\Sigma_g^1 \cup S)$ 
connecting $\ast'$ to $\ast$.  
\end{compactitem}
Define $\mu\colon \pi_1(\Sigma_g^1,\ast) \rightarrow \Lambda$ via the formula
\[\mu(x) = \mu'(\lambda \cdot x \cdot \lambda^{-1}) \quad \text{for $x \in \pi_1(\Sigma_g^1,\ast)$}.\]
It is immediate from the definitions that $\mu'$ is the $(S,\lambda)$-stabilization of $\mu$ to $\Sigma_{g'}^1$.
\end{proof}

\subsection{Vanishing surfaces}
\label{section:vanishsurfaces}

This section constructs the semisimplicial sets we need to apply Theorem \ref{theorem:stabilitymachine}
to the partial Torelli groups.

\subsubsection{Vanishing surfaces: definition and basic properties}
\label{section:subsurfacesdef}

We define the complexes separately for $A$-homology markings and $\Lambda$-markings.

\p{Vanishing subsurfaces, abelian}
We start by recalling the definition of the complex of vanishing subsurfaces for a homology marking
from the introduction.  Let $\mu$ be an $A$-homology marking on $\Sigma_g^1$.
Define $\Sur_{h}(\Sigma_g^1,\mu)$ to
be the full subcomplex of $\Sur_{h}(\Sigma_g^1)$ spanned by vertices $\iota\colon \Sigma_{h}^1 \rightarrow \Sigma_g^1$ such that $\mu$ vanishes on $\Sigma_h^1$ in the sense
of \S \ref{section:destabilizemarking}.
The group $\Torelli(\Sigma_g^1,\mu)$ acts on $\Sur_{h}(\Sigma_g^1,\mu)$.  Similarly, if $I \subset \partial \Sigma_g^1$
is a finite disjoint union of open intervals, then define $\TS_{h}(\Sigma_g^1,I,\mu)$ to be the
full subcomplex of $\TS_{h}(\Sigma_g^1,I)$
spanned by vertices $\iota\colon \tau(\Sigma_{h}^1) \rightarrow \Sigma_g^1$ whose restriction to $\Sigma_{h}^1$ is a vertex
of $\Sur_{h}(\Sigma_g^1,\mu)$.
Again, the group $\Torelli(\Sigma_g^1,\mu)$ acts on $\TS_{h}(\Sigma_g^1,I,\mu)$.

\p{Vanishing subsurfaces, nonabelian}
Let $\mu$ be a $\Lambda$-marking on $\Sigma_g^1$.  Define $\Sur_{h}(\Sigma_g^1,\mu)$ to
be the full subcomplex of $\Sur_{h}(\Sigma_g^1)$ spanned by vertices 
$\iota\colon \Sigma_{h}^1 \rightarrow \Sigma_g^1$ such that $\mu$ vanishes
on $\Sigma_h^1$ in the sense of \S \ref{section:destabilizemarking}.
The group $\Torelli(\Sigma_g^1,\mu)$ acts on $\Sur_{h}(\Sigma_g^1,\mu)$.  Similarly, if $I \subset \partial \Sigma_g^1$
is a finite disjoint union of open intervals, then define $\TS_{h}(\Sigma_g^1,I,\mu)$ to be the 
full subcomplex of $\TS_{h}(\Sigma_g^1,I)$
spanned by vertices $\iota\colon \tau(\Sigma_{h}^1) \rightarrow \Sigma_g^1$ whose restriction to $\Sigma_{h}^1$ is a vertex
of $\Sur_{h}(\Sigma_g^1,\mu)$.
Again, the group $\Torelli(\Sigma_g^1,\mu)$ acts on $\TS_{h}(\Sigma_g^1,I,\mu)$.

\p{Semisimplicial}
In the rest of this section, let $\mu$ be either an $A$-homology marking or a $\Lambda$-marking on $\Sigma_g^1$ and let $I \subset \partial \Sigma_g^1$ be a single interval.
We claim then 
that $\TS_{h}(\Sigma_g^1,I,\mu)$ is naturally a semisimplicial set.  The
key point here is that its simplices $\{\iota_0,\ldots,\iota_k\}$
possess a natural ordering based on the order their tethers leave $I$.

\p{Stabilizers}
The $\Mod(\Sigma_g^1)$-stabilizers of simplices of $\Sur_{h}(\Sigma_g^1)$ are poorly behaved.
The issue is that mapping classes can permute their vertices arbitrarily (which is not possible for
$\TS_{h}(\Sigma_g^1,I)$ since mapping classes must preserve the order in which the tethers leave
$I$).
This prevents their stabilizers from being
mapping class groups of subsurfaces.  For $\TS_{h}(\Sigma_g^1,I)$, however,
this issue does not occur, and the $\Mod(\Sigma_g^1)$-stabilizer of a simplex
$\{\iota_0,\ldots,\iota_k\}$ of $\TS_{h}(\Sigma_g^1,I)$ equals $\Mod(\Sigma)$, where
$\Sigma$ is the complement of an open regular neighborhood of
\[\partial \Sigma_g^1 \cup \iota_0\left(\tau\left(\Sigma_{h}^1\right)\right) \cup \cdots \cup \iota_k\left(\tau\left(\Sigma_{h}^1\right)\right).\]
We will call the complement of this open neighborhood the {\em stabilizer subsurface}
of the simplex.  See here, where the stabilizer subsurface is the complement of the shaded region:\\
\centerline{\psfig{file=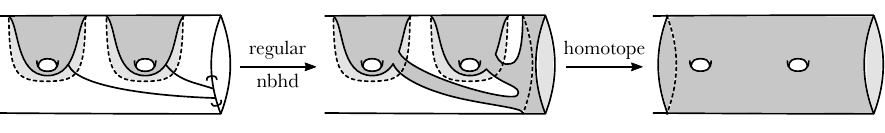,scale=100}}
The $\Torelli(\Sigma_g^1,\mu)$ version of this is the following lemma:

\begin{lemma}
\label{lemma:vanishtetheredstab}
Let $\mu$ be either an $A$-homology marking or a $\Lambda$-marking on $\Sigma_g^1$,
let $I \subset \partial \Sigma_g^1$ be an open interval,
and let $\sigma$ be a $k$-simplex of $\TS_{h}(\Sigma_g^1,I,\mu)$.
Let $\Sigma \cong \Sigma_{g-(k+1)h}^1$ be the stabilizer subsurface of $\sigma$.
Then there exists a marking $\mu_0$ of the same type as $\mu$ (either an $A$-homology marking
or a $\Lambda$-marking) on $\Sigma$ such that $\mu_0$ is a destabilization of $\mu$
and such that the $\Torelli(\Sigma_g^1,\mu)$-stabilizer of $\sigma$ is $\Torelli(\Sigma,\mu_0)$.
\end{lemma}
\begin{proof}
The proofs for $A$-homology markings and $\Lambda$-markings are similar, so we will give the details for $\Lambda$-markings.
Let $\ast \in \partial \Sigma_g^1$ and $\ast_0 \in \partial \Sigma$ be basepoints.  
Write $\sigma = \{\iota_0,\ldots,\iota_k\}$.  For $0 \leq i \leq k$, let $S_i = \iota_i(\Sigma_h^1)$.
Let $S$ be a subsurface of $\Sigma_g^1 \setminus \Interior(\Sigma)$ such that $S$ contains each $S_i$ and
$S \cong \Sigma_{(k+1)h}^1$, and let $\lambda$ be an embedded path in $\Sigma_g^1 \setminus \Interior(\Sigma \cup S)$
connecting $\ast$ to $\ast_0$:\\
\centerline{\psfig{file=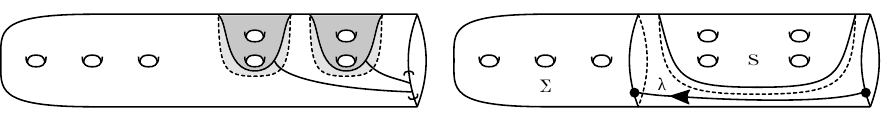,scale=100}}
Define $\mu_0\colon \pi_1(\Sigma,\ast_0) \rightarrow \Lambda$ via the formula
\[\mu_0(x) = \mu(\lambda \cdot x \cdot \lambda^{-1}) \quad \text{for $x \in \pi_1(\Sigma,\ast_0)$}.\]
It follows from the definitions that $\mu$ is the $(S,\lambda)$-stabilization of $\mu_0$.  Since the $\Mod(\Sigma_g^1)$-stabilizer of $\sigma$ is
$\Mod(\Sigma)$, it follows that the $\Torelli(\Sigma_g^1,\mu)$-stabilizer of $\sigma$ is 
$\Mod(\Sigma) \cap \Torelli(\Sigma_g^1,\mu)$, which by Lemma \ref{lemma:asbigaspossible} is $\Torelli(\Sigma,\mu_0)$.
\end{proof}

\subsubsection{Vanishing surfaces: high connectivity}
\label{section:vanishsubsurfacescon}

The following theorem subsumes Theorem \ref{maintheorem:vanishtetheredconone}.  

\begin{theorem}
\label{theorem:vanishsurfacescon}
Fix $g \geq h \geq 1$ and let $I \subset \partial \Sigma_g^1$ be a finite disjoint union
of open intervals.  Then the following hold.
\setlength{\parskip}{0pt}
\begin{compactitem}
\item Let $\mu$ be an $A$-homology marking on $\Sigma_g^1$.  The complexes
$\Sur_h(\Sigma_g^1,\mu)$ and $\TS_h(\Sigma_g^1,I,\mu)$ are both $\frac{g-(2\rank(A)+2h+1)}{\rank(A)+h+1}$-connected.
\item Let $\mu$ be a $\Lambda$-marking on $\Sigma_g^1$.  The complexes
$\Sur_h(\Sigma_g^1,\mu)$ and $\TS_h(\Sigma_g^1,I,\mu)$ are both $\frac{g-(2|\Lambda|+2h+1)}{|\Lambda|+h+1}$-connected.
\end{compactitem}
\setlength{\parskip}{\baselineskip}
\end{theorem}
\begin{proof}
The proofs for $A$-homology markings and $\Lambda$-markings are identical, so we will give the details for
$\Lambda$-markings.  Also, the proofs that $\Sur_{h}(\Sigma_g^b,\mu)$ and $\TS_{h}(\Sigma_g^b,I,\mu)$ are
$\frac{g-(2|\Lambda|+2h+1)}{|\Lambda|+h+1}$-connected are similar.  Keeping track of the tethers
introduces a few complications, so we will give the details for
$\TS_{h}(\Sigma_g^b,I,\mu)$ and leave $\Sur_h(\Sigma_g^b,\mu)$ to the reader.

We start by defining an auxiliary space.  Let $X$ be the simplicial complex
whose vertices are the union of the vertices of the complexes 
$\TS_{h}(\Sigma_g^1,I,\mu)$ and $\TS_{|\Lambda|+h}(\Sigma_g^1,I)$
and whose simplices are collections $\{\iota_0,\ldots,\iota_k\}$ of vertices
that can be isotoped such that their images are disjoint.  Both $\TS_h(\Sigma_g^1,I,\mu)$
and $\TS_{|\Lambda|+h}(\Sigma_g^1,I)$ are thus full subcomplexes of $X$.

We now prove that $X$ enjoys the connectivity property we are trying to prove
for $\TS_{h}(\Sigma_g^b,I,\mu)$.

\begin{claim}
The space $X$ is $\frac{g-(2|\Lambda|+2h+1)}{|\Lambda|+h+1}$-connected.
\end{claim}
\begin{proof}[Proof of claim]
Set $n = \frac{g-(2|\Lambda|+2h+1)}{|\Lambda|+h+1}$ and
$Y = \TS_{|\Lambda|+h}(\Sigma_g^1,I)$ and $Y' = \TS_{h}(\Sigma_g^1,I,\mu)$.  Theorem
\ref{maintheorem:subsurfacescon} says that $Y$ is $n$-connected, so it is enough
to prove that the pair $(X,Y)$ is $n$-connected.  To do this, we will apply
Corollary \ref{corollary:avoidsubcomplex}.  This requires showing the following.
Let $\sigma$ be a $k$-dimensional simplex of $Y' = \TS_{h}(\Sigma_g^1,I,\mu)$ and
let $L$ be the link of $\sigma$ in $X$.  Then we must show that
$L \cap Y$ is $(n-k-1)$-connected.

Write $\sigma = \{\iota_0,\ldots,\iota_{k}\}$.  Let $\Sigma'$ be the surface obtained by first removing the interiors of
\[\iota_0\left(\Sigma_{h}^1\right) \cup \cdots \cup \iota_{k}\left(\Sigma_{h}^1\right)\]
from $\Sigma_g^1$ and then cutting open the result along the images of the tethers:\\
\centerline{\psfig{file=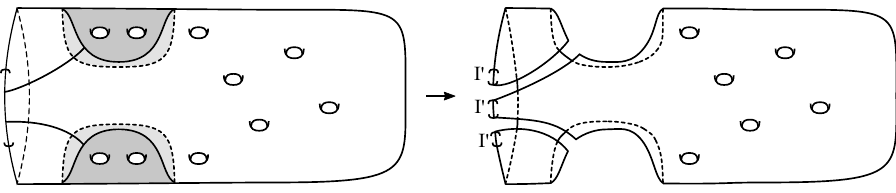,scale=100}}
The surface $\Sigma'$ is connected, and when the surface is cut open along the
tethers the open set
$I \subset \partial \Sigma_g^1$ is divided into a finer collection $I'$ of open
segments (as in the above example).
We then have
\[L \cap Y \cong \TS_{|\Lambda|+h}(\Sigma',I'),\]
so we must prove that $\TS_{|\Lambda|+h}(\Sigma',I')$ is $(n-k-1)$-connected.
Letting $g'$ be the genus of $\Sigma'$, Theorem \ref{maintheorem:subsurfacescon} says that
$\TS_{|\Lambda|+h}(\Sigma',I')$ is $\frac{g'-(2|\Lambda|+2h+1)}{|\Lambda|+h+1}$-connected,
so what we must prove is that
\[n-k-1 \leq \frac{g'-(2|\Lambda|+2h+1)}{|\Lambda|+h+1}.\]
Examining the construction of $\Sigma'$, we see that $g' = g - (k+1)h$.
We now calculate that
\begin{align*}
\frac{g'-(2|\Lambda|+2h+1)}{|\Lambda|+h+1}
&= \frac{g-(2|\Lambda|+2h+1)}{|\Lambda|+h+1} - \frac{(k+1)h}{|\Lambda|+h+1}
\geq n - (k+1).\qedhere
\end{align*}
\end{proof}

To complete the proof, it is enough to construct a retraction
$r\colon X \rightarrow \TS_{h}(\Sigma_g^1,I,\mu)$.  For
a vertex $\iota$ of $X$, we define $r(\iota)$ as follows.  If $\iota$ is a vertex
of $\TS_{h}(\Sigma_g^1,I,\mu)$, then $r(\iota) = \iota$.  If instead
$\iota$ is a vertex of $\TS_{|\Lambda|+h}(\Sigma_g^1,I)$, then
Proposition \ref{proposition:findvanish} implies that we can find
a subsurface $\Sigma_h^1 \hookrightarrow \iota(\Sigma_{|\Lambda|+h})$
such that $\mu$ vanishes on $\Sigma_h^1$.  
Define $r(\iota)$ to be the vertex of $\TS_h(\Sigma_g^1,I,\mu)$ obtained by 
adjoining the tether of $\iota$ and an arbitrary arc in 
$\iota(\Sigma_{|\Lambda|+h}^1)$ connecting the boundary to
$\Sigma_h^1$:\\
\centerline{\psfig{file=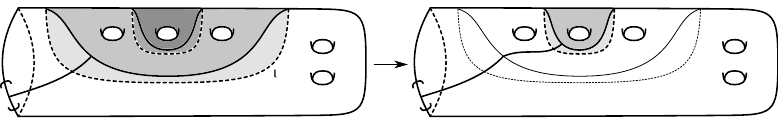,scale=100}}
Of course, $r(\iota)$ depends on various choices, but we simply make an arbitrary choice.
It is clear that this extends over the simplices of $X$ to give a retract
$r\colon X \rightarrow \TS_{h}(\Sigma_g^1,I,\mu)$.
\end{proof}

\subsubsection{Vanishing surfaces: transitivity}
\label{section:vanishsubsurfacestran}

The last fact about the complex of vanishing surfaces we will need is as follows.

\begin{lemma}
\label{lemma:vanishtetheredtran}
Fix $g \geq h \geq 1$ and let $I \subset \partial \Sigma_g^1$ be an open interval.
Then the following hold.
\setlength{\parskip}{0pt}
\begin{compactitem}
\item Let $\mu$ be an $A$-homology marking on $\Sigma_g^1$.  The group
$\Torelli(\Sigma_g^1,\mu)$ acts transitively on the $k$-simplices
of $\TS_h(\Sigma_g^1,I,\mu)$ if $g \geq 2h+2\rank(A)+1+k h$.
\item Let $\mu$ be a $\Lambda$-marking on $\Sigma_g^1$.  The group
$\Torelli(\Sigma_g^1,\mu)$ acts transitively on the $k$-simplices
of $\TS_h(\Sigma_g^1,I,\mu)$ if $g \geq 2h+2|\Lambda|+1+k h$. 
\end{compactitem}
\setlength{\parskip}{\baselineskip}
\end{lemma}
\begin{proof}
The proofs for $A$-homology markings and $\Lambda$-markings are identical, so
we will give the details for $\Lambda$-markings.
The proof will be by induction on $k$.  We start with the base case $k=0$.

\begin{claim}
$\Torelli(\Sigma_g^1,\mu)$ acts transitively on the $0$-simplices of $\TS_{h}(\Sigma_g^1,I,\mu)$ if 
$g \geq 2 h+2|\Lambda|+1$.
\end{claim}
\begin{proof}[Proof of claim]
In this case, Theorem \ref{theorem:vanishsurfacescon} says that $\TS_{h}(\Sigma_g^1,I,\mu)$ is connected, so it is enough
to prove that if $\iota_0$ and $\iota_1$ are vertices of $\TS_{h}(\Sigma_g^1,I,\mu)$
that are connected by an edge, then there exists some $f \in \Torelli(\Sigma_g^1,\mu)$ taking
$\iota_0$ to $\iota_1$.  Let $\Sigma$ be the stabilizer subsurface of the edge
$\{\iota_0,\iota_1\}$, and let $S_0 = \iota_0(\Sigma_h^1)$ and $S_1 = \iota_1(\Sigma_h^1)$.  
Let $S$ and $\eta$ and $\eta_0$ and $\eta_1$ be as follows:\setlength{\parskip}{0pt}
\begin{compactitem}
\item $S$ is a subsurface of $\Sigma_g^1 \setminus \Interior(\Sigma)$ containing $S_0$ and $S_1$
and satisfying $S \cong \Sigma_{2h}^1$.
\item $\eta$ is an embedded path in $\Sigma_g^1 \setminus \Interior(\Sigma \cup S)$ connecting a point
of $I$ to a basepoint of $S$ lying in $\partial S$.
\item For $i=0,1$, we have that $\eta_i$ is an embedded arc in $S \setminus \Interior(S_0 \cup S_1)$ connecting the basepoint
in $\partial S$ to a basepoint in $S_i$ lying in $\partial S_i$.
\item For $i=0,1$, the path $\eta \cdot \eta_i$ is isotopic to the tether of $\iota_i$ while keeping its initial
point in $I$ and its terminal point fixed.
\end{compactitem}
See here:\footnote{In this figure, $\eta_0$ and $\eta_1$ are disjoint aside from their initial points.  This can always be achieved, but
is not needed for our proof.}\setlength{\parskip}{\baselineskip}\\
\centerline{\psfig{file=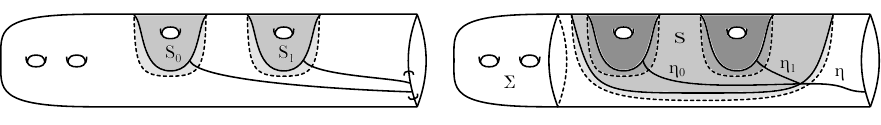,scale=100}}
It follows that there exists a $\Lambda$-marking $\mu_0$ on $\Sigma$ and some $\lambda$ such that
$\mu$ is the $(S,\lambda)$-stabilization of $\mu_0$.

Using the change of coordinates principle from \cite[\S 1.3.2]{FarbMargalitPrimer}, we can find
$F \in \Mod(S)$ taking $S_0 \cup \eta_0$ to something isotopic to $S_1 \cup \eta_1$.  This
isotopy will fix the common initial point of $\eta_0$ and $\eta_1$.  Let $f \in \Mod(\Sigma_g^1)$
be the image of $F$ under the the map $\Mod(S) \rightarrow \Mod(\Sigma_g^1)$.  Since $f$ is supported
on $S$, we have $f \in \Torelli(\Sigma_g^1,\mu)$, and by construction we have $f(\iota_0) = \iota_1$.
\end{proof}

Now assume that $k > 0$ and that the theorem is true for simplices of dimension $k-1$.
For some $g \geq 2h + 2|\Lambda|+1+kh$, let $\mu$ be a $\Lambda$-marking on $\Sigma_g^1$ and
$I \subset \partial \Sigma_g^1$ be an open interval.  Consider $k$-simplices
$\sigma$ and $\sigma'$ of $\TS_{h}(\Sigma_g^1,I,\mu)$.  Enumerate these simplices using
the natural ordering discussed above:
\begin{equation}
\label{eqn:enumerateorder}
\sigma = \{\iota_0,\ldots,\iota_{k}\} \quad \text{and} \quad
\sigma' = \{\iota'_0,\ldots,\iota'_{k}\}.
\end{equation}
We want to find some $f \in \Torelli(\Sigma_g^1,\mu)$ such that $f(\sigma) = \sigma'$.
By the base case $k=0$, there exists some $f_0 \in \Torelli(\Sigma_g^1,\mu)$ such that
$f(\iota_0) = \iota'_0$.  Replacing $\sigma$ by $f(\sigma)$, we can assume that
$\iota_0 = \iota'_0$.  

Define
\[\sigma_1 = \{\iota_1,\ldots,\iota_k\} \quad \text{and} \quad \sigma'_1 = \{\iota'_1,\ldots,\iota'_k\}.\]
Both $\sigma_1$ and $\sigma'_1$ are $(k-1)$-simplices in the link of the vertex $\iota_0$, and our goal
is to find an element $f_1$ in the $\Torelli(\Sigma_g^1,\mu)$-stabilizer of $\iota_0$ such that
$f_1(\sigma_1) = \sigma'_1$.  

Let $\Sigma'$ be the stabilizer subsurface of $\iota_0$ and let
$\mu'$ be the $\Lambda$-marking on $\Sigma'$ given by Lemma \ref{lemma:vanishtetheredstab}, so
the $\Torelli(\Sigma_g^1,\mu)$-stabilizer of $\iota_0$ is $\Torelli(\Sigma',\mu')$.
The surface $\Sigma'$ can be constructed
by removing the interior of $\iota_0(\Sigma_h^1)$ and then cutting open the result along the tether:
\centerline{\psfig{file=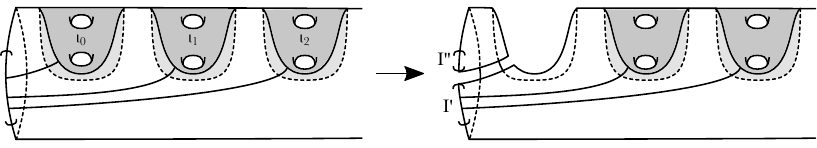,scale=100}}
We thus have $\Sigma' \cong \Sigma_{g-h}^1$.  Cutting along the tether divides the interval $I \subset \partial \Sigma_g^1$
into two disjoint intervals $I',I'' \subset \partial \Sigma'$, and the link of $\iota_0$ in
$\TS_{h}(\Sigma_g^1,I,\mu)$ can be identified with $\TS_h(\Sigma',I' \sqcup I'',\mu')$.  Identifying
$\sigma_1$ and $\sigma'_1$ with simplices in $\TS_h(\Sigma',I' \sqcup I'',\mu')$, the
key observation is that since we enumerated the simplices in \eqref{eqn:enumerateorder} using
the order coming from $I$, we have (possibly flipping $I'$ and $I''$) that
$\sigma_1,\sigma'_1 \subset \TS_h(\Sigma',I',\mu')$.  Since $\Sigma' \cong \Sigma_{g-h}^1$ and
\[g-h \geq (2h + 2|\Lambda|+1+kh) - h = 2h + 2|\Lambda|+1+(k-1)h,\]
we can apply our inductive hypothesis
and find some $f_1 \in \Torelli(\Sigma',\mu')$ with $f_1(\sigma_1) = \sigma_1'$, as desired.
\end{proof}

\subsection{Proof of stability for surfaces with one boundary component}
\label{section:proofone}

We now prove Theorems \ref{maintheorem:stable} and \ref{maintheorem:nonabelianstable}.

\begin{proof}[Proof of Theorem \ref{maintheorem:stable} and \ref{maintheorem:nonabelianstable}]
The proofs of the two theorems are identical, so we will give the details for
Theorem \ref{maintheorem:nonabelianstable}.
We start by recalling the statement and introducing some notation.
Let $\Lambda$ be a nontrivial finite group, let $\mu$ be a $\Lambda$-marking on $\Sigma_g^1$, and let $\mu'$ be
the stabilization of $\mu$ to $\Sigma_{g+1}^1$ in the sense of the introduction.\footnote{This uses a specific
choice of stabilization data $(S,\lambda)$}  Setting
\[c = |\Lambda|+2 \quad \text{and} \quad d = 2|\Lambda|+2,\] 
we want to prove that the map
$\HH_k(\Torelli(\Sigma_g^1,\mu)) \rightarrow \HH_k(\Torelli(\Sigma_{g+1}^1,\mu'))$
induced by the stabilization map
$\Torelli(\Sigma_g^1,\mu) \rightarrow \Torelli(\Sigma_{g+1}^1,\mu')$
is an isomorphism if $g \geq ck + d$ and a surjection
if $g = ck+d-1$.
We will prove this using Theorem \ref{theorem:stabilitymachine}.  This requires fitting 
$\Torelli(\Sigma_g^1,\mu) \hookrightarrow \Torelli(\Sigma_{g+1}^1,\mu')$ into an increasing
sequence of group $\{G_n\}$ and constructing appropriate simplicial complexes.  

Corollary \ref{corollary:destabilizeone} says that there exists an embedding
$\Sigma_{|\Lambda|}^1 \hookrightarrow \Sigma_g^1$ and a $\Lambda$-marking $\mu_{|\Lambda|}$ on
$\Sigma_{|\Lambda|}^1$ such that $\mu_{|\Lambda|}$ is a destabilization of $\mu$.
The embedding $\Sigma_{|\Lambda|}^1 \hookrightarrow \Sigma_g^1$ can be factored
into a sequence of embeddings 
\[\Sigma_{|\Lambda|}^1 \hookrightarrow \Sigma_{|\Lambda|+1}^1 \hookrightarrow \cdots \hookrightarrow \Sigma_g^1,\]
which can then be continued to
\[\Sigma_{|\Lambda|}^1 \hookrightarrow \Sigma_{|\Lambda|+1}^1 \hookrightarrow \cdots \hookrightarrow \Sigma_g^1 \hookrightarrow \Sigma_{g+1}^1 \hookrightarrow \Sigma_{g+2}^1 \hookrightarrow \cdots.\]
As in the following figure, we can break up the destabilization data $(S,\lambda)$ for the destabilization
$\mu_{|\Lambda|}$ of $\mu$ into stabilization data $(S_h,\lambda_h)$ for $|\Lambda|+1 \leq h \leq g$, where
$(S_h,\lambda_h)$ allows us to stabilize from $\Sigma_{h-1}^1$ to $\Sigma_h^1$:\\
\centerline{\psfig{file=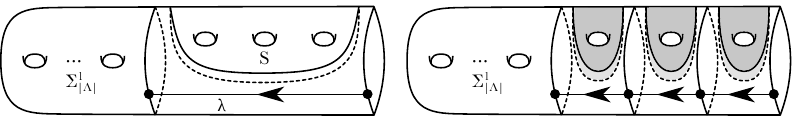,scale=100}}
Starting with $\mu_{|\Lambda|}$, for $h+1 \leq |\Lambda| \leq g$ inductively 
let $\mu_{h}$ be the $(S_h,\lambda_h)$-stabilization of $\mu_{h-1}$ on $\Sigma_{h-1}^1$ to $\mu_h$ on
$\Sigma_h^1$.  By construction, $\mu_g = \mu$.  Continue stabilizing (now using the choice of stabilization data
from the introduction) to define $\mu_h$ on $\Sigma_h^1$ for $h \geq g+1$, so $\mu_{g+1} = \mu'$.

We have thus fit our partial Torelli groups into an increasing sequence of groups
\[\Torelli(\Sigma_{|\Lambda|}^1,\mu_{|\Lambda|}) \subset \Torelli(\Sigma_{|\Lambda|+1}^1,\mu_{|\Lambda|+1}) \subset \Torelli(\Sigma_{|\Lambda|+2}^1,\mu_{|\Lambda|+2}) \subset \cdots.\]
For $h \geq |\Lambda|$, let $I_{h} \subset \partial \Sigma_{h}^1$ be an open interval.  
Theorem \ref{theorem:vanishsurfacescon} says that $\TS_1(\Sigma_h^1,I_h,\mu_h)$ is
$\frac{h-(d+1)}{c}$-connected, where $c$ and $d$ are as defined in the first paragraph.

For $n \geq 0$, let
\[G_n = \Torelli(\Sigma_{d+n},\mu_{d+n}) \quad \text{and} \quad
X_n = \TS_1(\Sigma_{d+n},I_{d+n},\mu_{d+n}).\]
For this to make sense, we must have $d+n \geq |\Lambda|$, which follows from
\[d+n = 2|\Lambda|+2+n \geq |\Lambda|.\]
We thus have an increasing sequence of groups
\[G_0 \subset G_1 \subset G_2 \subset \cdots\]
with $G_n$ acting on $X_n$.  The indexing convention here is chosen such that $X_1$ is $0$-connected and
more generally such that $X_n$ is $\frac{n-1}{c}$-connected, as in Theorem \ref{theorem:stabilitymachine}.
Our goal is to prove that the map $\HH_k(G_{n-1}) \rightarrow \HH_k(G_{n})$ is an isomorphism
for $n \geq ck+1$ and a surjection for $n = ck$, which will follow from Theorem \ref{theorem:stabilitymachine}
once we check its conditions:
\setlength{\parskip}{0pt}
\begin{compactitem}
\item The first is that $X_n$ is $\frac{n-1}{c}$-connected, which follows
from Theorem \ref{theorem:vanishsurfacescon}.
\item The second is that for $0 \leq i < n$, the group $G_{n-i-1}$ is the $G_n$-stabilizer
of some $i$-simplex of $X_n$, which follows from Lemma \ref{lemma:vanishtetheredstab}
via the following picture:
\end{compactitem}
\centerline{\psfig{file=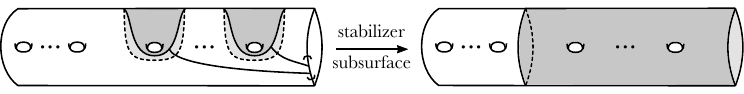,scale=100}}
\begin{compactitem}
\item The third is that for all $0 \leq i < n$, the group $G_n$ acts
transitively on the $i$-simplices of $X_n$, which follows from
Lemma \ref{lemma:vanishtetheredtran}.  For transitively on the $i$-simplices
this lemma requires that the genus $g = d+n$ used to define $G_n = \Torelli(\Sigma_{d+n},\mu_{d+n})$
satisfies $g \geq 3+2|\Lambda|+i$, which follows from the fact that
\[d+n = (2|\Lambda|+2) + n \geq (2|\Lambda|+2) + (i+1) = 3 + 2|\Lambda|+i.\]
\item The fourth is that for all $n \geq c+1$ and all $1$-simplices $e$ of $X_n$
whose boundary consists of vertices $v$ and $v'$, there exists some $\lambda \in G_n$
such that $\lambda(v) = v'$ and such that $\lambda$ commutes with all elements of $(G_n)_e$.
This actually does not require the condition $n \geq c+1$.
Let $\Sigma$ be the stabilizer subsurface of $e$, so by Lemma \ref{lemma:vanishtetheredstab} the
stabilizer $(G_n)_e$ consists of mapping classes supported on $\Sigma$.
The surface $\Sigma_{d+n}^1 \setminus \Interior(\Sigma)$ is diffeomorphic
to $\Sigma_2^2$ (as in the picture above), and in particular is connected.
The ``change of coordinates principle'' from \cite[\S 1.3.2]{FarbMargalitPrimer} implies that
we can find a mapping class $\lambda$ supported on
on $\Sigma_{d+n}^1 \setminus \Interior(\Sigma)$ taking the tethered torus $v$ to $v'$.
This $\lambda$ clearly lies in $G_n$ and commutes with $(G_n)_e$.\qedhere
\end{compactitem}
\setlength{\parskip}{\baselineskip}
\end{proof}

\section{Homology-marked partitioned surfaces}
\label{section:marked}

We now turn to partial Torelli groups on surfaces with multiple boundary components.  
Unfortunately, this introduces genuine
difficulties in the proofs, so quite a bit more technical setup is needed.  This section
contains the categorical framework we will need to even state our result.

Let $\Surf$ be the category whose objects are 
compact connected oriented surfaces with boundary 
and whose morphisms are orientation-preserving embeddings.
There is a functor from $\Surf$ to groups taking 
$\Sigma \in \Surf$ to $\Mod(\Sigma)$ and a morphism 
$\Sigma \hookrightarrow \Sigma'$ to the map
$\Mod(\Sigma) \rightarrow \Mod(\Sigma')$ that extends mapping classes by the identity.
In this section, we augment $\Surf$ to construct a new category $\PSurf$ on which we can
define partial Torelli groups.  This is done in two steps: in \S \ref{section:psurf}
we define the category $\PSurf$ along with a ``partitioned homology functor'', and
in \S \ref{section:mpsurf} we discuss homology markings 
and construct their associated partial Torelli groups.

\subsection{The category \texorpdfstring{$\PSurf$}{PSurf}}
\label{section:psurf}

We start with the partitioned surface category, which was introduced in
\cite{PutmanCutPaste}.

\p{Motivation}
This category captures aspects of the 
homology of a larger surface in which our surface is embedded.  For
instance, consider the following embedding of a genus $3$ surface $\Sigma$ with
$6$ boundary components into $\Sigma_7^1$:\\
\centerline{\psfig{file=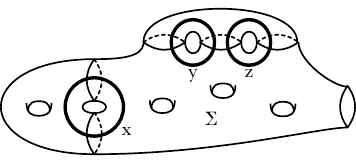,scale=100}}
For $f \in \Mod(\Sigma)$, the action of $f$ on $\HH_1(\Sigma)$ does not determine
the action of $f$ on $\HH_1(\Sigma_7^1)$.  The issue is that we also need
to know the action of $f$ on $[x],[y],[z] \in \HH_1(\Sigma_7^1)$.  The portions
of these homology classes that live on $\Sigma$ are arcs connecting boundary
components, so we must consider relative homology groups that incorporate
such arcs.  However, we do not want to allow all arcs connecting 
boundary components since some of these
cannot be completed to loops in the larger ambient surface.

\p{Category}
To that end, we define a category $\PSurf$ whose objects are pairs $(\Sigma, \cP)$ as follows:
\setlength{\parskip}{0pt}
\begin{compactitem}
\item $\Sigma$ is a compact connected oriented surface with boundary, and
\item $\cP$ is a partition of the components of $\partial \Sigma$.
\end{compactitem}
The partition $\cP$ tells us which boundary components are allowed to be 
connected by arcs.  The morphisms in $\PSurf$ from $(\Sigma,\cP)$ to
$(\Sigma',\cP')$ are orientation-preserving embeddings
$\Sigma \hookrightarrow \Sigma'$ that are compatible with the partitions
$\cP$ and $\cP'$ in the following sense.
For a component $S$ of $\Sigma' \setminus \Interior(\Sigma)$,
let $B_S$ (resp.\ $B_S'$) denote the set of components of $\partial \Sigma$ (resp.\ $\partial \Sigma'$)
that lie in $S$.  In the degenerate case where $S \cong S^1$ (so $S$ is a component of $\partial \Sigma$ and
$\partial \Sigma'$), we have $B_S = B'_S$.
Our compatibility requirements are then:
\setlength{\parskip}{0pt}
\begin{compactitem}
\item each $B_S$ is a subset of some $p \in \cP$, and
\item for all $p' \in \cP'$ and all $\partial_1',\partial_2' \in p'$ such that
$\partial_1' \in B_{S_1}'$ and $\partial_2' \in B_{S_2}'$ with $S_1 \neq S_2$, there exists
some $p \in \cP$ such that $B_{S_1} \cup B_{S_2} \subset p$.
\end{compactitem}
\setlength{\parskip}{\baselineskip}

\begin{example}
Let $\Sigma = \Sigma_0^6$ and 
$\cP = \{\{\partial_1,\partial_2,\partial_3,\partial_4\},\{\partial_5,\partial_6\}\}$
and $\Sigma' = \Sigma_3^3$ and
$\cP' = \{\{\partial'_1,\partial'_2\},\{\partial'_3\}\}$.  Here are two embeddings
$(\Sigma,\cP) \hookrightarrow (\Sigma',\cP')$ that are not $\PSurf$-morphisms
and one that is:\\
\centerline{\psfig{file=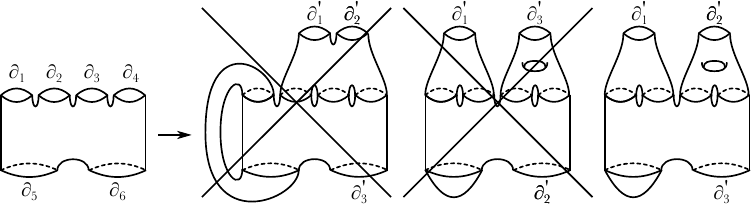,scale=100}}
We remark that the difference between the second and third embedding is the labeling
of the boundary components.
\end{example}

\p{Partitioned homology}
Consider some $(\Sigma,\cP) \in \PSurf$.  Say that components $\partial_1$ and $\partial_2$
of $\partial \Sigma$ are {\em $\cP$-adjacent} if there exists some $p \in \cP$ with
$\partial_1,\partial_2 \in p$.  Define
$\HH_1^{\cP}(\Sigma,\partial \Sigma)$ to be the subgroup
of the relative homology group
$\HH_1(\Sigma,\partial \Sigma)$ spanned by
the homology classes of oriented closed curves and arcs
connecting $\cP$-adjacent boundary components.  The group $\Mod(\Sigma)$ acts on
$\HH_1^{\cP}(\Sigma,\partial \Sigma)$.

\begin{remark}
This is slightly different from the partitioned homology group defined
in \cite{PutmanCutPaste}, which was {\em not} functorial.  
The Torelli groups defined via the above homology groups are thus different from those
in \cite{PutmanCutPaste}.
\end{remark}

\p{Functoriality}
The assignment 
\[(\Sigma,\cP) \mapsto \HH_1^{\cP}(\Sigma,\partial \Sigma)\] 
is a contravariant functor from $\PSurf$ to abelian groups.  To see this, consider
a $\PSurf$-morphism $\iota\colon (\Sigma,\cP) \rightarrow (\Sigma',\cP')$.  Identify
$\Sigma$ with its image under $\iota$.  We then have maps
\[\HH_1\left(\Sigma',\partial \Sigma'\right) 
\longrightarrow \HH_1\left(\Sigma',\Sigma' \setminus \Interior\left(\Sigma\right)\right)
\stackrel{\cong}{\longrightarrow} \HH_1(\Sigma,\partial \Sigma),\]
where the second map is the excision isomorphism.
From the definition of a $\PSurf$-morphism, it follows immediately that this composition
restricts to a map
\[\iota^{\ast}\colon \HH_1^{\cP'}(\Sigma',\partial \Sigma') \rightarrow \HH_1^{\cP}(\Sigma,\partial \Sigma).\]

\begin{example}
Let $\Sigma = \Sigma_0^4$ and $\Sigma' = \Sigma_4^3$.  Let $\cP$ (resp.\ $\cP'$)
be the partition of the components of $\partial \Sigma$ (resp.\ $\partial \Sigma'$)
consisting of a single partition element containing all the boundary components.
The following picture shows a $\PSurf$-morphism $\iota\colon (\Sigma,\cP) \rightarrow (\Sigma',\cP')$
along with $x_1,x_2 \in \HH_1^{\cP'}(\Sigma',\partial \Sigma')$ and
$\iota^{\ast}(x_1),\iota^{\ast}(x_2) \in \HH_1^{\cP}(\Sigma,\partial \Sigma)$:\\
\centerline{\psfig{file=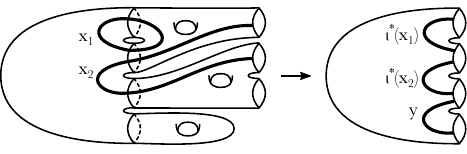,scale=100}}
To simplify the picture we do not indicate the orientations of the curves/arcs.
The above picture also shows an element $y \in \HH_1^{\cP}(\Sigma,\partial \Sigma)$ that
is not in the image of $\iota^{\ast}$.
\end{example}

\begin{example}
Let $\Sigma = \Sigma_0^4$ and $\Sigma' = \Sigma_2$.  Let $\{\partial_1,\ldots,\partial_4\}$
be the boundary components of $\Sigma$, and let $\cP = \{\{\partial_1,\partial_2\},\{\partial_3,\partial_4\}\}$
and $\cP' = \emptyset$.  The following picture shows a $\PSurf$-morphism $\iota\colon (\Sigma,\cP) \rightarrow (\Sigma',\cP')$
along with $x \in \HH_1^{\cP'}(\Sigma',\partial \Sigma')$ and
$\iota^{\ast}(x) \in \HH_1^{\cP}(\Sigma,\partial \Sigma)$:\\
\centerline{\psfig{file=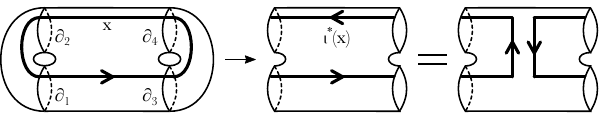,scale=100}}
As it is initially drawn,
$\iota^{\ast}(x)$ does not appear to be in $\HH_1^{\cP}(\Sigma,\partial \Sigma)$ since it
is a pair of arcs connecting boundary components that are not $\cP$-adjacent; however, as the
figure shows this 
pair of arcs in homologous to a pair of arcs connecting boundary components that are $\cP$-adjacent.
\end{example}

\p{Action on partitioned homology}
The mapping class group is a covariant functor from $\Surf$ to groups, while
the partitioned homology group is a contravariant functor from $\PSurf$ to
abelian groups.  They are related by the following ``push-pull'' formula.

\begin{lemma}
\label{lemma:pushpull}
Let $\iota\colon (\Sigma,\cP) \rightarrow (\Sigma',\cP')$ be a $\PSurf$-morphism.  Let
$\iota_{\ast}\colon \Mod(\Sigma) \rightarrow \Mod(\Sigma')$ be the induced map
on mapping class groups and let 
$\iota^{\ast}\colon \HH_1^{\cP'}(\Sigma',\partial \Sigma') \rightarrow \HH_1^{\cP}(\Sigma,\partial \Sigma)$ be the induced map on partitioned homology groups.  Then
\[\iota^{\ast}\left(\iota_{\ast}\left(f\right)\left(x'\right)\right) = f\left(\iota^{\ast}\left(x'\right)\right) \quad \quad \left(f \in \Mod\left(\Sigma\right), x' \in \HH_1^{\cP'}\left(\Sigma',\partial \Sigma'\right)\right).\]
\end{lemma}
\begin{proof}
Obvious.
\end{proof}

\subsection{Homology markings on \texorpdfstring{$\PSurf$}{PSurf}}
\label{section:mpsurf}

Recall that $A$ is a fixed finitely generated abelian group.

\p{Markings and partial Torelli groups}
Consider $(\Sigma,\cP) \in \PSurf$.  An {\em $A$-homology marking} 
on $(\Sigma,\cP)$ is a homomorphism 
$\mu\colon \HH_1^{\cP}(\Sigma,\partial \Sigma) \rightarrow A$.  
The associated {\em partial Torelli group} is
\[\Torelli(\Sigma,\cP,\mu) = \Set{$f \in \Mod(\Sigma)$}{$\mu(f(x))=\mu(x)$ for all $x \in \HH_1^{\cP}(\Sigma,\partial \Sigma)$}.\]

\p{Stabilizations}
If $\iota\colon (\Sigma,\cP) \rightarrow (\Sigma',\cP')$ is a $\PSurf$-morphism and
$\mu$ is an $A$-homology marking on $(\Sigma,\cP)$, then the {\em stabilization} of $\mu$
to $(\Sigma',\cP')$ is the composition
\[\HH_1^{\cP'}(\Sigma',\partial \Sigma') \stackrel{\iota^{\ast}}{\longrightarrow} \HH_1^{\cP}(\Sigma,\partial \Sigma) \stackrel{\mu}{\longrightarrow} A.\]
With this definition, we have the following lemma.

\begin{lemma}
\label{lemma:partialtorellifun}
Let $\iota\colon (\Sigma,\cP) \rightarrow (\Sigma',\cP')$ be a $\PSurf$-morphism,
let $\mu$ be an $A$-homology marking on $(\Sigma,\cP)$, and let $\mu'$ be the
stabilization of $\mu$ to $(\Sigma',\cP')$.  
Let $\iota_{\ast}\colon \Mod(\Sigma) \rightarrow \Mod(\Sigma')$ be the induced map.
Then $\iota_{\ast}(\Torelli(\Sigma,\cP,\mu)) \subset \Torelli(\Sigma',\cP',\mu')$.
\end{lemma}
\begin{proof}
Let $\iota^{\ast}\colon \HH_1^{\cP'}(\Sigma',\partial \Sigma') \rightarrow \HH_1^{\cP}(\Sigma,\partial \Sigma)$
be the induced map.  For $f \in \Torelli(\Sigma,\cP,\mu)$ and 
$x' \in \HH_1^{\cP'}(\Sigma',\partial \Sigma')$, we have
\[\mu'\left(\iota_{\ast}\left(f\right)\left(x'\right)\right) = \mu\left(\iota^{\ast}\left(\iota_{\ast}\left(f\right)\left(x'\right)\right)\right) = \mu\left(f\left(\iota^{\ast}\left(x'\right)\right)\right) = \mu\left(\iota^{\ast}\left(x'\right)\right) = \mu'\left(x'\right).\]
Here the second equality follows from Lemma \ref{lemma:pushpull} and the third from the fact that $f \in \Torelli(\Sigma,\cP,\mu)$.  The lemma follows.
\end{proof}

\section{Stability for surfaces with multiple boundary components}
\label{section:maintheorem}

In this section, we state our stability theorem for the partial Torelli groups on surfaces with multiple boundary
components and reduce this theorem
to a result that will be proved in the next section using the homological
stability machine.  The statement of our result is
in \S \ref{section:statement} and the reductions are in \S \ref{section:reductioncap}, \S \ref{section:boundarystab},
and \S \ref{section:doubleboundarystab}.

\subsection{Statement of result}
\label{section:statement}

To get around the issues with closed surfaces underlying Theorem \ref{maintheorem:closed} from the Introduction,
we will need to impose some conditions on our stabilization maps.

\p{Support}
If $\mu$ is an $A$-homology marking on $(\Sigma,\cP) \in \PSurf$, we say that
$\mu$ is {\em supported on a genus $h$ symplectic subsurface} if there exists
a $\PSurf$-morphism $(\Sigma',\cP') \rightarrow (\Sigma,\cP)$ with $\Sigma' \cong \Sigma_h^1$
and an $A$-homology marking $\mu'$ on $(\Sigma',\cP')$ such that $\mu$ is the stabilization
of $\mu'$ to $(\Sigma,\cP)$.
If there exists some $h \geq 1$ such that $\mu$ is supported on a genus $h$ symplectic
subsurface, then we will simply say that $\mu$ is {\em supported on a symplectic subsurface}.

\begin{remark}
\label{remark:nonsupport}
Not all $A$-homology markings are supported on a symplectic subsurface.
Indeed, letting $\partial_1$ and $\partial_2$ be $\cP$-adjacent boundary components
of $\Sigma$, this condition implies that we can find an arc $\alpha$
connecting $\partial_1$ to $\partial_2$ such that $\mu([\alpha])=0$; see here:\\
\centerline{\psfig{file=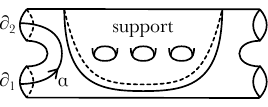,scale=100}}
It is easy to construct $A$-homology markings not
satisfying this property; for instance, let $A = \RH_0(\partial \Sigma)$ and
let $\mu\colon \HH_1^{\cP}(\Sigma,\partial \Sigma) \rightarrow A$ be the restriction
to $\HH_1^{\cP}(\Sigma,\partial \Sigma)$ of the boundary map
$\HH_1(\Sigma,\partial \Sigma) \rightarrow \RH_0(\partial \Sigma)$.
We will later show that this is the only obstruction; see
Lemma \ref{lemma:identifysupport} below.
\end{remark}

\p{Partition bijectivity}
Consider a $\PSurf$-morphism $(\Sigma,\cP) \rightarrow (\Sigma',\cP')$.  Identify
$\Sigma$ with its image in $\Sigma'$.  We will call this morphism {\em partition-bijective} if
the following holds for all $p \in \cP$:
\setlength{\parskip}{0pt}
\begin{compactitem}
\item Let $S$ be the union of the components of $\Sigma' \setminus \Interior(\Sigma)$ that
contain a boundary component in $p$.  Then there exists a unique $p' \in \cP'$ such that
$p'$ consists of the components of $S \cap \partial \Sigma'$.
\end{compactitem}
This condition implies in particular that $S$ contains components of $\partial \Sigma'$.
It rules out two kinds of morphisms:
\begin{compactitem}
\item Ones where for some $p \in \cP$ the union of the components of 
$\Sigma' \setminus \Interior(\Sigma)$ that contain a boundary component in $p$
contains no components of $\partial \Sigma'$.  See here:
\end{compactitem}
\centerline{\psfig{file=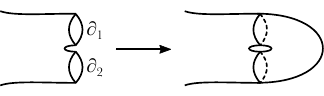,scale=100}}
\begin{compactitem}
\item[] Here $p = \{\partial_1,\partial_2\}$.
\item Ones where a single $p \in \cP$ ``splits'' into multiple elements of $\cP'$ like this:
\end{compactitem}
\centerline{\psfig{file=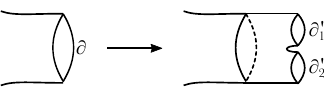,scale=100}}
\begin{compactitem}
\item[] Here $p = \{\partial\}$ and $\cP'$ contains both $\{\partial'_1\}$ and $\{\partial'_2\}$.
\end{compactitem}
\setlength{\parskip}{\baselineskip}

\p{Main theorem}
With this definition, we have the following theorem.

\begin{maintheorem}
\label{theorem:stableboundary}
Let $\mu$ be an $A$-homology marking on $(\Sigma,\cP) \in \PSurf$ 
that is supported on a symplectic subsurface.  Let
$(\Sigma,\cP) \rightarrow (\Sigma',\cP')$
be a partition-bijective $\PSurf$-morphism and let $\mu'$ be the stabilization
of $\mu$ to $(\Sigma',\cP')$.
Then the induced map $\HH_k(\Torelli(\Sigma,\cP,\mu)) \rightarrow \HH_k(\Torelli(\Sigma',\cP',\mu'))$
is an isomorphism if the genus of $\Sigma$ is at least
$(\rank(A)+2)k + (2\rank(A)+2)$.
\end{maintheorem}

\begin{remark}
Theorem \ref{theorem:stableboundary} does not assert that the map is a surjection if the genus of $\Sigma$ is at least
$(\rank(A)+2)k + (2\rank(A)+1)$.  We do not know if this is true -- while an appropriate surjectivity statement will follow
from our invocation of the homological stability machine, this will only cover certain special kinds of morphisms $(\Sigma,\cP) \rightarrow (\Sigma',\cP')$ (the
``double boundary stabilizations''), and the general case will be reduced to these special morphisms in a fairly involved
way.
\end{remark}

\p{Counterexamples}
We do not know whether or not the condition in Theorem \ref{theorem:stableboundary} 
that $\mu$ be supported on a symplectic subsurface is necessary.  However, the condition
that the morphism be partition-bijective is necessary.
Indeed, in \S \ref{section:closed} we will prove the following theorem.  The
condition of being {\em symplectically nondegenerate} in it will be defined in that
section; it is satisfied by most interesting homology markings.

\begin{theorem}
\label{theorem:counterexample}
Let $\mu$ be a symplectically nondegenerate $A$-homology marking on
$(\Sigma,\cP) \in \PSurf$ 
that is supported on a symplectic subsurface.  Let
$(\Sigma,\cP) \rightarrow (\Sigma',\cP')$
be a non-partition-bijective $\PSurf$-morphism and let $\mu'$ be the stabilization
of $\mu$ to $(\Sigma',\cP')$.  Assume that the genus of $\Sigma$
is at least $3\rank(A)+4$.
Then the induced map 
$\HH_1(\Torelli(\Sigma,\cP,\mu)) \rightarrow \HH_1(\Torelli(\Sigma',\cP',\mu'))$
is not an isomorphism.
\end{theorem}

\begin{remark}
The map is frequently not an isomorphism even when the genus of $\Sigma$ is smaller.
We use the genus assumption in Theorem \ref{theorem:counterexample} so we
can apply Theorem \ref{theorem:stableboundary} to change $\Sigma$ and $\Sigma'$ so
as to put ourselves in a situation where the phenomenon underlying
Theorem \ref{maintheorem:closed} occurs.
\end{remark}

\subsection{Reduction I: open cappings}
\label{section:reductioncap}

In this section, we reduce Theorem \ref{theorem:stableboundary} to certain
kinds of $\PSurf$-morphisms called open cappings, whose definition is below.

\p{Open cappings}
An {\em open capping} is a $\PSurf$-morphism 
$(\Sigma,\cP) \rightarrow (\Sigma',\cP')$ such that the following holds
for all $p \in \cP$:
\setlength{\parskip}{0pt}
\begin{compactitem}
\item Let $S$ be the union of the components of $\Sigma' \setminus \Interior(\Sigma)$ that
contain a boundary component in $p$.  Then $S$ is connected and $S \cap \partial \Sigma'$
consists of a single component.
\end{compactitem}
Unraveling the definition of a $\PSurf$-morphism, this implies that $\cP'$ is
the {\em discrete partition}, that is, the partition \setlength{\parskip}{\baselineskip}
$\cP' = \Set{$\{\partial'\}$}{$\partial'$ a component of $\partial \Sigma'$}$.
See the following example, where $\cP = \{\{\partial_1,\partial_2\}, \{\partial_3,\partial_4\}\}$:\\
\centerline{\psfig{file=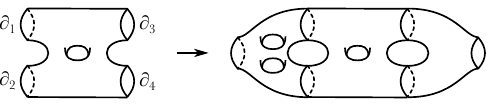,scale=100}}
By definition, an open capping is partition-bijective.

\begin{remark}
In \cite{PutmanCutPaste}, a capping is defined similarly to an open capping, but
where $\Sigma'$ is closed and $\partial S$ is simply an element of $\cP$.
\end{remark}

\p{Reduction}
The following is a special case of Theorem \ref{theorem:stableboundary}.

\begin{proposition}
\label{proposition:opencapping}
Let $\mu$ be an $A$-homology marking on $(\Sigma,\cP) \in \PSurf$ that is supported on a symplectic subsurface.
Let $(\Sigma,\cP) \rightarrow (\Sigma',\cP')$
be an open capping and let $\mu'$ be the stabilization
of $\mu$ to $(\Sigma',\cP')$.
Then the 
induced map $\HH_k(\Torelli(\Sigma,\cP,\mu)) \rightarrow \HH_k(\Torelli(\Sigma',\cP',\mu'))$
is an isomorphism if the genus of $\Sigma$ is at least $(\rank(A)+2)k + (2\rank(A)+2)$.
\end{proposition}

The proof of Proposition \ref{proposition:opencapping} begins in
\S \ref{section:boundarystab}.  First, we will use it to deduce Theorem
\ref{theorem:stableboundary}.

\begin{proof}[Proof of Theorem \ref{theorem:stableboundary}, assuming
Proposition \ref{proposition:opencapping}]
We start by recalling the statement of the theorem.
Let $\mu$ be an $A$-homology marking on $(\Sigma,\cP) \in \PSurf$ that is supported on a symplectic subsurface.
Let $(\Sigma,\cP) \rightarrow (\Sigma',\cP')$
be a partition-bijective $\PSurf$-morphism and let $\mu'$ be the stabilization
of $\mu$ to $(\Sigma',\cP')$.  Assume that the genus of
$\Sigma$ is at least $(\rank(A)+2)k + (2\rank(A)+2)$.
Our goal is to prove that
the induced map 
$\HH_k(\Torelli(\Sigma,\cP,\mu)) \rightarrow \HH_k(\Torelli(\Sigma',\cP',\mu'))$
is an isomorphism.

Identify $\Sigma$ with its image in $\Sigma'$.
The proof has two cases.  Recall that the discrete partition of the boundary
components of a surface $S$ is the partition $\Set{$\{\partial\}$}{$\partial$ a component of $\partial S$}$.

\begin{casea}
\label{casea:1}
$\cP$ is the discrete partition of $\Sigma$.
\end{casea}

Let $S_1,\ldots,S_b$ be the components of $\Sigma' \setminus \Interior(\Sigma)$.  For each
$1 \leq i \leq b$, let $B_{S_i}$ be the components of $\partial \Sigma$ that are contained
in $S_i$ and let $B'_{S_i}$ be the components of $\partial \Sigma'$ that are contained
in $S_i$.  The following hold:
\setlength{\parskip}{0pt}
\begin{compactitem}
\item Since $\cP$ is the discrete partition, each $B_{S_i}$ is a one-element set containing a single boundary component of 
$\Sigma$, and $\cP = \{B_{S_1},\ldots,B_{S_b}\}$.
\item Since the morphism $(\Sigma,\cP) \rightarrow (\Sigma',\cP')$ is partition-bijective,
each $B'_{S_i}$ is a non-empty set of boundary components of $\Sigma'$, and $\cP' = \{B'_{S_1},\ldots,B'_{S_b}\}$.
\end{compactitem}
See the following figure, where $\Sigma \cong \Sigma_1^3$ with the discrete partition $\cP = \{\{\partial_1\},\{\partial_2\},\{\partial_3\}\}$ and 
$\Sigma' \cong \Sigma_1^6$ with the partition $\cP' = \{\{\partial'_1\}, \{\partial'_2,\partial'_3\},\{\partial'_4,\partial'_5,\partial'_6\}\}$:\setlength{\parskip}{\baselineskip}\\
\centerline{\psfig{file=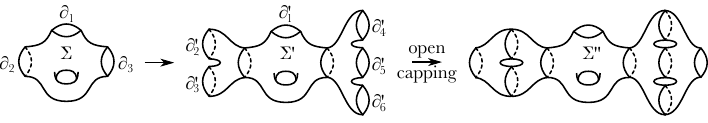,scale=100}}
As in that figure, let $(\Sigma',\cP') \rightarrow (\Sigma'',\cP'')$ be an open capping and
let $\mu''$ be the stabilization of $\mu'$ to $(\Sigma'',\cP'')$.
It follows from the above that the composition
\[(\Sigma,\cP) \longrightarrow (\Sigma',\cP') \longrightarrow (\Sigma'',\cP'')\]
is also an open capping.  We remark that this can fail if $\cP$ is not the discrete partition.  For instance,
consider the morphisms $(\Sigma,\cP) \rightarrow (\Sigma',\cP')$ and $(\Sigma',\cP') \rightarrow (\Sigma'',\cP'')$
in the following figure, where $\cP = \{\{\partial_1,\partial_2\}\}$ and $\cP' = \{\{\partial'_1\}\}$ and $\cP'' = \{\{\partial''_1\}\}$:\\
\centerline{\psfig{file=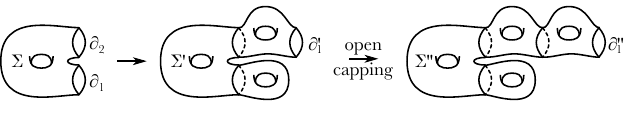,scale=100}}
We have maps
\[\HH_k(\Torelli(\Sigma,\cP,\mu)) \longrightarrow \HH_k(\Torelli(\Sigma',\cP',\mu')) \longrightarrow \HH_k(\Torelli(\Sigma'',\cP'',\mu'')).\]
Proposition \ref{proposition:opencapping} implies that
\[\HH_k(\Torelli(\Sigma,\cP,\mu)) \longrightarrow \HH_k(\Torelli(\Sigma'',\cP'',\mu''))
\quad \text{and} \quad
\HH_k(\Torelli(\Sigma',\cP',\mu')) \longrightarrow \HH_k(\Torelli(\Sigma'',\cP'',\mu''))\]
are isomorphisms.  We conclude that the map
\[\HH_k(\Torelli(\Sigma,\cP,\mu)) \longrightarrow \HH_k(\Torelli(\Sigma',\cP',\mu'))\]
is an isomorphism, as desired.

\begin{casea}
\label{casea:2}
$\cP$ is not the discrete partition of $\partial \Sigma$.
\end{casea}

Since $\mu$ is supported on a symplectic subsurface, we can find a 
$\PSurf$-morphism $(\Sigma'',\cP'') \rightarrow (\Sigma,\cP)$ with $\Sigma'' \cong \Sigma_h^1$
and an $A$-homology marking $\mu''$ on $(\Sigma'',\cP'')$ such that $\mu$ is the stabilization
of $\mu''$ to $(\Sigma,\cP)$.  We can factor $(\Sigma'',\cP'') \rightarrow (\Sigma,\cP)$
as 
\[(\Sigma'',\cP'') \rightarrow (\Sigma''',\cP''') \rightarrow (\Sigma,\cP)\]
such that $\Sigma'''$ has the same genus as $\Sigma$, such that $\cP'''$ is the discrete partition of $\partial \Sigma'''$, and such
that $(\Sigma''',\cP''') \rightarrow (\Sigma,\cP)$ is partition-bijective; see here:\\
\centerline{\psfig{file=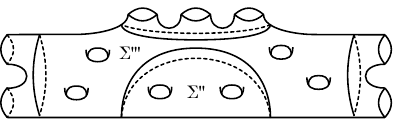,scale=100}}
In this example, $\cP$ consists of three sets of boundary components (the ones on the
left, right, and top).  Let $\mu'''$ be the stabilization of $\mu''$ to 
$(\Sigma''',\cP''')$.  We have maps
\[\HH_k(\Torelli(\Sigma''',\cP''',\mu''')) \longrightarrow \HH_k(\Torelli(\Sigma,\cP,\mu)) \longrightarrow \HH_k(\Torelli(\Sigma',\cP',\mu')).\]
Case \ref{casea:1} implies that the maps
\[\HH_k(\Torelli(\Sigma''',\cP''',\mu''')) \longrightarrow \HH_k(\Torelli(\Sigma,\cP,\mu))
\quad \text{and} \quad
\HH_k(\Torelli(\Sigma''',\cP''',\mu''')) \longrightarrow \HH_k(\Torelli(\Sigma',\cP',\mu'))\]
are isomorphisms.  We conclude that the map
\[\HH_k(\Torelli(\Sigma,\cP,\mu)) \longrightarrow \HH_k(\Torelli(\Sigma',\cP',\mu'))\]
is an isomorphism, as desired.
\end{proof}

\subsection{Reduction II: boundary stabilizations}
\label{section:boundarystab}

In this section, we reduce Proposition \ref{proposition:opencapping} to showing
that certain kinds of $\PSurf$-morphisms called increasing boundary stabilizations
and decreasing boundary stabilizations induce isomorphisms on homology.

\p{Increasing boundary stabilization}
Consider $(\Sigma,\cP) \in \PSurf$.  An {\em increasing boundary stabilization}
of $(\Sigma,\cP)$ is a $\PSurf$-morphism
$(\Sigma,\cP) \rightarrow (\Sigma',\cP')$ constructed as follows.
Let $\partial$ be a component of $\partial \Sigma$ and let $p \in \cP$ be the partition
element with $\partial \in p$.  Also, let
$\partial \Sigma_0^3 = \{\partial_1', \partial_2',\partial_3'\}$.
\setlength{\parskip}{0pt}
\begin{compactitem}
\item $\Sigma'$ is obtained by attaching $\Sigma_0^3$ to $\Sigma$ by gluing $\partial_1' \subset \Sigma_0^3$ to $\partial \subset \Sigma$.
\item $\cP'$ is obtained from $\cP$ by replacing $p$ with
$p' = \left(p \setminus \{\partial\}\right) \cup \{\partial_2', \partial_3'\}$.
\end{compactitem}
See here:\setlength{\parskip}{\baselineskip} \\
\centerline{\psfig{file=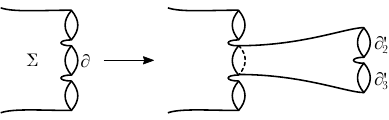,scale=100}}
In \S \ref{section:doubleboundarystab}, we will prove the following.

\begin{proposition}
\label{proposition:increasingboundarystab}
Let $\mu$ be an $A$-homology marking on $(\Sigma,\cP) \in \PSurf$ that is supported on a symplectic subsurface.
Let $(\Sigma,\cP) \rightarrow (\Sigma',\cP')$
be an increasing boundary stabilization and let $\mu'$ be the stabilization
of $\mu$ to $(\Sigma',\cP')$.
Then the induced map $\HH_k(\Torelli(\Sigma,\cP,\mu)) \rightarrow \HH_k(\Torelli(\Sigma',\cP',\mu'))$
is an isomorphism if the genus of $\Sigma$ is at least $(\rank(A)+2)k + (2\rank(A)+2)$.
\end{proposition}

\p{Decreasing boundary stabilization}
Consider $(\Sigma,\cP) \in \PSurf$.  A {\em decreasing boundary stabilization}
of $(\Sigma,\cP)$ is a $\PSurf$-morphism
$(\Sigma,\cP) \rightarrow (\Sigma',\cP')$ constructed as follows.
Let $\partial_1$ and $\partial_2$ be distinct components of $\partial \Sigma$ that
both lie in some $p \in \cP$, and let
$\partial \Sigma_0^3 = \{\partial_1', \partial_2',\partial_3'\}$.
\setlength{\parskip}{0pt}
\begin{compactitem}
\item $\Sigma'$ is obtained by attaching $\Sigma_0^3$ to $\Sigma$ by gluing
$\partial_1'$ and $\partial_2'$ to $\partial_1$ and $\partial_2$, respectively.
\item $\cP'$ is obtained from $\cP$ by replacing $p$ with
$p' = \left(p \setminus \{\partial_1,\partial_2\}\right) \cup \{\partial_3'\}$.
\end{compactitem}
See here:\setlength{\parskip}{\baselineskip} \\
\centerline{\psfig{file=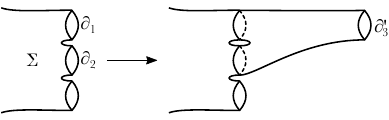,scale=100}}
In \S \ref{section:doubleboundarystab}, we will prove the following.

\begin{proposition}
\label{proposition:decreasingboundarystab}
Let $\mu$ be an $A$-homology marking on $(\Sigma,\cP) \in \PSurf$ that is supported on a symplectic subsurface.
Let $(\Sigma,\cP) \rightarrow (\Sigma',\cP')$
be a decreasing boundary stabilization and let $\mu'$ be the stabilization
of $\mu$ to $(\Sigma',\cP')$.
Then the induced map $\HH_k(\Torelli(\Sigma,\cP,\mu)) \rightarrow \HH_k(\Torelli(\Sigma',\cP',\mu'))$
is an isomorphism if the genus of $\Sigma$ is at least $(\rank(A)+2)k + (2\rank(A)+2)$.
\end{proposition}

\p{Deriving Proposition \ref{proposition:opencapping}}
As we said above, we will prove Propositions 
\ref{proposition:increasingboundarystab} and \ref{proposition:decreasingboundarystab} 
in \S \ref{section:doubleboundarystab}.
Here we will explain how to use them to prove Proposition \ref{proposition:opencapping}.

\begin{proof}[Proof of Proposition \ref{proposition:opencapping}, assuming
Propositions \ref{proposition:increasingboundarystab} and \ref{proposition:decreasingboundarystab}]
It is geometrically clear that an open capping
$(\Sigma,\cP) \rightarrow (\Sigma',\cP')$ can be factored as a composition
of increasing boundary stabilizations and decreasing boundary stabilizations.
For instance,\\
\centerline{\psfig{file=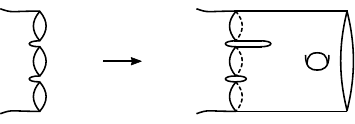,scale=100}}
can be factored as\\
\centerline{\psfig{file=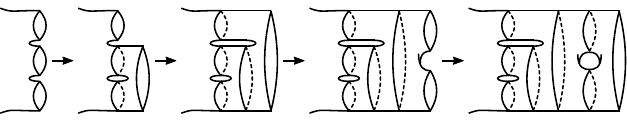,scale=100}}
The proposition follows.
\end{proof}

\subsection{Reduction III: double boundary stabilizations}
\label{section:doubleboundarystab}

In this section, we adapt a beautiful idea of Hatcher--Vogtmann \cite{HatcherVogtmannTethers}
to show how to reduce our two different boundary stabilizations
(increasing and decreasing) to a single kind of stabilization called a double
boundary stabilization.

\p{Double boundary stabilization}
Consider $(\Sigma,\cP) \in \PSurf$.  A {\em double boundary stabilization}
of $(\Sigma,\cP)$ is a $\PSurf$-morphism
$(\Sigma,\cP) \rightarrow (\Sigma',\cP')$ constructed as follows.
Let $\partial_1$ and $\partial_2$ be components of $\partial \Sigma$ that
lie in a single element $p \in \cP$.  Also, let
$\partial \Sigma_0^4 = \{\partial_1', \partial_2',\partial_3',\partial_4'\}$.
\setlength{\parskip}{0pt}
\begin{compactitem}
\item $\Sigma'$ is obtained by attaching $\Sigma_{0,4}$ to $\Sigma$ by gluing
$\partial_1'$ and $\partial_2'$ to $\partial_1$ and $\partial_2$, respectively.
\item $\cP'$ is obtained from $\cP$ by replacing $p$ with
$p' = \left(p \setminus \{\partial_1,\partial_2\}\right) \cup \{\partial_3', \partial_4'\}$.
\end{compactitem}
See here:\setlength{\parskip}{\baselineskip} \\
\centerline{\psfig{file=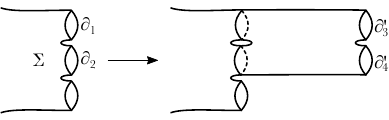,scale=100}}
In \S \ref{section:doubleboundarystabilization}, we will use the homological stability machine to prove the following.

\begin{proposition}
\label{proposition:doubleboundarystab}
Let $\mu$ be an $A$-homology marking on $(\Sigma,\cP) \in \PSurf$ that is supported on a symplectic subsurface.
Let $(\Sigma,\cP) \rightarrow (\Sigma',\cP')$
be a double boundary stabilization and let $\mu'$ be the stabilization
of $\mu$ to $(\Sigma',\cP')$.
Then the induced map $\HH_k(\Torelli(\Sigma,\cP,\mu)) \rightarrow \HH_k(\Torelli(\Sigma',\cP',\mu'))$
is an isomorphism if the genus of $\Sigma$ is at least $(\rank(A)+2)k + (2\rank(A)+2)$ and a surjection
if the genus of $\Sigma$ is $(\rank(A)+2)k + (2\rank(A)+1)$.
\end{proposition}

\p{Deriving Propositions \ref{proposition:increasingboundarystab} and \ref{proposition:decreasingboundarystab}}
As we said above, we will prove Proposition \ref{proposition:doubleboundarystab} in
\S \ref{section:doubleboundarystabilization}.  Here we will explain how to use
it to prove Propositions \ref{proposition:increasingboundarystab} and \ref{proposition:decreasingboundarystab}.

\begin{proof}[Proof of Proposition \ref{proposition:increasingboundarystab}, assuming Proposition \ref{proposition:doubleboundarystab}]
We start by recalling the statement.  Consider an increasing boundary stabilization
$(\Sigma,\cP) \rightarrow (\Sigma',\cP')$.  Let $\mu$ be an $A$-homology marking
on $(\Sigma,\cP)$ that is supported on a symplectic subsurface and let $\mu'$
be the stabilization of $\mu$ to $(\Sigma',\cP')$.
Assume that the genus of $\Sigma$ is at least $(\rank(A)+2)k + (2\rank(A)+2)$.
We must prove that the induced map $\HH_k(\Torelli(\Sigma,\cP,\mu)) \rightarrow \HH_k(\Torelli(\Sigma',\cP',\mu'))$
is an isomorphism.

The first observation is that the map $\Torelli(\Sigma,\cP,\mu) \rightarrow \Torelli(\Sigma',\cP',\mu')$ is split
injective via a splitting map $\Torelli(\Sigma',\cP',\mu') \rightarrow \Torelli(\Sigma,\cP,\mu)$ induced by gluing a disc
to one of the two components of $\partial \Sigma' \setminus \partial \Sigma$:\\
\centerline{\psfig{file=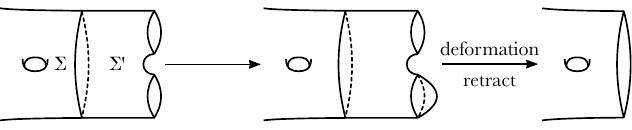,scale=100}}
The map $\HH_k(\Torelli(\Sigma,\cP,\mu)) \rightarrow \HH_k(\Torelli(\Sigma',\cP',\mu'))$ is thus injective, so
it is enough to prove that it is surjective.

Combining the fact that $\mu$ is supported on a symplectic subsurface with Corollary \ref{corollary:destabilizeone},
we see that $\mu$ is in fact supported on a symplectic subsurface of genus at most $\rank(A)$.  Since
the genus of $\Sigma$ is greater than $\rank(A)$, this implies that we can find a decreasing boundary stabilization
$(\Sigma'',\cP'') \rightarrow (\Sigma,\cP)$ and an $A$-homology marking $\mu''$
on $(\Sigma'',\cP'')$ that is supported on a symplectic subsurface such that
$\mu$ is the stabilization of $\mu''$ to $(\Sigma,\cP)$ and such that
the composition
\[(\Sigma'',\cP'') \rightarrow (\Sigma,\cP) \rightarrow (\Sigma',\cP')\]
is a double boundary stabilization; see here:\\
\centerline{\psfig{file=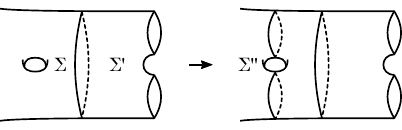,scale=100}}
The genus of $\Sigma''$ is one less than the genus of $\Sigma$, and thus at least $(\rank(A)+2)k + (2\rank(A)+1)$.
We can thus apply Proposition \ref{proposition:doubleboundarystab} to deduce that
the composition
\[\HH_k(\Torelli(\Sigma'',\cP'',\mu'')) \rightarrow \HH_k(\Torelli(\Sigma,\cP,\mu)) \rightarrow \HH_k(\Torelli(\Sigma',\cP',\mu'))\]
is surjective, and thus that the map $\HH_k(\Torelli(\Sigma,\cP,\mu)) \rightarrow \HH_k(\Torelli(\Sigma',\cP',\mu'))$
is surjective, as desired.
\end{proof}

\begin{proof}[Proof of Proposition \ref{proposition:decreasingboundarystab}, assuming Proposition \ref{proposition:doubleboundarystab}]
We start by recalling the statement.  Consider a decreasing boundary stabilization
$(\Sigma,\cP) \rightarrow (\Sigma',\cP')$.  Let $\mu$ be an $A$-homology marking
on $(\Sigma,\cP)$ that is supported on a symplectic subsurface and let $\mu'$
be the stabilization of $\mu$ to $(\Sigma',\cP')$.
Assume that the genus of $\Sigma$ is at least $(\rank(A)+2)k + (2\rank(A)+2)$.
We must prove that the induced map $\HH_k(\Torelli(\Sigma,\cP,\mu)) \rightarrow \HH_k(\Torelli(\Sigma',\cP',\mu'))$
is an isomorphism.

Let $\partial'$ be the component of $\partial \Sigma'$ that is not a component of $\partial \Sigma$.  As in the
following picture, we can construct an increasing boundary stabilization
$(\Sigma',\cP') \rightarrow (\Sigma'',\cP'')$ that attaches a $3$-holed torus to $\partial'$ such that
the composition
\[(\Sigma,\cP) \rightarrow (\Sigma',\cP') \rightarrow (\Sigma'',\cP'')\]
is a double boundary stabilization:\\
\centerline{\psfig{file=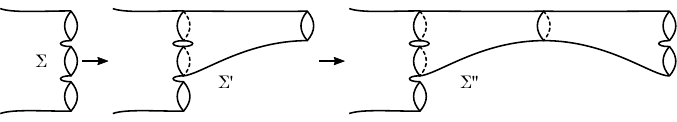,scale=100}}
Let $\mu''$ be the stabilization of $\mu'$ to $(\Sigma'',\cP'')$.  We then have maps
\begin{equation}
\label{eqn:thedouble}
\HH_k(\Torelli(\Sigma,\cP,\mu)) \rightarrow \HH_k(\Torelli(\Sigma',\cP',\mu')) \rightarrow \HH_k(\Torelli(\Sigma'',\cP'',\mu'')).
\end{equation}
Proposition \ref{proposition:doubleboundarystab} implies that the composition \eqref{eqn:thedouble} is an isomorphism,
and Proposition \ref{proposition:increasingboundarystab} implies that the map 
$\HH_k(\Torelli(\Sigma',\cP',\mu')) \rightarrow \HH_k(\Torelli(\Sigma'',\cP'',\mu''))$ is an isomorphism.  We conclude
that the map
$\HH_k(\Torelli(\Sigma,\cP,\mu)) \rightarrow \HH_k(\Torelli(\Sigma',\cP',\mu'))$ is an isomorphism, as desired.
\end{proof}

\section{Double boundary stabilization}
\label{section:doubleboundarystabilization}

Adapting an argument due to Hatcher--Vogtmann \cite{HatcherVogtmannTethers}, we will prove
Proposition \ref{proposition:doubleboundarystab} by studying a complex of ``order-preserving
double-tethered loops'' whose vertex-stabilizers yield double boundary stabilizations:\\
\centerline{\psfig{file=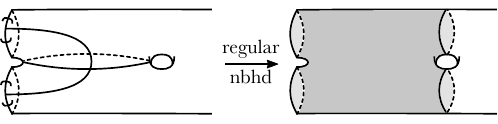,scale=100}}
We will require that the homology classes of both the loop and the ``double-tether''
arc vanish under the homology marking.  Getting the loop to vanish will be an easy
variant on the argument we used for vanishing surfaces in \S 
\ref{section:vanishsubsurfacescon}, but getting the double-tether to vanish
is harder and will require new ideas.  We will build up the complex in three
stages (tethered vanishing loops, then double-tethered vanishing loops, and then
finally order-preserving double-tethered vanishing loops)
in \S \ref{section:tetheredloops}--\ref{section:orderdoubletetheredloops}.
These five sections are preceded by two technical sections: \S \ref{section:symplecticsupport}
gives a necessary and sufficient condition for an $A$-homology marking to be supported
on a symplectic subsurface, and \S \ref{section:destabilize} is
about destabilizing $A$-homology markings.  After all this is complete,
we prove Proposition \ref{proposition:doubleboundarystab} in \S \ref{section:doubleboundarystabproof}.

\subsection{Identifying markings supported on a symplectic subsurface}
\label{section:symplecticsupport}

Consider some $(\Sigma,\cP) \in \PSurf$.  In this section, we give a necessary and sufficient
condition for an $A$-homology marking on $(\Sigma,\cP)$ to be supported on a symplectic subsurface.
This requires some preliminary definitions (which will also be used in later sections).

\p{Intersection map}
Let $q$ be a finite set of oriented simple closed curves on $\Sigma$ and let
$\Z[q]$ be the set of formal $\Z$-linear combinations of elements of $q$.  Define
the {\em $q$-intersection map} to be the map
$\isect_q\colon \HH_1^{\cP}(\Sigma,\partial \Sigma) \rightarrow \Z[q]$ defined as follows.
Let
\[\omega_{\Sigma} \colon \HH_1(\Sigma,\partial \Sigma) \times \HH_1(\Sigma) \rightarrow \Z.\]
be the algebraic intersection pairing.
For $x \in \HH_1^{\cP}(\Sigma,\partial \Sigma)$, we then set
\[\isect_q(x) = \sum_{\gamma \in q} \omega_{\Sigma}(x,[\gamma]) \cdot \gamma.\]

\p{Total boundary map}
For a set $q$ as above, define
\[\RZ[q] = \Set{$\sum_{\gamma \in q} c_{\gamma} \cdot \gamma \in \Z[q]$}{$\sum_{\gamma \in q} c_{\gamma} = 0$}.\]
Consider $p \in \cP$.  Each boundary component $\partial \in p$ is a simple closed curve on $\Sigma$, and
the orientation on $\Sigma$ induces an orientation on $\partial$ such that $\Interior(\Sigma)$ lies to the left
of $\partial$.  We thus have the map $\isect_p \colon \HH_1^{\cP}(\Sigma,\partial \Sigma) \rightarrow \Z[p]$.
Since $\HH_1^{\cP}(\Sigma,\partial \Sigma)$ is generated
by the homology classes of oriented loops and arcs connecting $\cP$-adjacent boundary
components, the image of $\isect_p$ is $\RZ[p]$.  Define
\[\RZ_{\cP} = \bigoplus_{p \in \cP} \RZ[p].\]
The {\em total boundary map} of $(\Sigma,\cP)$ is the map
$\isect_{\cP}\colon \HH_1^{\cP}(\Sigma,\partial \Sigma) \rightarrow \RZ_{\cP}$
obtained by taking the direct sum of all the $\isect_p$ for $p \in \cP$.

\begin{remark}
Each $\RZ[p]$ naturally lies in $\RH_0(\partial \Sigma)$, and the total boundary
map can be identified with the restriction to $\HH_1^{\cP}(\Sigma,\partial \Sigma)$
of the usual boundary map $\HH_1(\Sigma,\partial \Sigma) \rightarrow \RH_0(\partial \Sigma)$.
\end{remark}

\p{Symplectic support}
Now consider an $A$-homology marking $\mu$ on $(\Sigma,\cP)$.
Back in Remark \ref{remark:nonsupport}, we observed that a necessary condition
for $\mu$ to be supported on a symplectic subsurface is that
$\isect_{\cP}(\ker(\mu)) = \RZ_{\cP}$.  The following lemma
says that this condition is also sufficient:

\begin{lemma}
\label{lemma:identifysupport}
Let $\mu$ be an $A$-homology marking on $(\Sigma,\cP) \in \PSurf$.  Then
$\mu$ is supported on a symplectic subsurface if and only if
$\isect_{\cP}(\ker(\mu)) = \RZ_{\cP}$.
\end{lemma}
\begin{proof}
The nontrivial direction is that if $\isect_{\cP}(\ker(\mu)) = \RZ_{\cP}$, then
$\mu$ is supported on a symplectic subsurface, so that is what we prove.  Write
\[\cP = \{\{\partial_1^1,\ldots,\partial_{k_1}^1\},\{\partial_1^2,\ldots,\partial_{k_2}^2\},\ldots,\{\partial_1^n,\ldots,\partial_{k_n}^n\}\}.\]
Below we will prove that for all $1 \leq i \leq n$ and $1 \leq j < k_i$, we can find embedded arcs $\alpha_{ij}$ with
the following properties:
\setlength{\parskip}{0pt}
\begin{compactitem}
\item $\alpha_{ij}$ connects $\partial^i_j$ to $\partial^i_{j+1}$, and
\item the $\alpha_{ij}$ are pairwise disjoint, and
\item $\mu([\alpha_{ij}]) = 0$ for all $i$ and $j$.
\end{compactitem}
Letting $g$ be the genus of $\Sigma$, we can then find a subsurface $\Sigma'$ of $\Sigma$ that is homeomorphic
to $\Sigma_g^1$ such that $\Sigma'$ is disjoint from $\partial \Sigma$ and the $\alpha_{ij}$; see here:\\
\centerline{\psfig{file=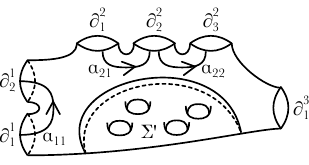,scale=100}}
Let $\cP' = \{\partial \Sigma'\}$, so $(\Sigma',\cP') \rightarrow (\Sigma,\cP)$
is a $\PSurf$-morphism.  It is easy to see that we can find an $A$-homology
marking $\mu'$ on $(\Sigma',\cP')$ such that $\mu$ is the stabilization of
$\mu'$ to $(\Sigma,\cP)$ (see Lemma \ref{lemma:destabilizemarking} below for a
more general result that implies this).
The lemma follows.\setlength{\parskip}{\baselineskip}

It remains to find the $\alpha_{ij}$.  The assumptions in the lemma imply that for
$1 \leq i \leq n$ and $1 \leq j < k_i$ we can find arcs $\alpha_{ij}$
(not necessarily embedded or pairwise disjoint) with the following properties:
\setlength{\parskip}{0pt}
\begin{compactitem}
\item $\alpha_{ij}$ connects $\partial^i_j$ to $\partial^i_{j+1}$, and
\item $\mu([\alpha_{ij}]) = 0$ for all $i$ and $j$.
\end{compactitem}
Homotoping the $\alpha_{ij}$, we can assume that their endpoints are disjoint from each other, their interiors
lie in the interior of $\Sigma$, and all intersections and self-intersections are transverse.  Choose these
$\alpha_{ij}$ so as to minimize the number of intersections and self-intersections.  We claim that the $\alpha_{ij}$
are then all embedded and pairwise disjoint from each other.  Assume otherwise.  Let $\alpha_{i_0,j_0}$ be
the first element of the ordered list
\[\alpha_{11},\alpha_{12},\ldots,\alpha_{1,k_1-1},\alpha_{21},\ldots,\alpha_{2,k_2-1},\alpha_{31},\ldots,\alpha_{n,k_n-1}\]
that intersects either itself or one of the other $\alpha_{ij}$.  As in the following picture, we can then
``slide'' the first intersection of $\alpha_{i_0,j_0}$ off of the union of the $\partial^{i_0}_j$ and
$\alpha_{i_0,j}$ with $j \leq j_0$:\\
\centerline{\psfig{file=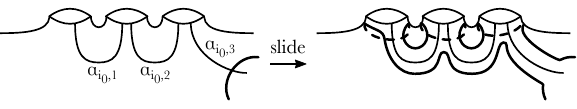,scale=100}}
Since the homology classes of all the $\partial^i_j$ are in the kernel of $\mu$, this does not change the value
of any of the $\mu([\alpha_{ij}])$, but it does eliminate one of the intersections, contradicting the minimality
of this number.\setlength{\parskip}{\baselineskip}
\end{proof}

\subsection{Destabilizing homology-marked partitioned surfaces}
\label{section:destabilize}

Consider $(\Sigma,\cP) \in \PSurf$.  This section is devoted to ``destabilizing'' $A$-homology
markings on $(\Sigma,\cP)$ to subsurfaces.

\p{Existence}
Let $\mu$ be an $A$-homology marking on $(\Sigma,\cP)$ and let
$(\Sigma',\cP') \rightarrow (\Sigma,\cP)$ be a $\PSurf$-morphism.
One obvious necessary condition for there to exist an $A$-homology marking $\mu'$
on $(\Sigma',\cP')$ whose stabilization to $(\Sigma,\cP)$ is $\mu$ is that
$\mu$ must vanish on elements of $\HH_1^{\cP}(\Sigma,\partial \Sigma)$ supported
on $\Sigma \setminus \Sigma'$.  This condition is also sufficient:

\begin{lemma}
\label{lemma:destabilizemarking}
Let $\mu$ be an $A$-homology marking on $(\Sigma,\cP) \in \PSurf$ and let
$(\Sigma',\cP') \rightarrow (\Sigma,\cP)$ be a $\PSurf$-morphism.  Then 
there exists an $A$-homology marking $\mu'$ on $(\Sigma',\cP')$ whose stabilization
to $(\Sigma,\cP)$ is $\mu$ if and only if 
$\mu(x)=0$ for all $x \in \HH_1^{\cP}(\Sigma,\partial \Sigma)$
supported on $\Sigma \setminus \Sigma'$.  
\end{lemma}
\begin{proof}
The nontrivial assertion here is that if $\mu(x)=0$ for all $x \in \HH_1^{\cP}(\Sigma,\partial \Sigma)$
supported on $\Sigma \setminus \Sigma'$, then $\mu'$ exists, so this is what we prove.
Let $\iota\colon (\Sigma',\cP') \rightarrow (\Sigma,\cP)$ be the inclusion.
We want to show that $\mu\colon \HH_1^{\cP}(\Sigma,\partial \Sigma) \rightarrow A$ factors
through
\[\iota^{\ast}\colon \HH_1^{\cP}(\Sigma,\partial \Sigma) \rightarrow \HH_1^{\cP'}(\Sigma',\partial \Sigma').\]
The cokernel of $\iota^{\ast}$ is obviously free abelian.  It is thus enough to prove that
$\mu$ vanishes on $\ker(\iota^{\ast})$.  To do this, we will show that
$\ker(\iota^{\ast})$ is generated by elements supported on $\Sigma \setminus \Sigma'$.
The map $\iota^{\ast}$ is the restriction to $\HH_1^{\cP}(\Sigma,\partial \Sigma)$ of
the composition
\[\HH_1\left(\Sigma,\partial \Sigma\right) 
\stackrel{f}{\longrightarrow} \HH_1\left(\Sigma,\Sigma \setminus \Interior\left(\Sigma'\right)\right)
\stackrel{\cong}{\longrightarrow} \HH_1(\Sigma',\partial \Sigma').\]
It is thus enough to show that all elements of $\ker(f)$ are supported on $\Sigma \setminus \Sigma'$.
The long exact sequence in homology for the triple $(\Sigma,\Sigma \setminus \Interior(\Sigma'), \partial \Sigma)$
implies that $\ker(f)$ is generated by the image of
\[\HH_1\left(\Sigma \setminus \Interior\left(\Sigma'\right), \partial \Sigma\right) \longrightarrow \HH_1\left(\Sigma,\partial \Sigma\right).\]
The desired result follows.
\end{proof}

\p{$\cP$-simple subsurfaces}
We now study when destabilizations of markings supported on symplectic subsurfaces are supported on symplectic
subsurfaces.  Rather than prove the most general result possible, we will focus on the case of
{\em $\cP$-simple subsurfaces} of $\Sigma$, which
are subsurfaces $\Sigma'$ satisfying the following conditions (see Example \ref{example:psimple} below):
\setlength{\parskip}{0pt}
\begin{compactitem}
\item $\Sigma'$ is connected.
\item The closure $S$ of $\Sigma \setminus \Sigma'$ is connected.
\item The set of components of $\partial S$ can be partitioned into two disjoint
nonempty subsets $q$ and $q'$ as follows:
\begin{compactitem}
\item The elements of $q'$ all lie in the interior of $\Sigma$, and are thus components of
$\partial \Sigma'$.  These will be called
the {\em interior boundary components}. 
\item The elements of $q$ are components of $\partial \Sigma \setminus \partial \Sigma'$ lying in a single $p \in \cP$.  These will be called the {\em exterior boundary components}.
\end{compactitem}
\end{compactitem}
Given a $\cP$-simple subsurface $\Sigma'$ of $\Sigma$, the {\em induced
partition} $\cP'$ of the components of $\partial \Sigma'$ is obtained
from $\cP$ by replacing $p$ with $(p \setminus q) \cup q'$, where $p$ and
$q$ and $q'$ are as above.  The map $(\Sigma',\cP') \rightarrow (\Sigma,\cP)$
is clearly a $\PSurf$-morphism. \setlength{\parskip}{\baselineskip}

\begin{example}
\label{example:psimple}
Let $\Sigma = \Sigma_8^5$ and 
$\cP = \{\{\partial_1,\partial_2,\partial_3\},\{\partial_4,\partial_5\}\}$.
Consider the following shaded subsurface $\Sigma'$ of $\Sigma$:\\
\centerline{\psfig{file=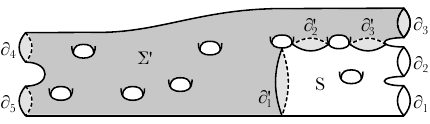,scale=100}}
The subsurface $\Sigma'$ is a $\cP$-simple subsurface with 
interior boundary components $q'=\{\partial'_1,\partial'_2,\partial'_3\}$,
exterior boundary components
$q=\{\partial_1,\partial_2\}$, and induced
partition $\cP' = \{\{\partial_1',\partial_2',\partial_3',\partial_3\},\{\partial_4,\partial_5\}\}$.
\end{example}

\p{Closed markings and intersection maps}
We now introduce some notation needed to state our result.
Let $(\Sigma,\cP) \in \PSurf$ and let $\mu$ be an $A$-homology marking on $(\Sigma,\cP)$.  
The associated {\em closed marking}
on $(\Sigma,\cP)$ is the map $\hmu\colon \HH_1(\Sigma) \rightarrow A$ defined via the composition
\[\HH_1(\Sigma) \longrightarrow \HH_1^{\cP}(\Sigma,\partial \Sigma) \stackrel{\mu}{\longrightarrow} A.\]
Also, for a finite set $q$ of oriented simple closed curves on $\Sigma$, define
the {\em closed $q$-intersection map} to be the map 
$\hisect_q\colon \HH_1(\Sigma) \rightarrow \Z[q]$ defined via the composition
\[\HH_1(\Sigma) \longrightarrow \HH_1^{\cP}(\Sigma,\partial \Sigma) \stackrel{\isect_q}{\longrightarrow} \Z[q].\]
If the elements of $q$ are disjoint and their union bounds a subsurface on one side
(with respect to the orientations on the curves of $q$), then the image of
$\hisect_q$ lies in $\RZ[q]$.

\p{Destabilizing and symplectic support}
With the above notation, we have the following lemma.

\begin{lemma}
\label{lemma:destabilizesymplectic}
Let $\mu$ be an $A$-homology marking on $(\Sigma,\cP)$ that is supported 
on a symplectic subsurface.  Let $\Sigma'$ be a $\cP$-simple subsurface of $\Sigma$
with induced partition $\cP'$ and let $\mu'$ be an $A$-homology marking on $(\Sigma',\cP')$
whose stabilization to $(\Sigma,\cP)$ is $\mu$.  Assume the following:
\setlength{\parskip}{0pt}
\begin{compactitem}
\item Let $q'$ be the interior boundary components of $\Sigma'$ and let
$\hmu\colon \HH_1(\Sigma) \rightarrow A$ be the closed marking associated to $\mu$.
Orient each $\partial' \in q'$ such 
that $\Sigma'$ lies to its left.  Then $\hisect_{q'}(\ker(\hmu)) = \RZ[q']$.
\end{compactitem}
Then $\mu'$ is supported on a symplectic subsurface.
\end{lemma}
\begin{proof}
By Lemma \ref{lemma:identifysupport}, we must prove that the map
\[\isect_{\cP'}\colon \HH_1^{\cP'}(\Sigma',\partial \Sigma') \longrightarrow \RZ_{\cP'}\]
takes $\ker(\mu')$ onto $\RZ_{\cP'}$.  Below we will prove two facts:
\setlength{\parskip}{0pt}
\begin{compactitem}
\item $\RZ[q'] \subset \isect_{\cP'}(\ker(\mu'))$.
\item Letting $\iota\colon (\Sigma',\cP') \rightarrow (\Sigma,\cP)$ be
the inclusion and
$\iota^{\ast}\colon \HH_1^{\cP}(\Sigma,\partial \Sigma) \rightarrow \HH_1^{\cP'}(\Sigma',\partial \Sigma')$
be the induced map, there exists a surjection 
$\beta\colon \RZ_{\cP} \twoheadrightarrow \RZ_{\cP'} / \RZ[q']$
such that the diagram
\begin{equation}
\label{eqn:commutativediagram}
\begin{gathered}
\xymatrix{
\HH_1^{\cP}(\Sigma,\partial \Sigma) \ar[rr]^-{\isect_{\cP}} \ar[d]^-{\iota^{\ast}} & & \RZ_{\cP} \ar@{->>}[d]^-{\beta} \\
\HH_1^{\cP'}(\Sigma',\partial \Sigma') \ar[r]^-{\isect_{\cP'}} & \RZ_{\cP'} \ar@{->>}[r]^-{\pi} & \RZ_{\cP'}/\RZ[q']}
\end{gathered}
\end{equation}
commutes.
\end{compactitem}
Assume these facts for the moment.  Since $\RZ[q'] \subset \isect_{\cP'}(\ker(\mu'))$,
to prove that $\isect_{\cP'}(\ker(\mu')) = \RZ_{\cP'}$ 
it is enough to prove that $\pi(\isect_{\cP'}(\ker(\mu'))) = \RZ_{\cP'}/\RZ[q']$.  
Since $\mu$ is supported on a symplectic subsurface, Lemma \ref{lemma:identifysupport}
says that $\isect_{\cP}(\ker(\mu)) = \RZ_{\cP}$, so \setlength{\parskip}{\baselineskip}
\begin{equation}
\label{eqn:bigequality}
\pi(\isect_{\cP'}(\iota^{\ast}(\ker(\mu)))) = \beta(\isect_{\cP}(\ker(\mu))) = \beta(\RZ_{\cP}) = \RZ_{\cP'}/\RZ[q'].
\end{equation}
Since $\mu$ is the stabilization of $\mu'$ to $(\Sigma,\cP)$, by definition we
have $\mu = \mu' \circ \iota^{\ast}$, so $\iota^{\ast}(\ker(\mu)) \subset \ker(\mu')$.
Plugging this into \eqref{eqn:bigequality}, we get that
\[\pi(\isect_{\cP'}(\ker(\mu')) = \RZ_{\cP'}/\RZ[q'],\]
as desired.

It remains to prove the above two facts.  We start with the first.  
Since elements of $\HH_1(\Sigma)$ can be represented by cycles that are
disjoint from all components of $\partial \Sigma$, the image of the composition
\[\HH_1(\Sigma) \rightarrow \HH_1^{\cP}(\Sigma,\partial \Sigma) \stackrel{\iota^{\ast}}{\longrightarrow} \HH_1^{\cP'}(\Sigma',\partial \Sigma') \stackrel{\isect_{\cP'}}{\longrightarrow} \RZ_{\cP'}\]
must lie in $\RZ[q'] \subset \RZ_{\cP'}$.  From its definition, it is clear that this
composition in fact equals $\hisect_{q'}$.  Our hypothesis about $\hisect_{q'}$ 
thus implies that 
\[\RZ[q'] \subset \isect_{\cP'}(\iota^{\ast}(\ker(\mu))) \subset \isect_{\cP'}(\ker(\mu')),\]
as desired.  Here we are using the fact (already observed in the previous paragraph) that
$\iota^{\ast}(\ker(\mu)) \subset \ker(\mu')$.

We now construct $\beta\colon \RZ_{\cP} \twoheadrightarrow \RZ_{\cP'} / \RZ[q']$.
Let $q$ be the exterior boundary 
components of $\Sigma'$.  Write $\cP = \{p_1,\ldots,p_k\}$ with $q \subset p_1$.  
Setting $p_1' = (p_1 \setminus q) \cup q'$, we then have $\cP' = \{p_1',p_2,\ldots,p_k\}$.
Thus
\[\RZ_{\cP} = \RZ[p_1] \oplus \bigoplus_{i=2}^k \RZ[p_i] \quad \text{and} \quad
\RZ_{\cP'}/\RZ[q'] = \RZ[p_1']/\RZ[q'] \oplus \bigoplus_{i=2}^k \RZ[p_i].\]
On the $\RZ[p_i]$ summand for $2 \leq i \leq k$, the map $\beta$ is the identity.
On the $\RZ[p_1]$ summand, the map $\beta$ is the restriction to $\RZ[p_1]$ of
the map
\[\Z[p_1] =\Z[p_1 \setminus q] \oplus \Z[q] \rightarrow \Z[p_1 \setminus q] \oplus \Z[q']/\RZ[q'] = \RZ[p_1']/\RZ[q']\]
that is the identity on $\Z[p_1 \setminus q]$ and takes every element of $q$ to the
generator of $\Z[q']/\RZ[q'] \cong \Z$.  This map $\beta$ is clearly a surjection.  The fact
that \eqref{eqn:commutativediagram} commutes follows from the fact that an arc in
$\Sigma$ from a component $\partial_1$ of $\partial \Sigma$ to a component $\partial_2$ of
$\partial \Sigma$ with $\partial_1$ and $\partial_2$ both lying in some $p_i$ has
the following algebraic intersection number with the union of the components of
$q$:
\setlength{\parskip}{0pt}
\begin{compactitem}
\item $0$ if $i \geq 2$ or if $i=1$ and $\partial_1,\partial_2 \in p_1 \setminus q$ or if
$i=1$ and $\partial_1,\partial_2 \in q$.
\item $1$ if $i=1$ and $\partial_1 \in p_1 \setminus q$ and $\partial_2 \in q$.
\item $-1$ if $i=1$ and $\partial_2 \in q$ and $\partial_1 \in p_1 \setminus q$.
\end{compactitem}
The reason for this is that each time the arc crosses from $\Sigma'$ to $S$, it adds $+1$ to its
total intersection with $q$, while each time it crosses from $S$ to $\Sigma'$ it adds $-1$ to its total
intersection with $q$.  See the following figure,
where $\Sigma$ is shaded and $S$ is unshaded:\setlength{\parskip}{\baselineskip}\\
\centerline{\psfig{file=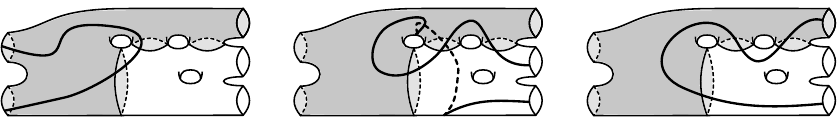,scale=100}}
The lemma follows.
\end{proof}

\subsection{The complex of tethered vanishing loops}
\label{section:tetheredloops}

We now begin our long trek to the complex of order-preserving double-tethered vanishing loops, starting with
the complex of tethered vanishing loops.  The definition takes several steps. 

\p{Tethered loops}
Define $\tau(S^1)$ to be the result of gluing $1 \in [0,1]$ to
a point of $S^1$.  The subset $[0,1] \in \tau(S^1)$
is the {\em tether} and $0 \in [0,1] \subset \tau(S^1)$ is the {\em initial point}
of the tether.  For a surface $\Sigma \in \Surf$ and a finite disjoint union
of open intervals $I \subset \partial \Sigma$, an {\em $I$-tethered loop} in
$\Sigma$ is an embedding $\iota\colon \tau(S^1) \rightarrow \Sigma$ with
the following two properties:
\setlength{\parskip}{0pt}
\begin{compactitem}
\item $\iota$ takes the initial point of the tether to a point of $I$, and
\item orienting $\iota(S^1)$ using the natural orientation of $S^1$, the
image $\iota([0,1])$ of the tether approaches $\iota(S^1)$ from its right.
\end{compactitem}
\setlength{\parskip}{\baselineskip}

\p{Complex of tethered loops}
For a surface $\Sigma \in \Surf$ and a finite disjoint union of open intervals
$I \subset \partial \Sigma$, the {\em complex of $I$-tethered loops}
on $\Sigma$, denoted $\TL(\Sigma,I)$, is the simplicial complex
whose $k$-simplices are collections $\{\iota_0,\ldots,\iota_k\}$
of isotopy classes of $I$-tethered loops on $\Sigma$ that can be
realized so as to be disjoint and not separate $\Sigma$:\\
\centerline{\psfig{file=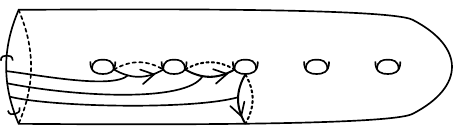,scale=100}}
This complex was introduced by
Hatcher--Vogtmann \cite{HatcherVogtmannTethers}, who proved that if $\Sigma$
has genus $g$ then $\TL(\Sigma,I)$ is $\frac{g-3}{2}$-connected (see \cite[Proposition 5.1]{HatcherVogtmannTethers}).  

\p{Complex of tethered vanishing loops}
Let $\mu$ be an $A$-homology marking on $(\Sigma,\cP) \in \PSurf$ and let $I \subset \partial \Sigma$
be a finite disjoint union of open intervals.  Define $\TL(\Sigma,I,\cP,\mu)$ to be the
subcomplex of $\TL(\Sigma,I)$ consisting of $k$-simplices
$\{\iota_0,\ldots,\iota_k\}$ satisfying the following conditions.
For $0 \leq i \leq k$, let $\gamma_i$ be the oriented loop
$(\iota_i)|_{S^1}$.  Set $\Gamma = \{\gamma_0,\ldots,\gamma_k\}$.  As 
in \S \ref{section:destabilize}, let $\hmu\colon \HH_1(\Sigma) \rightarrow A$
be the closed marking associated to $\mu$ and let
$\hisect_{\Gamma}\colon \HH_1(\Sigma) \rightarrow \Z[\Gamma]$ be the closed
$\Gamma$-intersection map.  We then require that $\hmu([\gamma_i])=0$ for all
$0 \leq i \leq k$ and that $\hisect_{\Gamma}(\ker(\hmu)) = \Z[\Gamma]$. 

\begin{remark}
This last condition might seem a little unmotivated, but is needed
to ensure that the stabilizer of our simplex is supported on a symplectic
subsurface (at least in favorable situations).  It 
clearly always holds when $\mu$ is supported on
a symplectic subsurface that is disjoint from the images of all
the $\iota_i$.  This is best illustrated by an example:\\
\centerline{\psfig{file=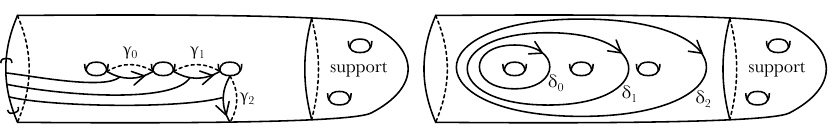,scale=100}}
If $\Gamma = \{\gamma_0,\gamma_1,\gamma_2\}$ and $\delta_0,\delta_1,\delta_2$ are as shown,
then $\hmu([\delta_i]) = 0$ and $\hisect_{\Gamma}([\delta_i]) = \gamma_i$ for $0 \leq i \leq 2$,
which implies that $\hisect_{\Gamma}(\ker(\hmu)) = \Z[\Gamma]$.
\end{remark}

\p{High connectivity}
Our main topological theorem about $\TL(\Sigma,I,\cP,\mu)$ is as follows.

\begin{theorem}
\label{theorem:vanishtetheredloopcon}
Let $\mu$ be an $A$-homology marking on $(\Sigma,\cP) \in \PSurf$, let
$I \subset \partial \Sigma$ be a finite disjoint union of open intervals, and
let $g$ be the genus of $\Sigma$.  Then
$\TL(\Sigma,I,\cP,\mu)$ is $\frac{g-(2\rank(A)+3)}{\rank(A)+2}$-connected.
\end{theorem}
\begin{proof}
The proof is very similar to that of Theorem \ref{theorem:vanishsurfacescon}.
We start by defining an auxiliary space.  Let $X$ be the simplicial complex
whose vertices are the union of the vertices of the spaces
$\TL(\Sigma,I,\cP,\mu)$ and $\TS_{\rank(A)+1}(\Sigma,I)$ and whose 
simplices are collections $\sigma$ of vertices with the following properties:
\setlength{\parskip}{0pt}
\begin{compactitem}
\item The vertices in $\sigma$ (which are embeddings of either $\tau(S^1)$
or $\tau(\Sigma_{\rank(A)+1}^1)$ into $\Sigma$) can be homotoped such that their
images are disjoint and do not separate $\Sigma$.
\item Let $\sigma' \subset \sigma$ be the subset consisting of vertices
of $\TL(\Sigma,I,\cP,\mu)$.  Then $\sigma'$ is a simplex of $\TL(\Sigma,I,\cP,\mu)$.
\end{compactitem}
Both $\TL(\Sigma,I,\cP,\mu)$ and $\TS_{\rank(A)+1}(\Sigma,I)$ are subcomplexes
of $X$.\setlength{\parskip}{\baselineskip}

By Theorem \ref{maintheorem:subsurfacescon}, the subcomplex
$\TS_{\rank(A)+1}(\Sigma,I)$ of $X$ is $\frac{g-(2\rank(A)+3)}{\rank(A)+2}$-connected.
An argument using Corollary \ref{corollary:avoidsubcomplex}
identical to the one in the proof of Theorem \ref{theorem:vanishsurfacescon}
shows that this implies that $X$ is $\frac{g-(2\rank(A)+3)}{\rank(A)+2}$-connected.
As in the proof of Theorem \ref{theorem:vanishsurfacescon}, this implies
that it is enough to construct a retraction $r\colon X \rightarrow \TL(\Sigma,I,\cP,\mu)$.

For a vertex $\iota$ of $X$, we define $r(\iota)$ as follows.  If $\iota$ is a vertex
of $\TL(\Sigma,I,\cP,\mu)$, then $r(\iota) = \iota$.  If instead
$\iota$ is a vertex of $\TS_{\rank(A)+1}(\Sigma,I)$, then we do the following.  
Let $\hmu\colon \HH_1(\Sigma) \rightarrow A$ be the closed marking associated
to $\mu$.  Define $\mu'\colon \HH_1(\Sigma_{\rank(A)+1}^1) \rightarrow A$ to be
the composition
\[\HH_1(\Sigma_{\rank(A)+1}^1) \cong \HH_1(\tau(\Sigma_{\rank(A)+1}^1)) \stackrel{\iota_{\ast}}{\longrightarrow} \HH_1(\Sigma) \stackrel{\hmu}{\longrightarrow} A.\]
Proposition \ref{proposition:findvanish} implies that there exists
a subsurface $S \subset \Sigma_{\rank(A)+1}^1$ with $S \cong \Sigma_1^1$ and $\mu'|_{\HH_1(S)} = 0$.
Let $\alpha$ be a nonseparating oriented simple closed curves in $S$.  Define $r(\iota)$ to be the vertex of
$\TL(\Sigma,I,\cP,\mu)$ obtained by adjoining the tether of $\iota$ and an
arbitrary arc in $\iota(\Sigma_{\rank(A)+1}^1)$ to $\iota(\alpha)$; see here:\\
\centerline{\psfig{file=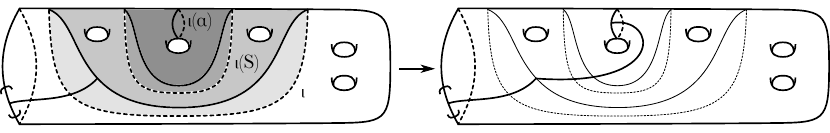,scale=100}}
To see that this is actually a vertex of $\TL(\Sigma,I,\cP,\mu)$, observe that
by construction we have 
\[\hmu([\iota(\alpha)]) = 0 \quad \text{and} \quad \iota_{\ast}(\HH_1(S)) \subset \ker(\hmu) \quad \text{and} \quad \hisect_{\{\iota(\alpha)\}}(\iota_{\ast}(\HH_1(S))) = \iota(\alpha).\]
Of course, $r(\iota)$ depends on various choices, but we simply make an arbitrary choice.

To complete the proof, we must prove that $r$ extends over the simplices of $X$.  Let
$\sigma$ be a simplex of $X$.  Enumerate the vertices of $\sigma$ as
$\{\iota_0,\ldots,\iota_k,\iota'_0,\ldots,\iota'_{\ell}\}$, where
the $\iota_i$ are vertices of $\TL(\Sigma,I,\cP,\mu)$ and the $\iota'_j$ are vertices of $\TS_{\rank(A)+1}(\Sigma,I)$.
We must prove that
\[r(\sigma) = \{\iota_0,\ldots,\iota_k,r(\iota'_0),\ldots,r(\iota'_{\ell})\}\]
is a simplex of $\TL(\Sigma,I,\cP,\mu)$.  The images of the vertices in $r(\sigma)$
can clearly be homotoped so as to be disjoint, so the only thing we must prove is the following.
For $0 \leq i \leq k$ and $0 \leq j \leq \ell$, let
$\gamma_i = \iota_i|_{S^1}$ and $\gamma'_j = \iota'_j|_{S^1}$.  Setting 
$\Gamma = \{\gamma_0,\ldots,\gamma_k,\gamma'_0,\ldots,\gamma'_{\ell}\}$, we have
to show that $\hisect_{\Gamma}(\ker(\hmu)) = \Z[\Gamma]$.  Setting
$\Gamma_1 = \{\gamma_0,\ldots,\gamma_k\}$ and $\Gamma_2 = \{\gamma'_0,\ldots,\gamma'_{\ell}\}$,
we will show that $\Gamma_1$ and $\Gamma_2$ are both contained in $\hisect_{\Gamma}(\ker(\hmu))$.

We start with $\Gamma_2$.  By construction, for $0 \leq j \leq \ell$ there exists a subsurface $S_j$ of $\Sigma$ with
$S_j \cong \Sigma_1^1$ such that the following hold:
\setlength{\parskip}{0pt}
\begin{compactitem}
\item $\gamma'_j \subset S_j$, and
\item the $S_j$ are disjoint from each other and from all the $\gamma_i$, and
\item regarding $\HH_1(S_j)$ as a subgroup of $\HH_1(\Sigma)$, we have $\HH_1(S_j) \subset \ker(\hmu)$.
\end{compactitem}
Since $\hisect_{\Gamma}(\HH_1(S_j)) = \gamma'_j$, we have $\gamma'_j \in \hisect_{\Gamma}(\ker(\hmu))$, as
desired.\setlength{\parskip}{\baselineskip}

It remains to show that $\Gamma_1 \subset \hisect_{\Gamma}(\ker(\hmu))$.  Since $\{\iota_0,\ldots,\iota_k\}$
is a simplex of $\TL(\Sigma,I,\cP,\mu)$, by definition we have $\hisect_{\Gamma_1}(\ker(\hmu)) = \Z[\Gamma_1]$.
For some $0 \leq i \leq k$, let $x \in \ker(\hmu)$ be such that $\hisect_{\Gamma_1}(x)  = \gamma_i$.
We then have $\hisect_{\Gamma}(x) = \gamma_i + z$ with $z \in \Z[\Gamma_2]$.  Since we already showed that
$\Z[\Gamma_2] \subset \hisect_{\Gamma}(\ker(\hmu))$, we conclude that $\gamma_i \in \hisect_{\Gamma}(\ker(\hmu))$, 
as desired.
\end{proof}

\subsection{The complex of double-tethered vanishing loops}
\label{section:doubletetheredloops}

The definition of the complex of double-tethered vanishing loops takes several steps.

\p{Double-tethered loops}
Define $\tau^2(S^1)$ to be the result of gluing $1 \in [0,2]$ to $S^1$.
We will call $[0,2] \subset \tau^2(S^1)$ the {\em double tether}; the point
$0 \in [0,2]$ is the double tether's {\em initial point} and $2 \in [0,2]$ is its
{\em terminal point}.  For a surface $\Sigma \in \Surf$
and finite disjoint unions of open intervals $I,J \subset \partial \Sigma$
with $I \cap J = \emptyset$, an {\em $(I,J)$-double-tethered loop} in
$\Sigma$ is an embedding $\iota\colon \tau^2(S^1) \rightarrow \Sigma$ with the following two
properties:
\setlength{\parskip}{0pt}
\begin{compactitem}
\item $\iota$ takes the initial point of the double tether to a point of $I$
and the terminal point of the double tether to a point of $J$, and
\item orienting $\iota(S^1)$ using the natural orientation of $S^1$, the
image $\iota([0,1])$ approaches $\iota(S^1)$ from its right and
the image $\iota([1,2])$ leaves $\iota(S^1)$ from its left.
\end{compactitem}
See here:\setlength{\parskip}{\baselineskip}\\
\centerline{\psfig{file=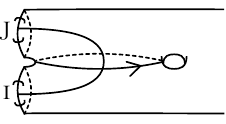,scale=100}}
We remark that right now we allow boundary components that contain components of
both $I$ and $J$, but later when we discuss double-tethered vanishing loops
our hypotheses will exclude this possibility.

\p{Complex of double-tethered loops}
For a surface $\Sigma \in \Surf$ and finite disjoint unions of open intervals
$I,J \subset \partial \Sigma$ with $I \cap J = \emptyset$, 
the {\em complex of $(I,J)$-double-tethered loops}
on $\Sigma$, denoted $\DTL(\Sigma,I,J)$, is the simplicial complex
whose $k$-simplices are collections $\{\iota_0,\ldots,\iota_k\}$
of isotopy classes of $(I,J)$-double-tethered loops on $\Sigma$ that can be
realized so as to be disjoint and not separate $\Sigma$.  See here:\\
\centerline{\psfig{file=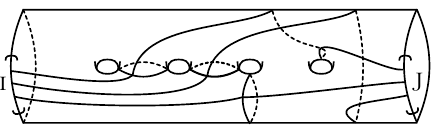,scale=100}}
This complex was introduced by
Hatcher--Vogtmann \cite{HatcherVogtmannTethers}, who proved that if $\Sigma$
has genus $g$ then like $\TL(\Sigma,I)$ it is $\frac{g-3}{2}$-connected (see
\cite[Proposition 5.2]{HatcherVogtmannTethers}).

\p{$\cP$-adjacency}
Consider $(\Sigma,\cP) \in \PSurf$ and let $I,J \subset \partial \Sigma$
be finite disjoint unions of open intervals with $I \cap J = \emptyset$.  Recall
that components $\partial$ and $\partial'$ of $\partial \Sigma$ are said to be
$\cP$-adjacent if there exists some $p \in \cP$ such that $\partial,\partial' \in p$.
We will say that $I$ and $J$ are {\em $\cP$-adjacent} if for all components
$\partial_I$ and $\partial_J$ of $\partial \Sigma$ such that $\partial_I$ contains
a component of $I$ and $\partial_J$ contains a component of $J$, the components
$\partial_I$ and $\partial_J$ are distinct and $\cP$-adjacent.  

\p{Complex of double-tethered vanishing loops}
Let $\mu$ be an $A$-homology marking on $(\Sigma,\cP) \in \PSurf$ and let $I,J \subset \partial \Sigma$
be $\cP$-adjacent disjoint unions of open intervals with $I \cap J = \emptyset$.  In particular,
there are no boundary components of $\Sigma$ containing components of both $I$ and $J$.
Define $\DTL(\Sigma,I,J,\cP,\mu)$ to be the subcomplex of $\DTL(\Sigma,I,J)$ consisting
of $k$-simplices $\{\iota_0,\ldots,\iota_k\}$ satisfying the following conditions.
Let $\hmu \colon \HH_1(\Sigma) \rightarrow A$ be the closed marking associated to $\mu$.
\setlength{\parskip}{0pt}
\begin{compactitem}
\item For $0 \leq i \leq k$, let $\gamma_i$ be the oriented loop
$(\iota_i)|_{S^1}$ and $\alpha_i$ be the oriented arc $(\iota_i)|_{[0,2]}$.  We
then require that $\hmu([\iota_i])=0$ and $\mu([\alpha_i])=0$.  This second condition
makes sense since $I$ and $J$ are $\cP$-adjacent.
\item Set $\Gamma = \{\gamma_0,\ldots,\gamma_k\}$.  We then require that
$\hisect_{\Gamma}(\ker(\hmu)) = \Z[\Gamma]$.
\end{compactitem}
Identifying $\tau(S^1)$ with the union of $[0,1]$ and $S^1$ in $\tau^2(S^1)$, these
conditions imply that $\{(\iota_0)|_{\tau(S^1)},\ldots,(\iota_k)|_{\tau(S^1)}\}$ is
a simplex of $\TL(\Sigma,I,\cP,\mu)$.
\setlength{\parskip}{\baselineskip}

\subsection{The complex of mixed-tethered vanishing loops}
\label{section:mixedtetheredloops}

Our main theorem about the complex of double-tethered vanishing loops says that it
is highly-connected.  We will prove this in \S \ref{section:doubletetheredloopscon} below.
This section is devoted to an intermediate complex that will play a technical role in that proof.

\p{Complex of mixed-tethered vanishing loops}
Let $\mu$ be an $A$-homology marking on $(\Sigma,\cP) \in \PSurf$ and let $I,J \subset \partial \Sigma$
be $\cP$-adjacent disjoint unions of open intervals with $I \cap J = \emptyset$.
Let $\hmu \colon \HH_1(\Sigma) \rightarrow A$ be the closed marking associated to $\mu$.
Define $\MTL(\Sigma,I,J,\cP,\mu)$ to 
be the simplicial complex whose $k$-simplices are sets $\{\iota_0,\ldots,\iota_k\}$, where
each $\iota_i$ is the isotopy class of either an $I$-tethered loop or an
$(I,J)$-double-tethered loop and where the following conditions are satisfied.
\setlength{\parskip}{0pt}
\begin{compactitem}
\item The $\iota_i$ can be realized such that their images are disjoint and do not separate $\Sigma$.
\item For $0 \leq i \leq k$, let $\gamma_i$ be the oriented loop
$(\iota_i)|_{S^1}$.  We then require that $\hmu([\gamma_i])=0$.
\item For $0 \leq i \leq k$ such that $\iota_i$ is an $(I,J)$-double-tethered
loop, let $\alpha_i$ be the oriented arc $(\iota_i)|_{[0,2]}$.  We then
require that $\mu([\alpha_i])=0$.
\item Set $\Gamma = \{\gamma_0,\ldots,\gamma_k\}$.  We then require that
$\hisect_{\Gamma}(\ker(\hmu)) = \Z[\Gamma]$.
\end{compactitem}
These conditions ensure that both $\DTL(\Sigma,I,J,\cP,\mu)$ and \setlength{\parskip}{\baselineskip}
$\TL(\Sigma,I,\cP,\mu)$ are full subcomplexes of $\MTL(\Sigma,I,J,\cP,\mu)$.

\p{Links}
Our first task will be to identify links in $\MTL(\Sigma,I,J,\cP,\mu)$.

\begin{lemma}
\label{lemma:mtllinks}
Let $\mu$ be an $A$-homology marking on $(\Sigma,\cP) \in \PSurf$.  Let $I,J \subset \partial \Sigma$
be $\cP$-adjacent finite disjoint unions of open intervals with $I \cap J = \emptyset$.
Finally, let $\sigma$ be a $k$-simplex of $\MTL(\Sigma,I,J,\cP,\mu)$.  Then there exists
some $(\Sigma',\cP') \in \PSurf$, an $A$-homology marking $\mu'$ on $(\Sigma',\cP')$, and
$\cP'$-adjacent finite disjoint unions of open intervals $I',J' \subset \partial \Sigma'$ with 
$I' \cap J' = \emptyset$ such that the following hold.
\setlength{\parskip}{0pt}
\begin{compactitem}
\item The link of $\sigma$ is isomorphic to $\MTL(\Sigma',I',J',\cP',\mu')$.  Moreover, the intersections
of the link of $\sigma$ with $\TL(\Sigma,I,\cP,\mu)$ and $\DTL(\Sigma,I,J,\cP,\mu)$ are
$\TL(\Sigma',I',\cP',\mu')$ and $\DTL(\Sigma',I',J',\cP',\mu')$, respectively.
\item If $\Sigma$ is a genus $g$ surface, then $\Sigma'$ is a genus $(g-k-1)$-surface.
\item If $\mu$ is supported on a symplectic subsurface, then so is $\mu'$. \setlength{\parskip}{\baselineskip}
\end{compactitem}
\end{lemma}
\begin{proof}
It is enough to deal with the case where $\sigma$ has dimension $0$; the general case will
then follow by applying the dimension $0$ case repeatedly.  We thus can assume
that $\sigma = \{\iota\}$, where $\iota$ is either an $I$-tethered loop or
an $(I,J)$-double-tethered loop.  The two cases are similar, so we will give the
details for when $\iota$ is an $(I,J)$-double-tethered loop.  Let
$\Sigma'$ be the result of cutting $\Sigma$ open along the image of $\iota$:\\
\centerline{\psfig{file=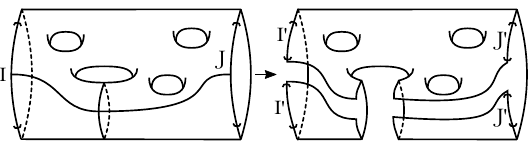,scale=100}}
We remark that the fact that $I$ and $J$ are $\cP$-adjacent implies that the
initial and terminal points of the double tether are on distinct boundary components. 

By isotoping $\Sigma'$ into the interior of $\Sigma$, we can regard $\Sigma'$ as a $\cP$-simple
subsurface of $\Sigma$; see here:\\
\centerline{\psfig{file=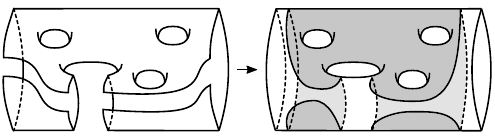,scale=100}}
Let $\cP'$ be the induced partition of the components of $\partial \Sigma'$.
By Lemma \ref{lemma:destabilizemarking}, there exists
an $A$-homology marking $\mu'$ on $(\Sigma',\cP')$ such that $\mu$ is the stabilization
of $\mu'$ to $(\Sigma,\cP)$.

As is clear from the above figure, when forming $\Sigma'$ the sets $I$ and $J$
are divided into finer collections $I'$ and $J'$ of open intervals in
$\partial \Sigma'$ such that the link of $\sigma$ is isomorphic to
$\MTL(\Sigma',I',J',\cP',\mu')$.  By construction, $\Sigma'$ has genus
$g-1$.  The only thing that remains to be proved is that if $\mu$ is
supported on a symplectic subsurface, then so is $\mu'$.  Letting
$\hmu\colon \HH_1(\Sigma) \rightarrow A$ be the closed marking associated
to $\mu$ and $q$ be the interior boundary components of $\Sigma'$ (as
in the definition of a $\cP$-simple subsurface in \S \ref{section:destabilize}),
Lemma \ref{lemma:destabilizesymplectic} says that it is enough to
prove that $\hisect_{q}(\ker(\hmu)) = \RZ[q]$.

Let $\gamma = \iota|_{S^1}$.  Since $\iota$ is a vertex of $\MTL(\Sigma,I,J,\cP,\mu)$,
there exists some $x \in \ker(\hmu)$ such that $\hisect_{\gamma}(x) = \gamma$.  By
construction, we have $q = \{\gamma_1,\gamma_2\}$, where $\gamma_1$ (resp.\ $\gamma_2$)
is obtained by band-summing $\gamma$ with a component of $\partial \Sigma$ containing
a component of $I$ (resp.\ $J$).  The orientations on the $\gamma_i$ are such
that $\gamma_1$ is homologous in $\HH_1(\Sigma,\partial \Sigma)$ to $\gamma$
and $\gamma_2$ is homologous to $-\gamma$.  It follows that
$\hisect_q(x) = \gamma_1 - \gamma_2$, which generates $\RZ[q]$.  The lemma follows.
\end{proof}

\p{Completing a tethered loop to a double-tethered loop}
As a first application of Lemma \ref{lemma:mtllinks} (or, rather, its proof), we prove the following.

\begin{lemma}
\label{lemma:completetl}
Let $\mu$ be an $A$-homology marking on $(\Sigma,\cP) \in \PSurf$ that
is supported on a symplectic subsurface.  Let $I,J \subset \partial \Sigma$
be $\cP$-adjacent finite disjoint unions of open intervals with $I \cap J = \emptyset$.
Then for all vertices $\iota$ of $\TL(\Sigma,I,\cP,\mu)$, there exists a vertex
$\hiota$ of $\DTL(\Sigma,I,J,\cP,\mu)$ such that $\hiota|_{\tau(S^1)} = \iota$.
\end{lemma}
\begin{proof}
Let $(\Sigma',\cP')$ and $I',J'$ and $\mu'$ be the output of applying Lemma \ref{lemma:mtllinks}
to the $0$-simplex $\{\iota\}$ of $\TL(\Sigma,I,\cP,\mu) \subset \MTL(\Sigma,I,J,\cP,\mu)$.  The
$A$-homology marking $\mu'$ on $(\Sigma',\cP')$ is thus supported on a symplectic subsurface.  As
in the following figure, it is enough to find an embedded arc $\alpha$ in $\Sigma'$ connecting
the endpoint $p_0$ of the tether of $\iota$ to a point of $J$ such that
$\mu'([\alpha]) = 0$:\\
\centerline{\psfig{file=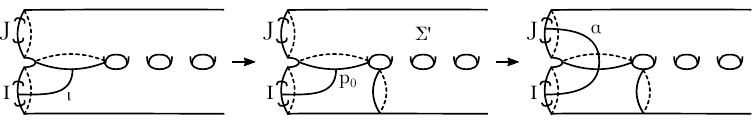,scale=100}}
Since $\mu'$ is supported on a symplectic subsurface, Lemma \ref{lemma:identifysupport} implies that
there exists an immersed arc $\alpha$ (not necessarily embedded) connecting $p_0$ to a point of $J$ such that
$\mu'([\alpha])=0$.  Choose $\alpha$ so as to have the fewest possible self-intersections.  Then $\alpha$
is embedded; indeed, if it has a self-intersection, then as in the following figure we can ``comb'' its
first self-intersection over the component of $\partial \Sigma'$ containing $p_0$:\\
\centerline{\psfig{file=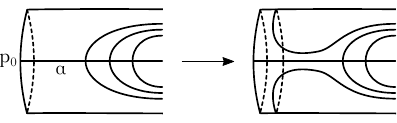,scale=100}}
This has the effect of removing a self-intersection from $\alpha$, but since $\mu'$ vanishes on
all components of $\partial \Sigma'$ it does not change the fact that $\mu'([\alpha])=0$.  The lemma
follows.
\end{proof}

\p{High connectivity}
We close this section by proving that $\MTL(\Sigma,I,J,\cP,\mu)$ is highly connected.

\begin{theorem}
\label{theorem:vanishmixedtetheredloopscon}
Let $\mu$ be an $A$-homology marking on $(\Sigma,\cP) \in \PSurf$.  Let $I,J \subset \partial \Sigma$
be $\cP$-adjacent finite disjoint unions of open intervals with $I \cap J = \emptyset$ and let
$g$ be the genus of $\Sigma$.  Then
$\MTL(\Sigma,I,J,\cP,\mu)$ is $\frac{g-(2\rank(A)+3)}{\rank(A)+2}$-connected.
\end{theorem}
\begin{proof}
Set $n = \frac{g-(2\rank(A)+3)}{\rank(A)+2}$ and
\[X = \MTL(\Sigma,I,J,\cP,\mu) \quad \text{and} \quad Y = \TL(\Sigma,I,J,\cP,\mu) \quad \text{and} \quad Y' = \DTL(\Sigma,I,J,\cP,\mu).\]
Theorem \ref{theorem:vanishtetheredloopcon} says that $Y$ is $n$-connected, so it is enough
to prove that the pair $(X,Y)$ is $n$-connected.  To do this, we will apply
Corollary \ref{corollary:avoidsubcomplex}.  This requires showing the following.
Let $\sigma$ be a $k$-dimensional simplex of $Y'$ and
let $L$ be the link of $\sigma$ in $X$.  Then we must show that
$L \cap Y$ is $(n-k-1)$-connected.

Lemma \ref{lemma:mtllinks} says that $L \cap Y \cong \TL(\Sigma',I',J',\cP',\mu')$, where
$\Sigma',I',J',\cP',\mu'$ are as follows:
\setlength{\parskip}{0pt}
\begin{compactitem}
\item $(\Sigma',\cP') \in \PSurf$ with $\Sigma'$ a genus $(g-k-1)$ surface.
\item $\mu'$ is an $A$-homology marking on $(\Sigma',\cP')$.
\item $I',J' \subset \partial \Sigma'$ are $\cP'$-adjacent finite disjoint unions of open intervals
with $I' \cap J' = \emptyset$.
\end{compactitem}
Theorem \ref{theorem:vanishtetheredloopcon} thus says that
$L \cap Y$ is $n'$-connected for
\[n' = \frac{g'-(2\rank(A)+3)}{\rank(A)+2}
= \frac{g-(2\rank(A)+3)}{\rank(A)+2} - \frac{k+1}{\rank(A)+2} \geq n-k-1. \qedhere\]
\end{proof}

\subsection{High connectivity of the complex of double-tethered vanishing loops}
\label{section:doubletetheredloopscon}

In this section, we finally prove that the complex of double-tethered vanishing loops
is highly connected:

\begin{theorem}
\label{theorem:vanishdoubletetheredloopcon}
Let $\mu$ be an $A$-homology marking on $(\Sigma,\cP) \in \PSurf$ that
is supported on a symplectic subsurface.  Let $I,J \subset \partial \Sigma$
be $\cP$-adjacent finite disjoint unions of open intervals with $I \cap J = \emptyset$
and let $g$ be the genus of $\Sigma$.  Then
$\DTL(\Sigma,I,J,\cP,\mu)$ is $\frac{g-(2\rank(A)+3)}{\rank(A)+2}$-connected.
\end{theorem}

The proof of Theorem \ref{theorem:vanishdoubletetheredloopcon} requires the following lemma.
Say that a simplicial map $f\colon M \rightarrow X$ between
simplicial complexes is {\em locally injective} if $f|_{\sigma}$ is injective for all
simplices $\sigma$ of $M$.  

\begin{lemma}
\label{lemma:localinjectivity}
Let $M$ be a compact $n$-dimensional manifold (possibly with boundary) equipped with a combinatorial
triangulation, let $X$ be a simplicial complex, and let $f\colon M \rightarrow X$ be a simplicial
map.  Assume the following hold:
\setlength{\parskip}{0pt}
\begin{compactitem}
\item $f|_{\partial M}$ is locally injective.
\item For all simplices $\sigma$ of $X$, the link of $\sigma$ in $X$ is $(n-\dim(\sigma)-2)$-connected.
\end{compactitem}
Then after possibly subdividing simplices
of $M$ lying in its interior, $f$ is homotopic through maps fixing $\partial M$ to a simplicial map
$f'\colon M \rightarrow X$ that is locally injective. \setlength{\parskip}{\baselineskip}
\end{lemma}
\begin{proof}
We remark that the proof of this is very similar to Hatcher--Vogtmann's proof of
Proposition \ref{proposition:avoidbad} above, though it seems hard to deduce it from
that proposition.  This result is also related to \cite[Theorem 2.4]{GalatiusRandalWilliams}.

The proof will be by induction on $n$.  The base case $n=0$ is trivial, so
assume that $n>0$ and that the result is true for all smaller dimensions.
Call a simplex $\sigma$ of $M$ a {\em noninjective simplex} if for all vertices $v$ of $\sigma$, there exists a vertex
$v'$ of $\sigma$ with $v \neq v'$ but $f(v) = f(v')$.
If $M$ has no noninjective simplices, then we are done.  Assume, therefore, that $M$ has noninjective simplices,
and let $\sigma$ be a noninjective simplex of $M$ whose dimension is as large as possible.
Since no simplices of $\partial M$ are noninjective, the simplex $\sigma$ does not lie in $\partial M$.  Letting
$L \subset M$ be the link of $\sigma$, this implies that $L \cong S^{n-\dim(\sigma)-1}$.  Letting
$L'$ be the link of $f(\sigma)$ in $X$, the maximality of the dimension of $\sigma$ implies
two things:
\setlength{\parskip}{0pt}
\begin{compactitem}
\item $f(L) \subset L'$
\item The restriction of $f$ to $L$ is locally injective.
\end{compactitem}
Our assumptions imply that $L'$ has connectivity at least
\[n-\dim(f(\sigma)) - 2 \geq n-(\dim(\sigma)-1) - 2 = n-\dim(\sigma)-1;\]
here we are using the fact that $f|_{\sigma}$ is not injective.  We can thus extend $f|_L$ to a map
\[F\colon D^{n-\dim(\sigma)} \rightarrow L'\]
that is simplicial with respect to some combinatorial triangulation of $D^{n-\dim(\sigma)}$
that restricts to $L \cong S^{n-\dim(\sigma)-1}$ on $\partial D^{n-\dim(\sigma)}$.  Since $\dim(\sigma) \geq 1$ and $F|_{\partial D^{n-\dim(\sigma)}} = f|_L$
is locally injective, we can apply our inductive hypothesis to $F$ and ensure that $F$ is locally injective.
The star $S$ of $\sigma$ is isomorphic to the join $\sigma \ast L$.  Subdividing
$M$ and homotoping $f$, we can replace $S \subset D^{n-\dim(\sigma)}$
with $\partial \sigma \ast D^{n-\dim(\sigma)}$ and $f|_S$ with $f|_{\partial \sigma} \ast F$.
Here are pictures of this operation for $n=2$ and $\dim(\sigma) \in \{0,1,2\}$; on the left
hand side is $S$, and on the right hand side is $\partial \sigma \ast D^{n-\dim(\sigma)}$:\\
\centerline{\psfig{file=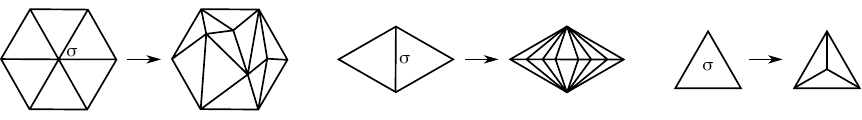,scale=100}}
In doing this, we have eliminated the noninjective simplex $\sigma$ without introducing any new
noninjective simplices.  Repeating this over and over again, we can eliminate all noninjective simplices,
and we are done.
\end{proof}

\begin{proof}[Proof of Theorem \ref{theorem:vanishdoubletetheredloopcon}]
We will prove by induction on $n$ that $\DTL(\Sigma,I,J,\cP,\mu)$ is $n$-connected
for $-1 \leq n \leq \frac{g-(2\rank(A)+3)}{\rank(A)+2}$.  The base case $n=-1$ simply
asserts that $\DTL(\Sigma,I,J,\cP,\mu)$ is nonempty when $\frac{g-(2\rank(A)+3)}{\rank(A)+2} \geq -1$.
In this case, Theorem \ref{theorem:vanishtetheredloopcon}
asserts that $\TL(\Sigma,I,\cP,\mu) \neq \emptyset$, and thus Lemma \ref{lemma:completetl} implies that
$\DTL(\Sigma,I,J,\cP,\mu) \neq \emptyset$, as desired.

Assume now that $0 \leq n \leq \frac{g-(2\rank(A)+3)}{\rank(A)+2}$ and that
all complexes $\DTL(\Sigma',I',J',\cP',\mu')$ as in the theorem are $n'$-connected
for $n' = \min\{n-1,\frac{g'-(2\rank(A)+3)}{\rank(A)+2}\}$, where $g'$ is the genus of $\Sigma'$.
We must prove that $\DTL(\Sigma,I,J,\cP,\mu)$ is $n$-connected.  

Set $X = \MTL(\Sigma,I,J,\cP,\mu)$.  The complex $\DTL(\Sigma,I,J,\cP,\mu)$ 
that we want to show is $n$-connected is a subcomplex of $X$, and 
Theorem \ref{theorem:vanishmixedtetheredloopscon}
says that the connectivity of $X$ is at least
\[\frac{g-(2\rank(A)+3)}{\rank(A)+2} \geq n.\]
Define
$Y$ to be the subcomplex of $X$ consisting of simplices containing at most
one vertex of $\TL(\Sigma,I,\cP,\mu)$, so
\[\DTL(\Sigma,I,J,\cP,\mu) \subsetneq Y \subsetneq X.\]
The first step is as follows.

\begin{claims}
\label{claim:ycon}
The complex $Y$ is $n$-connected.
\end{claims}
\begin{proof}[Proof of claim]
We know that $X$ is $n$-connected, so to prove that its subcomplex $Y$ is $n$-connected
it is enough to prove that the pair $(X,Y)$ is $(n+1)$-connected.
We will do this using Proposition \ref{proposition:avoidbad}.  For this, we
must identify a set $\cB$ of ``bad simplices'' of $X$ and verify the three
hypotheses of the proposition.  Define $\cB$ to be the set of all simplices 
of $\TL(\Sigma,I,\cP,\mu) \subset X$ whose dimension is at least $1$. 

We now verify the hypotheses of Proposition \ref{proposition:avoidbad}.  The first two are
easy:
\setlength{\parskip}{0pt}
\begin{compactitem}
\item (i) says that a simplex of $X$ lies in $Y$ if and only if none of its faces
lie in $\cB$, which is obvious.
\item (ii) says that if $\sigma_1,\sigma_2 \in \cB$ are such that $\sigma_1 \cup \sigma_2$ is a simplex
of $X$, then $\sigma_1 \cup \sigma_2 \in \cB$, which again is obvious.
\end{compactitem}
The only thing left to check is (iii), which says that for all $k$-dimensional $\sigma \in \cB$, the
complex $G(X,\sigma,\cB)$ has connectivity at least \setlength{\parskip}{\baselineskip}
$(n+1) - k - 1 = n - k$.

Let $L$ be the link of $\sigma$ in $X$.
Examining its definition in \S \ref{section:linkarguments}, we see that
\[G(X,\sigma,\cB) \cong L \cap \DTL(\Sigma,I,J,\cP,\mu).\]
Lemma \ref{lemma:mtllinks} says that 
$L \cap \DTL(\Sigma,I,J,\cP,\mu) \cong \DTL(\Sigma',I',J',\cP',\mu')$, where
$\Sigma',I',J',\cP',\mu'$ are as follows:
\setlength{\parskip}{0pt}
\begin{compactitem}
\item $(\Sigma',\cP') \in \PSurf$ with $\Sigma'$ a genus $g'=g-k-1$ surface.
\item $\mu'$ is an $A$-homology marking on $(\Sigma',\cP')$ that is supported on a symplectic subsurface.
\item $I',J' \subset \partial \Sigma'$ are $\cP'$-adjacent finite disjoint unions of open intervals
with $I' \cap J' = \emptyset$.
\end{compactitem}
Our goal is thus to show that $\DTL(\Sigma',I',J',\cP',\mu')$ is $(n-k)$-connected.
Our inductive hypothesis shows that $\DTL(\Sigma',I',J',\cP',\mu')$ is $n'$-connected
for
\begin{align*}
n' &= \min\{n-1,\frac{g'-(2\rank(A)+3)}{\rank(A)+2}\} \\
&= \min\{n-1,\frac{g-(2\rank(A)+3)}{\rank(A)+2} - \frac{k+1}{\rank(A)+2}\} \\
&\geq \min\{n-1,n-\frac{k+1}{2}\} \geq n-k.
\end{align*}
Here we are using the fact that by the definition of $\cB$, we have $k \geq 1$, and
thus $k \geq \frac{k+1}{2}$.
\end{proof}

This allows us to fill $n$-spheres in $\DTL(\Sigma,I,J,\cP,\mu)$ with $(n+1)$-discs in $Y$.  We will
modify these $(n+1)$-discs such that they lie in $\DTL(\Sigma,I,J,\cP,\mu)$.  For technical reasons,
we will need our spheres and discs to be locally injective.  That this is possible is the content
of the following two steps.

\begin{claims}
\label{claim:localinjectivespheres}
Equip the $n$-sphere $S^n$ with a combinatorial triangulation and let 
$f\colon S^n \rightarrow \DTL(\Sigma,I,J,\cP,\mu)$ be a simplicial map.
Then after possibly subdividing $S^n$, the map $f$ is homotopic to a locally
injective simplicial map.
\end{claims}
\begin{proof}[Proof of claim]
By Lemma \ref{lemma:localinjectivity}, this will follow if we can show that for all $k$-simplices
$\sigma$ of $\DTL(\Sigma,I,J,\cP,\mu)$, the link $L$ of $\sigma$ is $(n-k-2)$-connected.  Applying
Lemma \ref{lemma:mtllinks}, we see that $L \cong \DTL(\Sigma',I',J',\cP',\mu')$, where
$\Sigma',I',J',\cP',\mu'$ are as follows:
\setlength{\parskip}{0pt}
\begin{compactitem}
\item $(\Sigma',\cP') \in \PSurf$ with $\Sigma'$ a genus $g'=g-k-1$ surface.
\item $\mu'$ is an $A$-homology marking on $(\Sigma',\cP')$ that is supported on a symplectic
subsurface.
\item $I',J' \subset \partial \Sigma'$ are $\cP'$-adjacent finite disjoint unions of open intervals
with $I' \cap J' = \emptyset$.
\end{compactitem}
Our inductive hypothesis thus says that $L \cong \DTL(\Sigma',I',J',\cP',\mu')$ is $n'$-connected
for
\begin{align*}
n' &= \min\{n-1,\frac{g'-(2\rank(A)+3)}{\rank(A)+2}\} \\
&= \min\{n-1,\frac{g-(2\rank(A)+3)}{\rank(A)+2} - \frac{k+1}{\rank(A)+2}\} \\
&\geq \min\{n-1,n-\frac{k+1}{\rank(A)+2}\} \geq n-k-2,
\end{align*}
as desired.
\end{proof}

\begin{claims}
\label{claim:yconbetter}
Equip the $n$-sphere $S^n$ with a combinatorial triangulation and let
$f\colon S^n \rightarrow Y$ be a locally injective simplicial map that extends to a simplicial
map of a combinatorial triangulation of $D^{n+1}$.
Then there exists a combinatorial triangulation of $D^{n+1}$ that restricts to our
given one on $\partial D^{n+1} = S^n$ and a locally injective simplicial
map $F\colon D^{n+1} \rightarrow Y$ such that $F|_{\partial D^{n+1}} = f$.
\end{claims}
\begin{proof}[Proof of claim]
By Lemma \ref{lemma:localinjectivity}, this will follow if we can show that for all $k$-simplices
$\sigma$ of $Y$, the link $L$ of $\sigma$ is $(n-k-1)$-connected.  As temporary notation,
write $Y(\Sigma,I,J,\cP,\mu)$ for $Y$.  By Lemma \ref{lemma:mtllinks}, we have
either
\[L \cong \DTL(\Sigma',I',J',\cP',\mu') \quad \text{or} \quad L \cong Y(\Sigma',I',J',\cP',\mu')\]
depending on whether or not $\sigma$ contains a vertex of $\TL(\Sigma,I,\cP,\mu)$.
Here $\Sigma',I',J',\cP',\mu'$ are as follows:
\setlength{\parskip}{0pt}
\begin{compactitem}
\item $(\Sigma',\cP') \in \PSurf$ with $\Sigma'$ a genus $g'=g-k-1$ surface.
\item $\mu'$ is an $A$-homology marking on $(\Sigma',\cP')$ that is supported on a symplectic
subsurface.
\item $I',J' \subset \partial \Sigma'$ are $\cP'$-adjacent finite disjoint unions of open intervals
with $I' \cap J' = \emptyset$.
\end{compactitem}
Applying either our inductive hypothesis or Claim \ref{claim:ycon}, we see
that $L$ is $n'$-connected for
\begin{align*}
n' &= \min\{n-1,\frac{g'-(2\rank(A)+3)}{\rank(A)+2}\} \\
&= \min\{n-1,\frac{g-(2\rank(A)+3)}{\rank(A)+2} - \frac{k+1}{\rank(A)+2}\} \\
&\geq \min\{n-1,n-\frac{k+1}{\rank(A)+2}\} \geq n-k-1,
\end{align*}
as desired.
\end{proof}

We now finally turn to proving that $\DTL(\Sigma,I,J,\cP,\mu)$ is $n$-connected.
Our inductive hypothesis says that it is $(n-1)$-connected, so it is enough
to prove that every continuous map $f\colon S^n \rightarrow \DTL(\Sigma,I,J,\cP,\mu)$
can be extended to a continuous map $F\colon D^{n+1} \rightarrow \DTL(\Sigma,I,J,\cP,\mu)$.
Using simplicial approximation, we can assume that $f$ is simplicial with respect
to a combinatorial triangulation of $S^n$.  Next, using 
Claim \ref{claim:localinjectivespheres} we can ensure that $f$ is locally injective.
The complex $\DTL(\Sigma,I,J,\cP,\mu)$ is a subcomplex of $Y$ and
Claim \ref{claim:ycon} says that $Y$ is $n$-connected, so we can extend
$f$ to a continuous map $F\colon D^{n+1} \rightarrow Y$, which by the relative
version of simplicial approximation we can ensure is simplicial with respect
to a combinatorial triangulation of $D^{n+1}$ that restricts to our given one on
$S^n$.  Finally, applying Claim \ref{claim:yconbetter} we can ensure that
$F$ is locally injective.

If $F$ does not map any vertices of $D^{n+1}$ to $\TL(\Sigma,I,\cP,\mu)$,
then the image of $F$ lies in $\DTL(\Sigma,I,J,\cP,\mu)$ and we are done.  Assume,
therefore, that $x$ is a vertex of $D^{n+1}$ such that
$F(x)$ is a vertex $\iota\colon \tau(S^1) \rightarrow \Sigma$
of $\TL(\Sigma,I,\cP,\mu)$.  Let $L \subset D^{n+1}$ be the link of $x$
and let $\cL \subset Y$ be the link of $\iota = F(x)$.  Since $F$ is
locally injective, we have $F(L) \subset \cL$.  Also, since simplices
of $Y$ can contain at most one vertex of $\TL(\Sigma,I,\cP,\mu)$,
we have $\cL \subset \DTL(\Sigma,I,J,\cP,\mu)$.

By Lemma \ref{lemma:completetl}, we can find a vertex $\hiota\colon \tau^2(S^1) \rightarrow \Sigma$ of
$\DTL(\Sigma,I,J,\cP,\mu)$ such that $\hiota|_{\tau(S^1)} = \iota$.  Let $\hcL$
be the link of $\hiota$ in $\DTL(\Sigma,I,J,\cP,\mu)$, so $\hcL \subset \cL$.  As we said
above, we have $F(L) \subset \cL$.  If $F(L) \subset \hcL$,
then we could redefine $F$ to take $x$ to $\hiota$ instead of $\iota$.  Repeating this process
would modify $F$ such that its image lies in $\DTL(\Sigma,I,J,\cP,\mu)$,
and we would be done.

Unfortunately, it might not be the case that $F(L) \subset \hcL$.  We will therefore have to
perform a more complicated modification to $F$.
Since $L \subset D^{n+}$ is the link of the vertex $x$ and $x$ does not
lie in $\partial D^{n+1}$, we have $L \cong S^n$.  Recall that
\[F(L) \subset \cL \subset \DTL(\Sigma,I,J,\cP,\mu).\]
Among all simplicial maps
\[G\colon L \rightarrow \cL \subset \DTL(\Sigma,I,J,\cP,\mu)\]
that are homotopic to $F|_L$ through maps $S^n \rightarrow \DTL(\Sigma,I,J,\cP,\mu)$,
pick the one that minimizes the total number of intersections between the
image of $\hiota\colon \tau^2(S^1) \rightarrow \Sigma$ and the images
of $G(y)\colon \tau^2(S^1) \rightarrow \Sigma$ as $y$ ranges over the vertices
of $L$.  

Below in Claim \ref{claim:disjoint} we will prove that with this
choice there are in fact no such intersections, and thus the image $G(L)$
lies in the link $\hcL$ of $\hiota$ in $\DTL(\Sigma,I,J,\cP,\mu)$.  
Letting $\ast$ denote the join, we can then replace the restriction of $F$ to the subset
\[x \ast L \cong \text{pt} \ast S^n \cong D^{n+1}\]
of $D^{n+1}$ with the following two pieces:\setlength{\parskip}{0pt}
\begin{compactitem}
\item An annular region that is a combinatorial triangulation of $S^n \times [0,1]$ both
of whose boundary components are $L$.  On this region, $F$ maps
to a homotopy from $F|_L$ to $G$.
\item The cone $x \ast L \cong \text{pt} \ast S^n \cong D^{n+1}$, on which $F$ is defined
to equal $G$ on $L$ and to take $x$ to $\hiota$.
\end{compactitem}
See the following figure, where the shaded region is the homotopy from $F|_L$ to $G$:\setlength{\parskip}{\baselineskip}\\
\centerline{\psfig{file=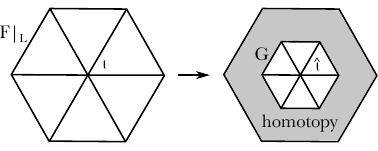,scale=100}}
This redefines $F$ such that $F(x) = \hiota$ without introducing any other vertices mapping
to vertices of $\TL(\Sigma,I,J,\cP,\mu)$, completing the proof.

It remains to prove the aforementioned claim about $G\colon L \rightarrow \cL \subset \DTL(\Sigma,I,J,\cP,\mu)$. 

\begin{claims}
\label{claim:disjoint}
For all vertices $y$ of $L$, we can choose a representative of 
$G(y)\colon \tau^2(S^1) \rightarrow \Sigma$ whose image
is disjoint from the image of $\hiota\colon \tau^2(S^1) \rightarrow \Sigma$.
\end{claims}
\begin{proof}[Proof of claim]
Assume otherwise.
Since the image of $G$ lies in the link $\cL$ of $\iota$, 
we can choose representatives of the $G(y)$ for $y \in L$ that
are disjoint from the image of $\iota\colon \tau(S^1) \rightarrow \Sigma$.  Pick
these representatives such that their intersections with the image of
$\hiota|_{[1,2]}\colon [1,2] \rightarrow \Sigma$ are transverse and all distinct.
Let $y$ be the vertex of $L$ such that the image of 
$\eta:=G(y)\colon \tau^2(S^1) \rightarrow \Sigma$ intersects the image of
$\hiota|_{[1,2]}\colon [1,2] \rightarrow \Sigma$ in the first of these intersection points
(enumerated from $\hiota(1)$ to $\hiota(2)$).

The argument is slightly different depending on whether this intersection point is
contained in the image under $\eta\colon \tau^2(S^1) \rightarrow \Sigma$ of
$[0,1]$ or $S^1$ or $[1,2]$.  We will give the details for when this intersection
point is contained in $\eta(S^1)$; the other cases are similar.

As in the following figure, let $\eta'\colon \tau^2(S^1) \rightarrow \Sigma$ be
the result of ``sliding'' the intersection point of $\eta$ in question
across $\iota(S^1)$ via the initial segment of $\iota([1,2])$:\\
\centerline{\psfig{file=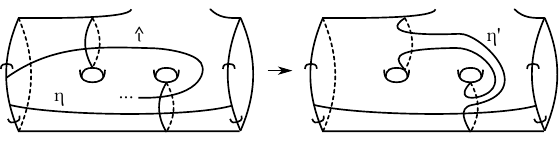,scale=100}}
The image of $\eta'$ intersects the image of $\hiota$ in one fewer place
than the image of $\eta$.  Define 
\[G'\colon L \rightarrow \cL \subset \DTL(\Sigma,I,J,\cP,\mu)\]
to be the map which equals $G$ except at the vertex $y$, where $G'(y) = \eta'$
instead of $\eta$.  It is easy to see that $G'$ is indeed a simplicial
map.  Since the image of $\eta'$ intersects the image of $\hiota$ in one fewer place than 
the image of $\eta$, to derive a contradiction to the minimality of the total
number of these intersections it is enough to prove that $G$ and $G'$ are
homotopic through maps landing in $\DTL(\Sigma,I,J,\cP,\mu)$.

Define $L' \cong S^{n-1}$ to be the link of $y$ in $S^n$, define $\cL_{\eta}$
to be the link of $\eta$ in $\DTL(\Sigma,I,J,\cP,\mu)$, and define
$\cL_{\eta'}$ to be the link of $\eta'$ in $\DTL(\Sigma,I,J,\cP,\mu)$.
We have $G|_{L'} = G'|_{L'}$, and the image $G(L') = G'(L')$ lies in
$\cL_{\eta} \cap \cL_{\eta'}$.  Below we will prove that the map
$G|_{L'}\colon L' \rightarrow \cL_{\eta} \cap \cL_{\eta'}$ can be homotoped
to a constant map.  This will imply that $G$ and $G'$ are homotopic through
maps lying in $\DTL(\Sigma,I,J,\cP,\mu)$ via a homotopy like the one
in this figure:\\
\centerline{\psfig{file=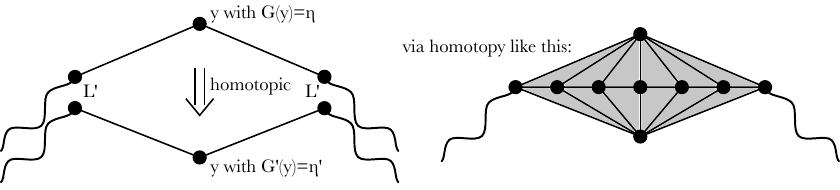,scale=100}}
This figure depicts the case $n=1$: pictured is a fragment of $L \cong S^1$, along
with the vertex $y$ and $L' \cong S^0$.

Since $L' \cong S^{n-1}$, to prove that the map
$G|_{L'}\colon L' \rightarrow \cL_{\eta} \cap \cL_{\eta'}$ can be homotoped
to a constant map, it is enough to prove that $\cL_{\eta} \cap \cL_{\eta'}$ is
$(n-1)$-connected.  Define $\zeta$ to be the union of $\eta(\tau^2(S^1))$, of
$\iota(S^1)$, and of the portion of the arc of $\iota([1,2])$ connecting
$\iota(0) \in \iota(S^1)$ to a point of $\eta(S^1)$; see here:\\
\centerline{\psfig{file=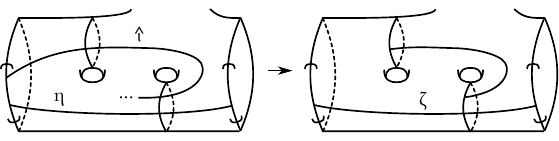,scale=100}}
The images of both
$\eta$ and $\eta'$ are contained in a regular neighborhood of $\zeta$.
Let $\Sigma'$ be the surface obtained by cutting open $\Sigma$ along $\zeta$.
The surface $\Sigma'$ thus has genus $g'=g-2$.  Moreover, an argument
identical to that in the proof of Lemma \ref{lemma:mtllinks} shows that
there exist a partition $\cP'$ of the components of $\partial \Sigma'$,
an $A$-homology marking $\mu'$ on $(\Sigma',\cP')$, and
$\cP'$-adjacent finite disjoint unions of open intervals $I',J' \subset \partial \Sigma'$ with
$I' \cap J' = \emptyset$ such that the following hold:
\setlength{\parskip}{0pt}
\begin{compactitem}
\item $\cL_{\eta} \cap \cL_{\eta'} \cong \MTL(\Sigma',I',J',\cP',\mu')$.  
\item $\mu'$ is supported on a symplectic subsurface. \setlength{\parskip}{\baselineskip}
\end{compactitem}
Our inductive hypothesis thus says that $\cL_{\eta} \cap \cL_{\eta'} \cong \MTL(\Sigma',I',J',\cP',\mu')$ is $n'$-connected for
\begin{align*}
n' &= \min\{n-1,\frac{g'-(2\rank(A)+3)}{\rank(A)+2}\} \\
&= \min\{n-1,\frac{g-(2\rank(A)+3)}{\rank(A)+2} - \frac{2}{\rank(A)+2}\} \\
&\geq \min\{n-1,n-\frac{2}{\rank(A)+2}\} = n-1,
\end{align*}
as desired.
\end{proof}

This completes the proof of Theorem \ref{theorem:vanishdoubletetheredloopcon}.
\end{proof}

\subsection{The complex of order-preserving double-tethered vanishing loops}
\label{section:orderdoubletetheredloops}

We finally come to the complex of order-preserving double-tethered vanishing loops.

\p{Complex of order-preserving double-tethered loops}
Let $\Sigma \in \Surf$ be a surface and let $I,J \subset \partial \Sigma$
be disjoint open intervals.  Orient $I$ such that $\Sigma$ lies on its right
and $J$ such that $\Sigma$ lies on its left.  These two orientations induces
two natural orderings on simplices of $\DTL(\Sigma,I,J)$.  The {\em complex
of order-preserving $(I,J)$-double-tethered loops}, denoted $\ODTL(\Sigma,I,J)$,
is the subcomplex of $\DTL(\Sigma,I,J)$ consisting of simplices such that these
two orderings agree.  Here is an example of such a simplex:\\
\centerline{\psfig{file=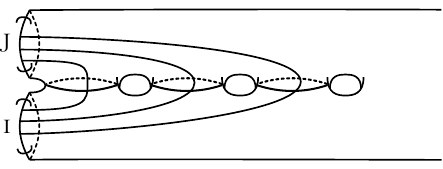,scale=100}}
The complex $\ODTL(\Sigma,I,J)$ was introduced by Hatcher--Vogtmann 
\cite{HatcherVogtmannTethers}, who proved that if $\Sigma$ has genus $g$ then
(like $\TL(\Sigma,I)$ and $\DTL(\Sigma,I,J)$) it is $(g-3)/2$-connected (see
\cite[Proposition 5.3]{HatcherVogtmannTethers}).

\p{Complex of order-preserving double-tethered vanishing loops}
Let $\mu$ be an $A$-homology marking on $(\Sigma,\cP) \in \PSurf$ and let $I,J \subset \partial \Sigma$
be disjoint $\cP$-adjacent open intervals in $\partial \Sigma$.  Define
the complex $\ODTL(\Sigma,I,J,\cP,\mu)$ to be the intersection of
$\DTL(\Sigma,I,J,\cP,\mu)$ with $\ODTL(\Sigma,I,J)$.  The orientations on $I$ and
$J$ endow $\ODTL(\Sigma,I,J,\cP,\mu)$ with a natural ordering on its simplices, and
thus with the structure of a semisimplicial set.

\p{High connectivity}
The following theorem asserts that $\ODTL(\Sigma,I,J,\cP,\mu)$ has the same
connectivity that Theorem \ref{theorem:vanishdoubletetheredloopcon} says
that $\DTL(\Sigma,I,J,\cP,\mu)$ enjoys.

\begin{theorem}
\label{theorem:vanishorderdoubletetheredloopcon}
Let $\mu$ be an $A$-homology marking on $(\Sigma,\cP) \in \PSurf$ that
is supported on a symplectic subsurface.  Let $I,J \subset \partial \Sigma$
be $\cP$-adjacent disjoint open intervals 
and let $g$ be the genus of $\Sigma$.  Then
$\ODTL(\Sigma,I,J,\cP,\mu)$ is $\frac{g-(2\rank(A)+3)}{\rank(A)+2}$-connected.
\end{theorem}
\begin{proof}
In \cite[Proposition 5.3]{HatcherVogtmannTethers}, Hatcher--Vogtmann show
how to derive the fact that $\ODTL(\Sigma,I,J)$ is $\frac{g-3}{2}$-connected
from the fact that $\DTL(\Sigma,I,J)$ is $\frac{g-3}{2}$-connected.  Their argument works
word-for-word to prove this theorem.
\end{proof}

\p{Stabilizers}
In the remainder of this section, we will be interested in the case where $I$ and $J$
are open intervals in distinct components $\partial_I$ and $\partial_J$ 
of $\partial \Sigma$ (much of what we say will also hold if $\partial_I = \partial_J$,
but the pictures would be a bit different).  The $\Mod(\Sigma)$-stabilizer of a simplex 
$\sigma = \{\iota_0,\ldots,\iota_k\}$ of $\ODTL(\Sigma,I,J)$ is the mapping class group 
of the complement $\Sigma'$ of an open regular neighborhood of
\[\partial_I \cup \partial_J \cup \iota_0\left(\tau^2\left(S^1\right)\right) \cup \cdots \cup \iota_k\left(\tau^2\left(S^1\right)\right).\]
We will call this the {\em stabilizer subsurface} of $\sigma$.  See here:\\
\centerline{\psfig{file=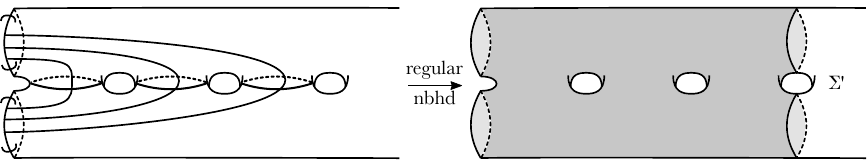,scale=100}}
If $\partial_I$ and $\partial_J$ are $\cP$-adjacent, then the 
surface $\Sigma'$ is a $\cP$-simple subsurface of $\Sigma$, and thus has
an induced partition $\cP'$.  The following lemma
records some of its properties if $\sigma$ is a simplex of $\ODTL(\Sigma,I,\cP,\mu)$ for
an $A$-homology marking $\mu$ on $(\Sigma,\cP)$.

\begin{lemma}
\label{lemma:vanishdoubleloopstab}
Let $\mu$ be an $A$-homology marking on $(\Sigma,\cP)$ and let $I,J$ be open intervals
in distinct $\cP$-adjacent components of $\partial \Sigma$.  Let $\sigma$ be a simplex
of $\ODTL(\Sigma,I,J,\cP,\mu)$, let $\Sigma'$ be its stabilizer subsurface, and let $\cP'$
be the induced partition of $\partial \Sigma'$.
Then there exists an $A$-homology marking $\mu'$ on $(\Sigma',\cP')$
such that $\mu$ is the stabilization of $\mu'$.  Moreover, if $\mu$ is supported
on a symplectic subsurface then so is $\mu'$.
\end{lemma}
\begin{proof}
Identical to that of Lemma \ref{lemma:mtllinks}.
\end{proof}

\p{Transitivity}
The final fact we need about these complexes is as follows.

\begin{lemma}
\label{lemma:doubletetheredtran}
Let $\mu$ be an $A$-homology marking on $(\Sigma,\cP) \in \PSurf$ that is supported on a symplectic subsurface and
let $I,J$ be open intervals in distinct $\cP$-adjacent components of
$\partial \Sigma$.  The group
$\Torelli(\Sigma,\cP,\mu)$ acts transitively on the $k$-simplices
of $\ODTL(\Sigma,I,J,\cP,\mu)$ if the genus of $\Sigma$ is at least $2\rank(A)+3+k$.
\end{lemma}
\begin{proof}
Just like in the proof of Lemma \ref{lemma:vanishtetheredtran}, this will be by induction on $k$.
In fact, once we prove the base case $k=0$ the inductive step is handled
exactly like Lemma \ref{lemma:vanishtetheredtran}, so we will only give the details for $k=0$.

So assume that the genus of $\Sigma$ is at least $2 \rank(A)+3$.  Theorem \ref{theorem:vanishorderdoubletetheredloopcon}
then implies that $\ODTL(\Sigma,I,J,\cP,\mu)$ is connected, so to prove that $\Torelli(\Sigma,\cP,\mu)$ acts
transitively on its vertices it is enough to prove that if $\iota_0,\iota_1\colon \tau^2(S^1) \rightarrow \Sigma$
are vertices that are joined by an edge, then there exists some $f \in \Torelli(\Sigma,\cP,\mu)$ such that
$f(\iota_0) = \iota_1$.  Let $\Sigma'$ be the stabilizer subsurface of $\{\iota_0,\iota_1\}$ and let
$\cP'$ be the induced partition of $\partial \Sigma'$.  By Lemma \ref{lemma:vanishdoubleloopstab}, there
exists an $A$-homology marking $\mu'$ on $(\Sigma',\cP')$ that is supported on a symplectic subsurface such that
$\mu$ is the stabilization of $\mu'$ to $(\Sigma,\cP)$.  Let $S \cong \Sigma_h^1$ be a subsurface of $\Sigma'$
on which $\mu'$ is supported.  

The ``change of coordinates principle'' from 
\cite[\S 1.3.2]{FarbMargalitPrimer} implies that there is a mapping class $f'$ on $\Sigma' \setminus \Interior(S)$
with $f'(\iota_0) = \iota_1$.  Let $f \in \Mod(\Sigma)$ be the result of extending $f'$ over $S$ by the identity.
Since $\mu$ is supported on $S$, we have $f \in \Torelli(\Sigma,\cP,\mu)$ and $f(\iota_0) = \iota_1$, as desired. 
\end{proof}

\subsection{The double boundary stabilization proof}
\label{section:doubleboundarystabproof}

We now prove Proposition \ref{proposition:doubleboundarystab}.

\begin{proof}[Proof of Proposition \ref{proposition:doubleboundarystab}]
We start by recalling the statement and introducing some notation.
Let $\mu$ be an $A$-homology marking on $(\Sigma,\cP) \in \PSurf$ that is 
supported on a symplectic subsurface.  Let $(\Sigma,\cP) \rightarrow (\Sigma',\cP')$
be a double boundary stabilization and let $\mu'$ be the stabilization
of $\mu$ to $(\Sigma',\cP')$.
Setting
\[c = \rank(A)+2 \quad \text{and} \quad d = 2\rank(A)+2,\]
we want to prove that the induced map $\HH_k(\Torelli(\Sigma,\cP,\mu)) \rightarrow \HH_k(\Torelli(\Sigma',\cP',\mu'))$
is an isomorphism if the genus of $\Sigma$ is at least $ck + d$ and a surjection
if the genus of $\Sigma$ is $ck + d-1$.
We will prove this using Theorem \ref{theorem:stabilitymachine}.  This requires fitting
$\Torelli(\Sigma,\cP,\mu) \rightarrow \Torelli(\Sigma',\cP',\mu')$ into an increasing
sequence of group $\{G_n\}$ and constructing appropriate simplicial complexes.

As notation, let $(S_g,\cP_g) = (\Sigma,\cP)$ and $\mu_g = \mu$ and $(S_{g+1},\cP_{g+1}) = (\Sigma',\cP')$ and $\mu_{g+1} = \mu'$.  In a double boundary stabilization like
$(S_g,\cP_g) \rightarrow (S_{g+1},\cP_{g+1})$, two boundary components of 
$\Sigma_0^4$ are glued to two boundary components of $S_g$ to form $S_{g+1}$.  We will
call the two boundary components of $S_g$ to which $\Sigma_0^4$ is glued the
{\em attaching components} and the two components of 
$\partial \Sigma_0^4 \cap \partial S_{g+1}$ the {\em new components}.

By assumption, $\mu_g$ is supported on a genus $h$ symplectic subsurface for some $h$, i.e., 
there exists a $\PSurf$-morphism $(T,\cP_T) \rightarrow (S_g,\cP_g)$ with 
$T \cong \Sigma_h^1$ and an $A$-homology marking $\mu_T$ on $(T,\cP_T)$ such that 
$\mu_g$ is the stabilization of $\mu_T$ to $(S_g,\cP_g)$.  Applying Corollary 
\ref{corollary:destabilizeone} to $\mu_T$, we can assume without loss of generality 
that $h \leq \rank(A)$.  We can then factor 
$(T,\cP_T) \rightarrow (S_g,\cP_g)$ into an increasing sequence of subsurfaces
\[(T,\cP_T) \rightarrow (S_h,\cP_h) \rightarrow (S_{h+1},\cP_{h+1}) \rightarrow \cdots \rightarrow (S_g,\cP_g)\]
with the following properties:
\setlength{\parskip}{0pt}
\begin{compactitem}
\item[(i)] each $S_r$ has genus $r$, and 
\item[(ii)] each $(S_r,\cP_r) \rightarrow (S_{r+1},\cP_{r+1})$ is a double boundary
stabilization, and
\item[(iii)] for $r>h$, the attaching components of $(S_r,\cP_r) \rightarrow (S_{r+1},\cP_{r+1})$ equal the new components of $(S_{r-1},\cP_{r-1}) \rightarrow (S_r,\cP_r)$.
\end{compactitem}
See here:\setlength{\parskip}{\baselineskip}\\
\centerline{\psfig{file=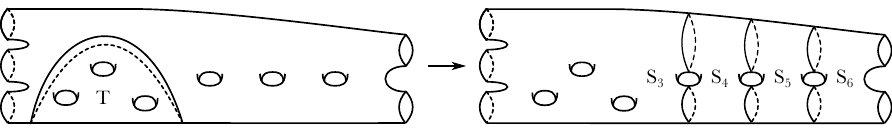,scale=100}}
This can then be continued indefinitely to form an increasing sequence of subsurfaces
\[(T,\cP_T) \rightarrow (S_h,\cP_h) \rightarrow \cdots \rightarrow (S_g,\cP_g) \rightarrow (S_{g+1},\cP_{g+1}) \rightarrow (S_{g+2},\cP_{g+2}) \rightarrow \cdots\]
satisfying (i)-(iii).  Here $(S_{g+1},\cP_{g+1})$ is as defined above.  For $r \geq h$, let
$\mu_r$ be the stabilization of $\mu_T$ to $(S_r,\mu_r)$.  This agrees with our previous
definitions of $\mu_g$ and $\mu_{g+1}$.

We thus have an increasing sequence of groups
\[\Torelli(S_h,\cP_h,\mu_h) \subset \Torelli(S_{h+1},\cP_{h+1},\mu_{h+1}) \subset \Torelli(S_{h+2}^1,\cP_{h+2},\mu_{h+2}) \subset \cdots.\]
For $r \geq h$, let $I_r, J_r \subset \partial S_r$ 
be open intervals in the two attaching components 
for $(S_r,\cP_r) \rightarrow (S_{r+1},\cP_{r+1})$.
Theorem \ref{theorem:vanishorderdoubletetheredloopcon} 
says that $\ODTL(S_r,I_r,J_r,\cP_r,\mu_r)$ is
$\frac{r-(d+1)}{c}$-connected (where $c$ and $d$ are as defined in the first paragraph).

For $n \geq 0$, let
\[G_n = \Torelli(S_{d+n},\cP_{d+n},\mu_{d+n}) \quad \text{and} \quad
X_n = \ODTL(S_{d+n},I_{d+n},J_{d+n},\cP_{d+n},\mu_{d+n}).\]
For this to make sense, we must have $d+n \geq h$, which follows from
\[d+n = 2\rank(A)+2+n \geq \rank(A) \geq h.\]
We thus have an increasing sequence of groups
\[G_0 \subset G_1 \subset G_2 \subset \cdots\]
with $G_n$ acting on $X_n$.  The indexing convention here is chosen such that $X_1$ is $0$-connected and
more generally such that $X_n$ is $\frac{n-1}{c}$-connected, as in Theorem \ref{theorem:stabilitymachine}.
Our goal is to prove that the map $\HH_k(G_{n-1}) \rightarrow \HH_k(G_{n})$ is an isomorphism
for $n \geq ck+1$ and a surjection for $n = ck$, which will follow from Theorem \ref{theorem:stabilitymachine}
once we check its conditions:
\setlength{\parskip}{0pt}
\begin{compactitem}
\item The first is that $X_n$ is $\frac{n-1}{c}$-connected, which follows
from Theorem \ref{theorem:vanishorderdoubletetheredloopcon}.
\item The second is that for $0 \leq i < n$, the group $G_{n-i-1}$ is the $G_n$-stabilizer
of some $i$-simplex of $X_n$, which follows from Lemma \ref{lemma:vanishdoubleloopstab}
via the following picture:
\end{compactitem}
\centerline{\psfig{file=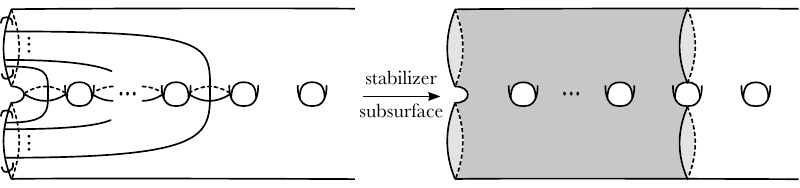,scale=100}}
\begin{compactitem}
\item The third is that for all $0 \leq i < n$, the group $G_n$ acts
transitively on the $i$-simplices of $X_n$, which follows from
Lemma \ref{lemma:doubletetheredtran}.
\item The fourth is that for all $n \geq c+1$ and all $1$-simplices $e$ of $X_n$
whose boundary consists of vertices $v$ and $v'$, there exists some $\lambda \in G_n$
such that $\lambda(v) = v'$ and such that $\lambda$ commutes with all elements of $(G_n)_e$.
Let $S'$ be the stabilizer subsurface of $e$, so by Lemma \ref{lemma:vanishdoubleloopstab} the
stabilizer $(G_n)_e$ consists of mapping classes supported on $S'$.
The surface $S_{d+n} \setminus \Interior(S')$ is diffeomorphic
to $\Sigma_1^4$ (as in the picture above), and in particular is connected.
The ``change of coordinates principle'' from \cite[\S 1.3.2]{FarbMargalitPrimer} implies that
we can find a mapping class $\lambda$ supported on
on $S_{d+n} \setminus \Interior(S')$ taking the double-tethered loop $v$ to $v'$.
Lemma \ref{lemma:vanishdoubleloopstab} implies that $\mu_{d+n}$ can be destabilized
to an $A$-homology marking on $S'$ (with respect to an appropriate partition) that
is supported on a symplectic subsurface.  This implies that
$\lambda$ lies in $G_n = \Torelli(S_{d+n},\cP_{d+n},\mu_{d+n})$ and commutes with $(G_n)_e$.\qedhere
\end{compactitem}
\setlength{\parskip}{\baselineskip}
\end{proof}

\section{Non-stability}
\label{section:closed}

This section concerns situations where homological stability does {\em not} occur.
The highlights are the proofs of Theorems \ref{maintheorem:closed} and
\ref{theorem:counterexample}.

\p{Disc-pushing subgroup}
Let $\Sigma \in \Surf$ be a surface and let $\partial$ be a component of $\partial \Sigma$.
Let $\hSigma$ be the result of gluing a disc to $\partial$.  The embedding
$\Sigma \hookrightarrow \hSigma$ induces a homomorphism
$\Mod(\Sigma) \rightarrow \Mod(\hSigma)$, which is easily seen to 
be surjective.  Its kernel, denoted $\DP(\partial)$, is the {\em disc-pushing subgroup}, 
and is isomorphic
to the fundamental group of the unit tangent bundle $U\hSigma$ of $\hSigma$;
see \cite[\S 4.2.5]{FarbMargalitPrimer}.  Elements of $\DP(\partial)$ ``push'' $\partial$ 
around paths in $\hSigma$ while allowing it to rotate.

\p{Disc-pushing and partial Torelli}
If $\partial$ is the single component of $\partial \Sigma_g^1$, then
$\DP(\partial) \subset \Mod(\Sigma_g^1)$ is contained in the Torelli group
$\Torelli(\Sigma_g^1)$, and thus is also contained in $\Torelli(\Sigma_g^1,\mu)$
for any $A$-homology marking $\mu$ on $\Sigma_g^1$.  
The following lemma generalizes this to the
partial Torelli groups on surfaces with multiple boundary components.

\begin{lemma}
\label{lemma:dptorelli}
Let $\mu$ be an $A$-homology marking on $(\Sigma,\cP) \in \PSurf$ and let
$\partial$ be a component of $\partial \Sigma$ such that $\{\partial\} \in \cP$.
Then $\DP(\partial) \subset \Torelli(\Sigma,\cP,\mu)$.
\end{lemma}
\begin{proof}
Let $f \in \DP(\partial)$ and let $x \in \HH_1^{\cP}(\Sigma,\partial \Sigma)$.  It
is enough to prove that $f(x)=x$.  Let $\hSigma$ be the result of gluing a disc 
to $\partial$ and let $\hcP = \cP \setminus \{\{\partial\}\}$.  
We thus have a $\PSurf$-morphism $\iota\colon (\Sigma,\cP) \rightarrow (\hSigma,\hcP)$.  
Since the homology classes of arcs connecting $\partial$ to other components
of $\partial \Sigma$ do not contribute to $\HH_1^{\cP}(\Sigma,\partial \Sigma)$, the map
$\iota^{\ast}\colon \HH_1^{\hcP}(\hSigma,\partial \hSigma) \rightarrow \HH_1^{\cP}(\Sigma,\partial \Sigma)$
is a surjection (in fact, it is an isomorphism, but we will not need this).  We can
thus write $x = \iota^{\ast}(\hx)$ for some $\hx \in \HH_1^{\hcP}(\hSigma,\partial \hSigma)$.
Since
\[f \in \DP(\partial) = \ker(\Mod(\Sigma) \stackrel{\iota_{\ast}}{\longrightarrow} \Mod(\hSigma)),\]
we clearly have
$\iota_{\ast}(f)(\hx) = \hx$, so Lemma \ref{lemma:pushpull} implies that
\[x = \iota^{\ast}(\hx) = \iota^{\ast}(\iota_{\ast}(f)(\hx)) = f(\iota^{\ast}(\hx)) = f(x),\]
as desired.
\end{proof}

\p{Johnson homomorphism}
Fix some $g \geq 2$ and let $H = \HH_1(\Sigma_g^1)$.  The Johnson 
homomorphism \cite{JohnsonHomo} is an important homomorphism
$\tau\colon \Torelli(\Sigma_g^1) \rightarrow \wedge^3 H$.  Letting $\partial$
be the single component of $\partial \Sigma_g^1$, it interacts with
the disc-pushing subgroup $\DP(\partial) \cong \pi_1(U\Sigma_g)$ 
in the following way.  Let $\omega \in \wedge^2 H$ be the {\em symplectic element}, i.e.,
the element corresponding to the algebraic intersection pairing under
the isomorphism
\[(\wedge^2 H)^{\ast} \cong \wedge^2 H^{\ast} \cong \wedge^2 H,\]
where we identify $H$ with its dual $H^{\ast}$ via Poincar\'{e} duality.  We then
have an injection $H \hookrightarrow \wedge^3 H$ taking $h \in H$ to $h \wedge \omega$.
The restriction of $\tau$ to $\DP(\partial)$ is the composition
\[\DP(\partial) \cong \pi_1(U\Sigma_g) \longrightarrow \pi_1(\Sigma_g) \longrightarrow H \stackrel{- \wedge \omega}{\longrightarrow} \wedge^3 H.\]

\p{Symplectic nondegeneracy}
Let $\mu$ be an $A$-homology marking on $(\Sigma,\cP) \in \PSurf$.  The {\em $\mu$-symplectic
element} $\omega_{\mu} \in \wedge^2 A$ is as follows.  Let $H$ be the quotient of 
$\HH_1(\Sigma)$ by the subgroup generated by the loops around the boundary components.  
Since $H$ is the first homology group of the closed surface obtained by gluing discs
to all components of $\partial \Sigma$, there is a symplectic element
$\omega \in \wedge^2 H$.  The closed marking $\hmu\colon \HH_1(\Sigma) \rightarrow A$
factors through a homomorphism $H \rightarrow A$, and $\omega_{\mu}$ is the image
of $\omega \in \wedge^2 H$ under the induced map $\wedge^2 H \rightarrow \wedge^2 A$.
We then have a map $A \rightarrow \wedge^3 A$ taking $a \in A$ to $a \wedge \omega_{\mu}$.
We will say that $\mu$ is {\em symplectically nondegenerate} if this map is nonzero.

\begin{example}
\label{example:nondegenerate}
Let $V$ be a symplectic subspace of $\HH_1(\Sigma_g^1)$, so 
$\HH_1(\Sigma_g^1) = V \oplus V^{\perp}$, and let $\mu\colon \HH_1(\Sigma_g^1) \rightarrow V$
be the orthogonal projection.  We claim that $\mu$ is symplectically nondegenerate
if and only if $V$ has genus at least $2$.  Indeed,
$\omega_{\mu} \in \wedge^2 V$ equals the symplectic
element arising from the symplectic form on $V$, and the map
$V \rightarrow \wedge^3 V$ taking $v \in V$ to $v \wedge \omega_{\mu}$ is nonzero
precisely when $V$ has genus at least $2$.  We remark that if $V$ has genus $0$ or $1$
then $\wedge^3 V = 0$, so the map $V \rightarrow \wedge^3 V$ is automatically the zero map.
\end{example}

\p{Partial Johnson homomorphism}
The homomorphism given by the following lemma is a version of the Johnson homomorphism
for the partial Torelli groups.

\begin{lemma}
\label{lemma:partialjohnson}
Let $\mu$ be an symplectically nondegenerate $A$-homology marking on 
$(\Sigma,\cP) \in \PSurf$ and let
$\partial$ be a component of $\partial \Sigma$ such that
$\{\partial\} \in \cP$ (and thus by Lemma \ref{lemma:dptorelli} such that
$\DP(\partial) \subset \Torelli(\Sigma,\cP,\mu)$).  Then there exists
a homomorphism $\tau\colon \Torelli(\Sigma,\cP,\mu) \rightarrow \HH_3(A)$
whose restriction to $\DP(\partial)$ is nontrivial.
\end{lemma}

\begin{remark}
The target group $\HH_3(A)$ contains $\wedge^3 A$, though sometimes it is a bit larger.
\end{remark}

\begin{proof}[Proof of Lemma \ref{lemma:partialjohnson}]
Let $\Sigma'$ be the result of gluing discs to all components of $\partial \Sigma$
except for $\partial$, let $\cP' = \{\{\partial\}\}$, and let $\mu'$ be the
stabilization of $\mu$ to $(\Sigma',\cP')$.  From their definitions, it follows that the
$\mu'$-symplectic element $\omega_{\mu'} \in \wedge^2 A$ is the same as the
$\mu$-symplectic element $\omega_{\mu} \in \wedge^2 A$, so $\mu'$ is symplectically
nondegenerate.
In \cite[Theorem 5.8]{BroaddusFarbPutman}, Broaddus--Farb--Putman construct a
homomorphism 
\[\tau'\colon \Torelli(\Sigma',\cP',\mu') \rightarrow \HH_3(A).\]
We remark that their notation is a little different from ours -- the
group $W$ in the statement of \cite[Theorem 5.8]{BroaddusFarbPutman} should
be taken to be $W = \ker(\mu')$.  Let $\DP'(\partial)$ be the disc-pushing
subgroup of $\Torelli(\Sigma',\cP',\mu')$, let $\hSigma'$ be the result of gluing
a disc to the component $\partial$ of $\partial \Sigma'$, and let
$\hmu'\colon \HH_1(\Sigma') \rightarrow A$ be the closed marking associated to $\mu'$.
One of the characteristic properties
of $\tau'$ is that its restriction to $\DP'(\partial)$ is
\[\DP'(\partial) = \pi_1(U\hSigma') \rightarrow \pi_1(\hSigma') \rightarrow \HH_1(\hSigma') = \HH_1(\Sigma') \stackrel{\hmu'}{\longrightarrow} A \stackrel{-\wedge \omega_{\mu'}}{\longrightarrow} \wedge^3 A \hookrightarrow \HH_3(A).\]
In particular, since $\mu'$ is symplectically nondegenerate 
the restriction of $\tau'$ to $\DP'(\partial)$
is nontrivial.  Let $\tau\colon \Torelli(\Sigma,\cP,\mu) \rightarrow \HH_3(A)$ be
the composition of $\tau'$ with the map 
$\Torelli(\Sigma,\cP,\mu) \rightarrow \Torelli(\Sigma',\cP',\mu')$.  The restriction of
this latter map to $\DP(\partial)$ is a surjection $\DP(\partial) \rightarrow \DP'(\partial)$,
so the restriction of $\tau$ to $\DP(\partial)$ is nontrivial, as desired.
\end{proof}

\p{Closing up surfaces and nonstability}
In light of Example \ref{example:nondegenerate} above, the following theorem
generalizes Theorem \ref{maintheorem:closed}.

\begin{theorem}
\label{theorem:noninjective}
Let $\mu$ be a symplectically nondegenerate $A$-homology marking on 
$(\Sigma,\cP)$, let $(\Sigma,\cP) \rightarrow (\Sigma',\cP')$ be a
$\PSurf$-morphism, and let $\mu'$ be the stabilization of $\mu$ to $(\Sigma',\cP')$.
Assume that there exists a component $\partial$ of $\partial \Sigma$ with
$\{\partial\} \in \cP$ whose image in $\Sigma'$ bounds a disc.  Then
the map $\HH_1(\Torelli(\Sigma,\cP,\mu)) \rightarrow \HH_1(\Torelli(\Sigma',\cP',\mu'))$
is not injective.
\end{theorem}
\begin{proof}
Lemma \ref{lemma:dptorelli} implies that $\DP(\partial) \subset \Torelli(\Sigma,\cP,\mu)$,
and Lemma \ref{lemma:partialjohnson} implies that there exists a homomorphism
from $\Torelli(\Sigma,\cP,\mu)$ to an abelian group whose restriction to
$\DP(\partial)$ is nontrivial.  Since 
\[\DP(\partial) \subset \ker(\Torelli(\Sigma,\cP,\mu) \rightarrow \Torelli(\Sigma',\cP',\mu')),\]
this implies that the induced map on abelianizations is not injective, as desired.
\end{proof}

\p{General nonstability}
We now prove Theorem \ref{theorem:counterexample}.

\begin{proof}[Proof of Theorem \ref{theorem:counterexample}]
We start by recalling what we must prove.
Let $\mu$ be a symplectically nondegenerate $A$-homology marking
on $(\Sigma,\cP) \in \PSurf$ that is supported on a symplectic subsurface.
Let $(\Sigma,\cP) \rightarrow (\Sigma',\cP')$
be a non-partition-bijective $\PSurf$-morphism and let $\mu'$ be the stabilization
of $\mu$ to $(\Sigma',\cP')$.  Assume that the genus of $\Sigma$
is at least $3\rank(A)+4$.  We must prove that the
induced map $\HH_1(\Torelli(\Sigma,\cP,\mu)) \rightarrow \HH_1(\Torelli(\Sigma',\cP',\mu'))$
is not an isomorphism.  
We will ultimately prove this by reducing it to Theorem \ref{theorem:noninjective} above.

Identify $\Sigma$ with its image in $\Sigma'$. 
We start with the following reduction.  Recall that
for a surface $S$, the discrete partition of
the components of $\partial S$ is 
$\Set{$\{\partial\}$}{$\partial$ a component of $\partial S$}$.

\begin{claim}
We can assume without loss of generality that $\cP$ and $\cP'$ are the discrete
partitions of the components of $\partial \Sigma$ and $\partial \Sigma'$ and 
that the genera of $\Sigma$ and $\Sigma'$ are the same.
\end{claim}
\begin{proof}[Proof of claim]
We do this in three steps:
\setlength{\parskip}{0pt}
\begin{compactitem}
\item First, let $(\Sigma',\cP') \rightarrow (\Sigma'',\cP'')$ be an open capping (see
\S \ref{section:reductioncap}; this implies in particular that $\cP''$
is the discrete partition of $\partial \Sigma''$) 
and let $\mu''$ be the stabilization of $\mu'$
to $(\Sigma'',\cP'')$.  Since open cappings are partition-bijective, Theorem
\ref{theorem:stableboundary} implies that the map
$\HH_1(\Torelli(\Sigma',\cP',\mu')) \rightarrow \HH_1(\Torelli(\Sigma'',\cP'',\mu''))$
is an isomorphism.  The composition 
\[(\Sigma,\cP) \rightarrow (\Sigma',\cP') \rightarrow (\Sigma'',\cP'')\]
is still not partition-bijective, so replacing $(\Sigma',\cP')$ and $\mu'$
with $(\Sigma'',\cP'')$ and $\mu''$, we can assume without loss of generality
that $\cP'$ is the discrete partition of $\partial \Sigma'$.
\item Next, just like in Case \ref{casea:2} of the proof of 
Theorem \ref{theorem:stableboundary} in \S \ref{section:reductioncap},
we can use the fact that $\mu$ is supported
on a symplectic subsurface to find a partition-bijective $\PSurf$-morphism
$(\Sigma''',\cP''') \rightarrow (\Sigma,\cP)$ and an $A$-homology marking
$\mu'''$ on $(\Sigma''',\cP''')$ such that $\mu$ is the stabilization of
$\mu'''$ to $(\Sigma,\cP)$, such that $\cP'''$ is the discrete
partition of $\partial \Sigma'''$, and such that the genera of $\Sigma'''$
and $\Sigma$ are the same.  Theorem \ref{theorem:stableboundary} implies
that the map
$\HH_1(\Torelli(\Sigma''',\cP''',\mu''')) \rightarrow \HH_1(\Torelli(\Sigma,\cP,\mu))$
is an isomorphism.  The composition
\[(\Sigma''',\cP''') \rightarrow (\Sigma,\cP) \rightarrow (\Sigma',\cP')\]
is still not partition-bijective, so replacing $(\Sigma,\cP)$ and $\mu$
with $(\Sigma''',\cP''')$ and $\mu'''$, we can assume without loss of generality
that $\cP$ is the discrete partition of $\partial \Sigma$.
\item We have now ensured that $\cP$ and $\cP'$ are the discrete partitions, and
it remains to show that we can ensure that the genera of $\Sigma$ and $\Sigma'$ are
the same.  As in the following picture, we can factor
$(\Sigma,\cP) \rightarrow (\Sigma',\cP')$ into
\[(\Sigma,\cP) \rightarrow (\Sigma^{(4)},\cP^{(4)}) \rightarrow (\Sigma',\cP')\]
where $(\Sigma,\cP) \rightarrow (\Sigma^{(4)},\cP^{(4)})$ is partition-bijective,
where $\cP^{(4)}$ is the discrete partition of $\partial \Sigma^{(4)}$, and
where the genera of $\Sigma^{(4)}$ and $\Sigma'$ are the same:
\end{compactitem}
\centerline{\psfig{file=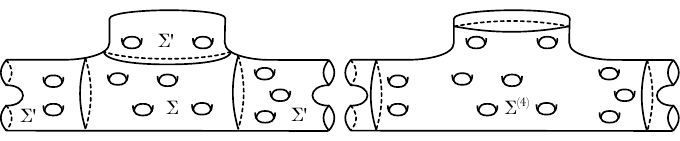,scale=100}}
\begin{compactitem}
\item[] Theorem \ref{theorem:stableboundary} implies that the map
$\HH_1(\Torelli(\Sigma,\cP)) \rightarrow \HH_1(\Torelli(\Sigma^{(4)},\cP^{(4)}))$
is an isomorphism.  Since the map $(\Sigma^{(4)},\cP^{(4)}) \rightarrow (\Sigma',\cP')$
is still not partition-bijective, we can replace $(\Sigma,\cP)$ with
$(\Sigma^{(4)},\cP^{(4)})$ and ensure that the genera of $\Sigma$ and $\Sigma'$ are the same.\qedhere
\end{compactitem}
\end{proof}

Since the genera of $\Sigma$ and $\Sigma'$ are the same, all components
of $\overline{\Sigma' \setminus \Sigma}$ are genus $0$ surfaces intersecting
$\Sigma$ in a single boundary component.  If any of these components are discs,
then Theorem \ref{theorem:noninjective} implies that the map
$\HH_1(\Torelli(\Sigma,\cP,\mu)) \rightarrow \HH_1(\Torelli(\Sigma',\cP',\mu'))$
is not injective, and we are done.  We can thus assume that no components
of $\overline{\Sigma' \setminus \Sigma}$ are discs.  Furthermore,
if any of these components are annuli, then we can 
deformation retract $\Sigma'$ over them without changing anything; doing this,
we can assume that none of them are annuli.

It follows that all the components of
$\overline{\Sigma' \setminus \Sigma}$ are genus $0$ surfaces with at least $3$
boundary components intersecting $\Sigma$ in a single boundary component.
Let $\{\partial_1,\ldots,\partial_k\}$ be a set of 
components of $\partial \Sigma'$ containing precisely one component in each
component of $\overline{\Sigma' \setminus \Sigma}$.  Let
$\Sigma''$ be the result of gluing discs to all components of $\Sigma'$
except for the $\partial_i$, let $\cP''$ be the discrete partition
of $\partial \Sigma''$ (so in particular $\{\partial_i\} \in \cP''$ for all $i$),
and let $\mu''$ be the stabilization of $\mu'$ to $(\Sigma'',\cP'')$.  All components of 
$\overline{\Sigma'' \setminus \Sigma}$ are annuli, so $\Sigma''$ 
deformation retracts to $\Sigma'$.

From this, we see that the composition
\[\Torelli(\Sigma,\cP,\mu)) \rightarrow \Torelli(\Sigma',\cP',\mu') \rightarrow \Torelli(\Sigma'',\cP'',\mu'')\]
is an isomorphism, and thus the composition
\begin{equation}
\label{eqn:bigcomp}
\HH_1(\Torelli(\Sigma,\cP,\mu))) \rightarrow \HH_1(\Torelli(\Sigma',\cP',\mu')) \rightarrow \HH_1(\Torelli(\Sigma'',\cP'',\mu''))
\end{equation}
is also an isomorphism.  Since $\cP'$ is the discrete partition and at least
one disc was glued to a component of $\partial \Sigma'$ when we formed
$\Sigma''$, Theorem \ref{theorem:noninjective} implies that the map
$\HH_1(\Torelli(\Sigma',\cP',\mu')) \rightarrow \HH_1(\Torelli(\Sigma'',\cP'',\mu''))$
is not injective.  Since the composition \eqref{eqn:bigcomp} is 
an isomorphism, we conclude that the map $\HH_1(\Torelli(\Sigma,\cP,\mu))) \rightarrow \HH_1(\Torelli(\Sigma',\cP',\mu'))$ is not surjective, and we are done.
\end{proof}

\begin{footnotesize}
\noindent
\begin{tabular*}{\linewidth}[t]{@{}p{\widthof{Department of Mathematics}+0.5in}@{}p{\linewidth - \widthof{Department of Mathematics} - 0.5in}@{}}
{\raggedright
Andrew Putman\\
Department of Mathematics\\
University of Notre Dame \\
255 Hurley Hall\\
Notre Dame, IN 46556\\
{\tt andyp@nd.edu}}
&
\end{tabular*}\hfill
\end{footnotesize}

\end{document}